%% file: main.tex
\RequirePackage{fix-cm}
\documentclass[smallextended,numbook,runningheads]{svjour3}     
\smartqed  

\usepackage{graphicx}
\usepackage{amsmath}
\usepackage{mathptmx}      
\usepackage[T1]{fontenc}
\usepackage{newtxtext,newtxmath}
\usepackage{hyperref}
\hypersetup{breaklinks,colorlinks,linkcolor=blue,citecolor=blue}
\usepackage{cleveref}
\usepackage{tikz}
\usepackage{pgfplots}\usepgfplotslibrary{fillbetween}
\pgfplotsset{compat=1.17}
\usepackage{booktabs}
\usepackage{array}
\usepackage{mathdots}

\makeatletter
\renewcommand*\env@matrix[1][\arraystretch]{%
  \edef\arraystretch{#1}%
  \hskip -\arraycolsep
  \let\@ifnextchar\new@ifnextchar
  \array{*\c@MaxMatrixCols c}}
\makeatother

\newcommand{\qedsymbol}{$\blacksquare$}
\newcommand{\morder}{{M_\text{O}}}
\newcommand{\morderth}{{M_\text{O}^{\text{th}}}}
\newcommand{\mom}{{\mathsf M}}
\newcommand{\tmom}{\widetilde{\mathsf M}}

\newcommand{\meq}{{M_\text{eqn}}}

\newcommand{\mdeg}{{M_\text{deg}}}
\newcommand{\melems}{{M_\text{elem}}}

\newcommand{\mpred}{{M_\text{P}}}
\newcommand{\mcorr}{{M_\text{C}}}
\newcommand{\half}{\frac{1}{2}}
\newcommand{\Tm}{{\mathcal T}}

\newcommand{\lampm}{{\lambda_{i\pm\frac{1}{2}}}}
\newcommand{\lmax}{{\lambda_{\max}}}
\newcommand{\lph}{{\lambda_{i+\frac{1}{2}}}}
\newcommand{\lmh}{{\lambda_{i-\frac{1}{2}}}}

\newcommand{\NF}{{\mathcal F}}

\newcommand{\bunderline}[1]{\underline{#1}}
\renewcommand{\vec}[1]{{\bunderline{#1}}}

\newcommand{\mat}[1]{{\bunderline{\bunderline{#1}}}}

\newcommand{\grad}{\vec{\nabla}}
\newcommand{\reals}{\mathbb R}
\newcommand{\nats}{{\mathbb N}_0}
\newcommand{\overbar}[1]{\mkern 1.5mu\overline{\mkern-1.5mu#1\mkern-1.5mu}\mkern 1.5mu}
\def\heat{h}
\def\rvec{{\mathcal R}}
\def\ve{\varepsilon}

\makeatletter
\renewcommand*\env@matrix[1][\arraystretch]{%
  \edef\arraystretch{#1}%
  \hskip -\arraycolsep
  \let\@ifnextchar\new@ifnextchar
  \array{*\c@MaxMatrixCols c}}
\makeatother

\journalname{}

\begin{document}

\title{Positivity-Preserving Lax-Wendroff Discontinuous Galerkin Schemes for Quadrature-Based Moment-Closure Approximations of Kinetic Models\thanks{This research was partially funded by NSF Grants DMS--1620128
and DMS--2012699.}
}

\titlerunning{LxW-DG schemes for quadrature-based moment-closures}        

\author{Erica R. Johnson         \and
        James A. Rossmanith\footnote{Corresponding author} \and
        Christine Vaughan
}

\authorrunning{E.R. Johnson, J.A. Rossmanith, and C. Vaughan} 

\institute{Erica R. Johnson \at
              Southwest Research Institute, Space Science Department,  \\
              6220 Culebra Road, San Antonio, Texas 78238, USA \\
              \email{\href{mailto:erica.johnson@swri.edu}{erica.johnson@swri.edu}}           
           \and
           James A. Rossmanith \at
              Iowa State University, Department of
	  Mathematics, \\ 411 Morrill Road, Ames, Iowa 50011, USA \\
	  \email{\href{mailto:rossmani@iastate.edu}{rossmani@iastate.edu}}
	  \and
	  Christine Vaughan \at
	  Physicians Health Plan \\ 1400 East Michigan Ave, Lansing, Michigan 48912, USA \\
	  \email{\href{mailto:cvaughan880@gmail.com}{cvaughan880@gmail.com}}
}

\date{}

\maketitle

\begin{abstract}
The quadrature-based method of moments (QMOM) offers a promising class of approximation techniques for reducing kinetic equations to fluid equations that are valid beyond thermodynamic equilibrium. In this work, we study a particular five-moment variant of QMOM known as HyQMOM and establish that this system is moment-invertible over a convex region in solution space. We then develop a high-order discontinuous Galerkin (DG) scheme for solving the resulting fluid system. The scheme is based on a predictor-corrector approach, where the prediction is a localized space-time DG scheme. The nonlinear algebraic system in this prediction is solved using a Picard iteration. The correction is a straightforward explicit update based on the time-integral of the evolution equation, where the space-time prediction replaces all instances of the exact solution. In the absence of limiters, the high-order scheme does not guarantee that solutions remain in the convex set over which HyQMOM is moment-realizable. To overcome this, we introduce novel limiters that rigorously guarantee that the computed solution does not leave the convex set of realizable solutions, thus guaranteeing the hyperbolicity of the system. We develop positivity-preserving limiters in both the prediction and correction steps and an oscillation limiter that damps unphysical oscillations near shocks. We also develop a novel extension of this scheme to include a BGK collision operator; the proposed method is shown to be asymptotic-preserving in the high-collision limit. The HyQMOM and the HyQMOM-BGK solvers are verified on several test cases, demonstrating high-order accuracy on smooth problems and shock-capturing capability on problems with shocks. The asymptotic-preserving property of the HyQMOM-BGK solver is also numerically verified.

\end{abstract}

\tableofcontents

\section{Introduction}
\label{sec:intro}
\input{introduction.tex}

\section{Brief review of moment-realizability}
\label{sec:mom_realize}
\input{momrealize.tex}

\section{HyQMOM: Hyperbolic quadrature-based moment closure}
\label{sec:hyqmom}
\input{equation_intro.tex}

\section{Locally-implicit Lax-Wendroff discontinuous Galerkin}
\label{sec:numerics}
\input{numerics.tex}

\section{HyQMOM limiters}
\label{sec:limiters}
\input{limiters.tex}

\section{Collisionless HyQMOM numerical examples}
\label{sec:examples}
\input{examples.tex}

\section{Extension to HyQMOM-BGK}
\label{sec:numerics_bgk}
\input{numerics_bgk.tex}

\section{HyQMOM-BGK numerical examples}
\label{sec:examples_bgk}
\input{examples_bgk.tex}

\section{Conclusions}
\label{sec:conclusions}
\input{conclusions.tex}

\appendix

\section{Appendix}
\label{sec:appendix}
\input{appendix.tex}

\section*{Acknowledgments}
This research was partially funded by NSF Grants DMS--1620128
and DMS--2012699.

\section*{Statements and Declarations}
\begin{description}
\item[{\bf Funding.}] This research was funded by Iowa State University in Ames, Iowa, USA and US National Science Foundation Grants DMS--1620128
and DMS--2012699.

\medskip

\item[{\bf Competing interests.}] The authors have no 
conflicts of interest to disclose.

\medskip

\item[{\bf Data availability statement.}] Data sharing does not apply to this article as no datasets were generated or analyzed during the current study.
\end{description}

\bibliographystyle{spmpsci}

\end{document}

%% file: introduction.tex

Kinetic Boltzmann equations model the non-equilibrium dynamics of a wide variety of fluids, including gases, multiphase flows, and plasma.
These equations have the following general form:
\begin{equation}
\label{eqn:boltzmann_full_3D3V}
f_{,t} +\vec{v}\cdot \grad_{\vec{x}} f +
\vec{\mathcal{F}} \cdot \grad_{\vec{v}} f  
=\mathbb{C}(f),
\end{equation}
where $f\left(t,\vec{x},\vec{v}\right): \reals_{\ge 0} \times
\reals^D \times \reals^V  \mapsto \reals_{\ge 0}$ is the distribution function that describes the state of the fluid, $t \in \reals_{\ge 0}$ is time, $\vec{x} \in \reals^D$ is the spatial coordinate, and $\vec{v}\in \reals^V$ is the velocity coordinate.
Additionally, $\vec{\mathcal{F}}\left(t,\vec{x},\vec{v}\right) 
 \in \reals_{\ge 0} \times \reals^D \times \reals^V \mapsto
 \reals^V$ is the forcing term that could include lift, drag, gravity, and other forces acting on the particles, and
 $\mathbb{C}(f): \reals_{\ge 0}\mapsto \reals$ is the collision term that describes direct particle-particle interactions.

Kinetic models of the form \eqref{eqn:boltzmann_full_3D3V}  offer two desirable
features: (1) the evolution equations have a relatively simple form (i.e., 
advection in phase space); and (2) the models are capable of accurately
describing a large class of physical phenomena that are important
in application problems. However, the main difficulty with kinetic models is that
their solutions live in high-dimensional phase space, which means that high
fidelity numerical computations are very expensive.

\subsection{Fluid models and the moment-closure problem}
One approach for reducing the computational complexity of kinetic models
is to replace them with so-called {\it fluid models}, which means that instead of
evolving the distribution function directly, one evolves a finite set of {\it moments} of the distribution function.
For example the $(\ell_1, \ell_2, \ell_3)$ moment of the distribution function, $f$, is defined as follows:
\begin{equation}
\mom_{(\ell_1, \ell_2, \ell_3)} := \int_{\reals^3} v_1^{\ell_1} v_2^{\ell_2} v_3^{\ell_3} f \left(t,\vec{x},\vec{v} \right)
\, dv_1 \, dv_2 \, dv_3,
\end{equation}
where $\vec{v} = (v_1, v_2, v_3)$. If moments of the kinetic equation \eqref{eqn:boltzmann_full_3D3V} are computed, we arrive at
three-dimensional evolution equations of the following form:
\begin{equation}
\label{eqn:mom3d}
\mom_{(\ell_1, \ell_2, \ell_3),t} + \mom_{(\ell_1+1, \ell_2, \ell_3),x_1}
+ \mom_{(\ell_1, \ell_2+1, \ell_3),x_2}+ \mom_{(\ell_1, \ell_2, \ell_3+1),x_3} = 0,
\end{equation}
where for simplicity, we have set the forcing and
collision operator to zero: $\vec{\mathcal{F}} \equiv \vec{0}$ and $\mathbb{C} \equiv 0$.

The key benefit of considering a finite set of evolution equations of the form \eqref{eqn:mom3d} over the fully kinetic
equation \eqref{eqn:boltzmann_full_3D3V} is the reduction in the number of independent variables from $1+D+V \le 7$
to $1+D \le 4$. However, we observe from \eqref{eqn:mom3d} the key challenge in fluid model approximations of the kinetic equation: to evolve the moment $\mom_{(\ell_1, \ell_2, \ell_3)}$, we need to know 
higher-order moments: $\mom_{(\ell_1+1, \ell_2, \ell_3)}$, $\mom_{(\ell_1, \ell_2+1, \ell_3)}$, and
$\mom_{(\ell_1+1, \ell_2, \ell_3+1)}$. 
This issue is known simply as the {\it moment-closure problem}. And in particular, to obtain a closed fluid system, 
one needs to somehow approximate the highest moments in the system. Furthermore, different choices lead to systems
of differential (or integro-differential) equations with vastly different mathematical properties.

\subsection{Moment-closure methods}
A standard approach for developing moment-closure approximation for \eqref{eqn:boltzmann_full_3D3V} is to assume
a specific ansatz for the distribution function:
\begin{equation}
f(t,\vec{x},\vec{v}) \approx \sum_{\ell=1}^M \psi_{\ell}\left( \vec{v}, \beta_{\ell}\left(t,\vec{x}\right) \right) 
\quad \text{or} \quad
f(t,\vec{x},\vec{v}) \approx \prod_{\ell=1}^M \psi_{\ell} \left( \vec{v}, \beta_{\ell}\left(t,\vec{x}\right) \right).
\end{equation}
The first systematic attempt at developing a moment-closure approach is due to the seminal
work of Grad \cite{article:grad49}, in which he proposed a moment-closure 
that assumed the distribution was a Maxwell-Boltzmann distribution multiplied by a polynomial in $\vec{v}$.
Since Grad's work, a vast body of literature has developed on various moment closures, including
modern modifications of Grad's closure (e.g., \cite{article:CaFaLi13,article:CaFaLi14,article:Koellermeier2021,article:Koellermeier2014,article:KoTo17}),
maximum entropy  \cite{article:Dreyer87,article:Lev96,article:MuRu93} and its numerous variants 
(e.g., \cite{article:AbdBru2016,article:BohmTorr2019}), and quadrature-based moment closures
(e.g., \cite{article:desjardins2008,article:Fox08,article:fox09,article:FoxLauVie2018,article:marchisio2005,article:PaDeFo2019}).
A full review of all methods is well beyond the scope of the current
work, but can be found in the paper of Torrilhon \cite{article:Torrilhon2009} and the references therein.

\subsection{Scope of this work}
In this work, we study a particular five-moment moment-closure known as the hyperbolic quadrature-based moment closure (HyQMOM), originally due to Fox, Laurent, and Vi\'e \cite{article:FoxLauVie2018}. Our focus here is only on the 
one-dimensional version of kinetic equation \eqref{eqn:boltzmann_full_3D3V} (i.e., 1D1V). We begin in \S\ref{sec:mom_realize} with a brief review of the moment-closure problem and a few strategies for producing fluid approximations with desirable mathematical properties.
In \S\ref{sec:hyqmom} we provide a brief review of the classical quadrature-based moment closure (QMOM), show its shortcomings, and then establish that the HyQMOM system is moment-invertible over a convex set in solution space. 
In \S\ref{sec:numerics} we introduce a novel high-order Lax-Wendroff discontinuous Galerkin scheme for solving the HyQMOM fluid system. The scheme is based on a predictor-corrector approach, where the prediction step is based on a localized space-time discontinuous Galerkin scheme. The nonlinear algebraic system that arises in this prediction step is solved using a Picard iteration. The correction is a straightforward explicit update based on the time-integral of the evolution equation, where the space-time prediction replaces all instances of the exact solution. In the absence of additional limiters, the proposed high-order scheme does not guarantee that the numerical solution remains in the convex set over which HyQMOM is moment-realizable. To overcome this challenge, we introduce in \S\ref{sec:limiters} novel limiters that rigorously guarantee that the computed solution does not leave the convex set over which moment-invertible and hyperbolicity of the fluid system is guaranteed. We develop positivity-preserving limiters in both the prediction and correction steps and an oscillation limiter that damps unphysical oscillations near shocks. In \S\ref{sec:examples} we perform convergence tests to verify the order of accuracy of the scheme and test the scheme on Riemann data to demonstrate the shock-capturing and robustness of the method. In \S\ref{sec:numerics_bgk} we develop an asymptotic-preserving \cite{article:Jin12} extension of the proposed scheme that allows us to solve a five-moment fluid model with a Bhatnagar-Gross-Krook (BGK) \cite{article:BGK54} collision operator. Finally, in \S\ref{sec:examples_bgk} we perform convergence tests to verify the order of accuracy of the scheme and verify the method on Riemann data with different Knudsen numbers. Conclusions are provided in \S\ref{sec:conclusions}.

%% file: momrealize.tex

For the remainder of the present work, we focus exclusively on the one-dimensional
version of \eqref{eqn:boltzmann_full_3D3V}. In particular, in this section, 
we focus on the transport portion of the kinetic equation, and thus
restrict ourselves to the 1D1V collisionless Boltzmann equation (aka Vlasov equation):
\begin{equation}
  \label{eqn:Boltzmann1d}
   f_{,t} + v f_{,x} = 0,
\end{equation}
where $f(t,x,v):\reals_{\ge 0} \times \reals \times \reals \mapsto \reals_{\ge 0}$ is the probability distribution function
that describes the state of the fluid. 

The moments of $f$ are defined as follows:
\begin{equation}
    \mom_{\ell} := \int_{\reals}  v^{\ell} \, f \, dv, \quad \text{for}\quad
    \ell \in \nats.
\end{equation}
A simple calculation reveals that
 for each $\ell \in \nats$, the moments satisfy the following equation:
\begin{equation}
   \mom_{\ell,t} + \mom_{\ell+1 ,x} = 0.
\end{equation}
The key difficulty is that the evolution of the $\ell^{\text{th}}$ moment
depends on the $( \ell + 1 )^{\text{st}}$ moment, meaning that 
the moment expansion does not produce a closed system.
Therefore, the key challenge for developing fluid approximations of kinetic
models is the following question: How does one close the
moment hierarchy?

\begin{definition}[Univariate moment-closure problem]
{\it
Let $S \in \nats$. Given only the first $S+1$ moments of a univariate distribution function $f(v)$: 
\begin{equation}
\mom_{\ell} = \int_{-\infty}^{\infty} v^{\ell} f(v) \, dv,
\quad \text{for} \quad \ell=0,1,\ldots,S, 
\end{equation}
find an approximation of the next moment, $\mom_{S+1}$, 
in terms of the given moments.
}
\end{definition}

The basic strategy in most moment-closure approaches is as follows:
(1) start with a finite set of moments (e.g., $\ell=0,1,\ldots,S$); 
(2) assume a form of the distribution function with several free parameters (typically, the number of free parameters is the same as the number
of moments that will be tracked); (3) determine the free parameters in the assumed
distribution function so that its moments match all of the known moments (this
part is called {\it moment-inversion}); and (4)
compute the next moment of the assumed distribution function, which
is then used to provide the flux in the evolution equation for $\mom_{S}$:
\begin{equation}
	\mom_{S,t} + \overbar{\mom}_{S+1,x} = 0,
\end{equation}
where $\overbar{\mom}_{S+1} = \overbar{\mom}_{S+1} \left( \mom_{0}, \mom_{1}, \ldots, \mom_{S} \right)$.

\subsection{Existence: Moment-realizability}
Before considering specific strategies for approximating the missing moment, $\mom_{S+1}$, it is worthwhile to discuss the general existence problem first. We begin by defining some important quantities relevant throughout this work, namely the mass density, macroscopic velocity, pressure, heat flux, and modified kurtosis:
\begin{equation}
\label{eqn:prim_vars_1}
\begin{gathered}
\rho := \int_{-\infty}^{\infty} f \, dv, \qquad
u := \frac{1}{\rho} \int_{-\infty}^{\infty} v f \, dv, \qquad
p := \int_{-\infty}^{\infty} (v-u)^2 f \, dv, \\
\heat :=  \int_{-\infty}^{\infty} (v-u)^3 f \, dv, \qquad \text{and} \qquad
k:= \int_{-\infty}^{\infty} (v-u)^4 f \, dv - \left( \frac{p^3+\rho q^2}{\rho \, p} \right),
\end{gathered}
\end{equation}
where we assume that $\rho, p > 0$. These {\it primitive variables} are directly linked to the moments $\mom_{\ell}$:
\begin{equation}
\label{eqn:prim_vars_2}
\begin{gathered}
\rho = \mom_0, \qquad u = \frac{\mom_1}{\mom_0},
\qquad p = \mom_2 - \frac{\mom_1^2}{\mom_0}, \qquad
\heat = \mom_3 - \frac{3 \mom_1 \mom_2}{\mom_0} + \frac{2 \mom_1^3}{\mom_0^2}, \\
 \text{and} \qquad
k = \frac{\mom_2^3 - 2 \mom_1 \mom_2 \mom_3 + \mom_0 \mom_3^2 + \mom_1^2 \mom_4 - \mom_0 \mom_2 \mom_4}{\mom_1^2 - \mom_0 \mom_2}.
\end{gathered}
\end{equation}
Using these we define the {\it normalized velocity variable}, $s$, and the {\it normalized moments}:
$\tmom_{j}$:
\begin{equation}
s := \frac{v-u}{\sqrt{T}} \qquad \text{and} \qquad
\tmom_j :=  \frac{\sqrt{T}}{\rho} \int_{-\infty}^{\infty} s^j \, 
f\left( v(s) \right) \, ds =  \int_{-\infty}^{\infty} s^j \, 
\widetilde f\left( s \right) \, ds,
\end{equation}
where $T=p/\rho$ is the temperature.
The moments and the normalized moments are related as follows:
\begin{align}
{\mom}_{\ell} &= \rho \, \sum_{j=0}^{\ell} \begin{pmatrix} \ell \\ j \end{pmatrix}
 T^{\frac{j}{2}} \, u^{\ell-j} \,
\tmom_{j} \qquad \text{and} \qquad
\tmom_j = {\rho}^{-1} \, T^{-\frac{j}{2}}  \sum_{\ell=0}^j \begin{pmatrix} j \\ \ell \end{pmatrix} \left(-u \right)^{j-\ell}  \, {\mom}_{\ell}.
\end{align}
By construction, the normalized moments have the following property:
\begin{equation}
\tmom_0 = 1, \qquad \tmom_1 = 0, \qquad \text{and} \qquad
\tmom_2 = 1.
\end{equation}

\begin{definition}[Realizable moments]
{\it 
The following rescaled moments:
\[
  \tmom_0 = 1, \quad \tmom_1 = 0, \quad
\tmom_2 = 1, \quad \tmom_3, \quad \ldots, \quad \tmom_{S},
\]
where $S \in {\mathbb Z}_{\ge 3}$ and $\bigl|\tmom_j\bigr| < \infty$  $\forall j \in {\mathbb Z}_{\ge 0}$,
are called {\bf realizable} if there exists a probability density function, $\widetilde{f}(s): \reals \mapsto \reals_{\ge 0}$, such that
\[
  \tmom_{j} = \int_{-\infty}^{\infty} s^j \, \widetilde{f}(s) \, ds \qquad
  \text{for} \quad j=0,1,\ldots,S.
\]
}
\end{definition}

This leads to an obvious question: for a given $S \in {\mathbb Z}_{\ge 3}$, under what conditions is the set of moments, $\{ \tmom_0=1, \, \tmom_1=0, \, \tmom_2=1, \, \tmom_3,
\ldots, \, \tmom_{S} \}$, realizable? This question is the celebrated {\bf truncated Hamburger moment problem} (e.g., see Chapter 9 of 
\cite{book:schmudgen2017}) for which we can state the following result. Note that we state only the case where $S$ is even, although a similar result also exists when $S$ is odd \cite{book:schmudgen2017}.

\begin{theorem}[Truncated Hamburger moment problem (adapted from
Theorem 9.27 of \cite{book:schmudgen2017})]
Let $S \in {\mathbb Z}_{>3}$ be an even integer. The set of moments:
\[
\{ \tmom_0=1, \, \tmom_1=0, \, \tmom_2=1, \, \tmom_3,
\ldots, \, \tmom_{S} \}
\]
is realizable if all of the Hankel determinants for $m=0,1,\ldots,S/2$
are positive:
\[
	D_{m} := \begin{vmatrix}
	1 & 0         & 1         & \tmom_{3} & \cdots & \tmom_{m}\\
	0 & 1         & \tmom_{3} & \tmom_{4} & \cdots & \tmom_{m+1}\\
    1 & \tmom_{3} & \tmom_{4} & \tmom_{5} & \cdots & \tmom_{m+2} \\
    \tmom_{3} & \tmom_{4} & \tmom_{5} & \tmom_{6} & \cdots & \tmom_{m+3}\\
    \vdots & \vdots & \vdots & \vdots &  & \vdots \\
    \tmom_m & \tmom_{m+1} &  \tmom_{m+2} & \tmom_{m+3} & \cdots & \tmom_{2m}
    \end{vmatrix} > 0.
\]
\end{theorem}

\subsection{Example: $S=4$ case}
As an example, which will become relevant later in this work, consider the case $S=4$,
where the three relevant Hankel determinants are 
\begin{equation}
D_0 = 1, \quad
D_1 = 1, \quad
D_2 = \begin{vmatrix} 1 & 0 & 1 \\
0 & 1 & \tmom_{3}  \\
1 & \tmom_{3} & \tmom_{4}
\end{vmatrix} = 
\tmom_{4} - \tmom_{3}^2 - 1 > 0.
\end{equation}
Thus, the realizability condition in the univariate $S=4$ case is
\begin{equation}
\label{eqn:realizability}
\tmom_{4} > \tmom_{3}^2 + 1,
\end{equation}
which is depicted in Figure \ref{fig:realizability}. Also shown in this figure is the location of the distribution in thermodynamic equilibrium (i.e., the Maxwellian distribution):
\begin{equation}
f(v) = \frac{\rho}{\sqrt{2\pi T}} e^{-\frac{(v-u)^2}{2T}} \quad \Longrightarrow \quad
\widetilde{f}(s) = \frac{1}{\sqrt{2\pi}} e^{-\frac{s^2}{2}} \quad \Longrightarrow \quad
\begin{cases} \tmom_{3} = 0 \\
\tmom_{4} = 3
\end{cases}.
\end{equation}
Finally, we note that the realizability condition \eqref{eqn:realizability} in terms of the primitive variables can be written as follows:
\begin{equation}
\label{eqn:realizability_prim}
\frac{\rho \, k}{p^2} > 0,
\end{equation}
which is satisfied if $\rho>0$, $p>0$, and $k>0$.

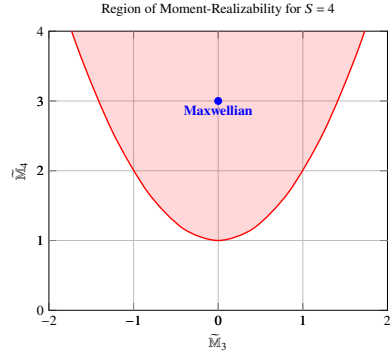
\begin{figure}
\begin{center}
\begin{tikzpicture}[scale=0.65]
  \begin{axis}[ xmin=-2, xmax=2, ymin=0, ymax=4, extra x ticks={-1,0,1}, extra y ticks={1,2,3}, extra tick style={grid=major}, xlabel=$\widetilde{\mathbb M}_3$,
  ylabel=$\widetilde{\mathbb M}_4$, title={Region of Moment-Realizability for $S=4$}]
    \addplot [name path=f, fill=gray, fill opacity=0.15, ultra thick, mark=none, smooth, thick, color=red] { 1+x^2 };
    \addplot [only marks, blue] coordinates { (0,3) } node[below] {{\bf Maxwellian}};
  \end{axis}  
  \end{tikzpicture}
\end{center}
\caption{Region of the moment-realizability when $S=4$. In other words, given the five rescaled moments: $\widetilde{\mathbb M}_0=1$, $\widetilde{\mathbb M}_1=0$, $\widetilde{\mathbb M}_2=1$, $\widetilde{\mathbb M}_3$, and $\widetilde{\mathbb M}_4$, 
there exists a positive distribution function matching the given moments provided
that these moments lie within the shaded pink region of the above graph. Thermodynamic equilibrium (i.e., the Maxwellian distribution) occurs at 
$\widetilde{\mathbb M}_3=0$ and $\widetilde{\mathbb M}_4=3$, which  is in the interior of the shaded pink region.\label{fig:realizability}}
\end{figure}

\subsection{Techniques for moment-closure}
There are many different moment closure methods for approximating the final moment, $\overbar{\mom}_{S+1}$, and each method has its own merits and challenges. A full review of all methods is well beyond the scope of the current
work, but can be found in the paper of Torrilhon \cite{article:Torrilhon2009} and the references therein. 
In this work, we settle for a brief summary of three broad classes of the most commonly used closures. 


\begin{description}
\item[{\bf Grad closure:}] This approach was originally developed
 by Grad \cite{article:grad49} in 1949, but 
 variants with improved hyperbolicity properties have been introduced more recently \cite{article:CaFaLi13,article:CaFaLi14,article:Koellermeier2021,article:Koellermeier2014,article:KoTo17}. The basic idea is that the distribution  function is approximated as a Maxwellian multiplied by a polynomial in $v$:
\begin{equation}
   \widetilde{f}(s) \sim  \frac{e^{-\frac{s^2}{2}}}{\sqrt{2\pi T}}  \left[ \rho +  
     \sum_{\ell=3}^{S} T^{-\frac{\ell}{2}} \beta_{\ell} \psi_{\ell}\left( s \right) \right],
\end{equation} 
where $\psi_{\ell}\left( s \right)$ is the Hermite polynomial of degree $\ell$ and $\vec{\beta}$ are coefficients chosen so that $\widetilde{f}(s)$ matches the first
$S+1$ input moments. A related approach widely used in applications is the R13 model, which regularizes the 13-moment Grad closure through additional terms from the Chapman-Enskog expansion; an excellent review of the R13 model can be found in Torrilhon \cite{article:Torrilhon2016}.

\medskip

\item [{\bf Maximum entropy closure:}] The maximum entropy closure \cite{article:Dreyer87,article:Lev96,article:MuRu93} and its numerous variants 
(e.g., see \cite{article:AbdBru2016,article:BohmTorr2019}) formulate the
moment-inversion problem as an optimization problem to maximize the entropy
under some assumed form of the distribution function. In the original formulation, the distribution function is approximated as an exponential of a polynomial:
\begin{equation}
   \widetilde{f}(s) \sim  e^{ \, \vec{\beta} \cdot \vec{{\Phi}}(s)} \qquad \text{where} \qquad \vec{\Phi}(s) = [1,s,s^2,\dots]^T,
\end{equation}
and the coefficients, $\vec{\beta}$, are chosen so that $\widetilde{f}(s)$ matches the first $S+1$ input moments.

\medskip

\item [{\bf Quadrature-based moment closures:}] In the quadrature-based method of moments (QMOM)  
 (e.g., \cite{article:desjardins2008,article:Fox08,article:fox09,article:FoxLauVie2018,article:marchisio2005,article:PaDeFo2019}),
  the distribution function is
represented as a sum of Dirac delta functions. This closure will be discussed in detail in \S\ref{sec:hyqmom}. 
\end{description}

%% file: equation_intro.tex

In this section, we review the quadrature-based moment-closure (QMOM) approach and describe in detail the five-moment hyperbolic regularization of QMOM (HyQMOM), which is the main focus of the current work.

\subsection{Classical QMOM Approach}
The classical quadrature-based moment-closure (QMOM) approach is widely used in modeling multiphase flows; key developments in this methodology have been
developed over the course of the past several years, e.g., see \cite{article:ChKaMa12,article:desjardins2008,book:Fox2003,article:Fox08,article:fox09,article:marchisio2005,article:Vikas11}.

The key idea is to assume that the distribution is a sum of Dirac delta functions whose locations (abscissas) and strengths (weights) are free parameters:
\begin{align}
\label{eqn:delta}
  {f}(t,x,v) \approx f^{\star}(t,x,v) := \sum_{j=1}^N \omega_j \delta \left( v-\mu_j \right),
\end{align}
where $\delta(v)$ is the Dirac delta and the quadrature 
weights, $\omega_j$, and abscissas, $\mu_j$, are all functions
of $t$ and $x$. This approach is reminiscent of other discrete velocity models such as the Broadwell model \cite{article:broad64a,article:broad64b,article:platkowski88}; however, a key difference is that the discrete velocities, $\mu_\ell$, change with the solution.

The moment inversion problem requires us to find $(\omega_j, \mu_j)$ for $j=1,\ldots,N$ by matching the
first $2N$ moments of $f(t,x,v)$:
\begin{equation}
\label{eqn:qmom_missing_moment}
\mom_{\ell} = \sum_{j=1}^N \omega_j \mu_j^{\ell} \quad \text{for} \quad \ell=0,1,\ldots,2N-1.
\end{equation}
The closure then comes from taking the next moment as follows:
\begin{equation}
\mom^{\star}_{2N} = \sum_{j=1}^N \omega_j \mu_j^{2N}.
\end{equation}

\subsubsection{Example: $N=2$ case}
\label{subsec:qmomN2}
As a simple example, let us consider the $N=2$ case.
The first four moments of $f^{\star}$ in \eqref{eqn:delta} with $N=2$ are
\begin{equation}
\label{eqn:mom_delta1}
  \mom_{\ell} = \omega_1 \mu^\ell_1 + \omega_2 \mu^\ell_2
  \quad \text{for} \quad \ell=0,1,2,3.
\end{equation}
The moment inversion problem is then this: given $(\mom_0,\mom_1, \mom_2,
\mom_3)$, find the parameters $(\mu_1, \mu_2, \omega_1, \omega_2)$
such that \eqref{eqn:mom_delta1} is satisfied.

This inversion problem is equivalent to finding the quadrature points and weights
for the following weighted Gaussian quadrature rule:
\begin{equation}
   \int_{-\infty}^{\infty} g(v) \, f(v) \, dv \, \approx \,  {\omega}_1 \, g \left( {\mu}_1 \right) 
   	+ {\omega}_2 \, g \left( {\mu}_2 \right),
\end{equation}
where $f(v)$ is a probability density function with moments $(\mom_0,\mom_1, \mom_2,
\mom_3)$.
If we attempt to make this quadrature
rule exact with $g(v) = 1$, $v$, $v^2$,  and $v^3$, we again arrive at \eqref{eqn:mom_delta1}.

To find the correct Gaussian quadrature rule, we invoke results from classical numerical analysis and look for polynomials of degree up to two that are orthogonal in the weighted $L^2(-\infty,\infty)$ inner product: $\langle \cdot , \cdot \rangle_{f}$.
Such polynomials are easily obtained via Gram-Schmidt, and indeed the relevant
one here is the quadratic polynomial:
\begin{align}
\label{eqn:qmom_psi2}
   \psi_2(v) &= p^2 - p \rho \left(v - u\right)^2 + \rho \heat  \left(v-u\right),
\end{align}
where $\rho$, $u$, $p$, and $\heat$ are defined by 
\eqref{eqn:prim_vars_1}--\eqref{eqn:prim_vars_2}.
The abscissas are the two distinct real roots of ${\psi}_2(v)$
and the weights can easily be obtained  by enforcing \eqref{eqn:mom_delta1}:
\begin{align}
  {\mu}_1, \, {\mu}_2 = u + \frac{\heat}{2p} \mp \sqrt{\frac{p}{\rho} + \left(\frac{\heat}{2p}\right)^2}, \quad {\omega}_1, \, \omega_2 = \frac{\rho}{2} \left[ 1 \pm \frac{\left(\frac{\heat}{2p}\right)}{\sqrt{\frac{p}{\rho} + \left(\frac{\heat}{2p}\right)^2}} \right].
\end{align}

\subsubsection{Weak hyperbolicity and linear degeneracy of QMOM for all $N\ge 1$}

While the above-described process can be used for any $N\ge 1$, it turns out that the resulting 
fluid equations are always only {\it weakly hyperbolic}. Furthermore, all of the waves in the Riemann
problem solution are {\it linearly degenerate}. We state these facts in the form
of \cref{theorem:weakhyp} in \cref{sec:appendix}; for completeness, we also provide a full
proof of this theorem in \cref{sec:appendix}. The theorem results are well-known in the QMOM literature
(e.g., see Chalons et al. \cite{article:ChKaMa12}), although
previously, no theorem had been presented to rigorously show both the weak hyperbolicity and linear degeneracy of all the waves
for all $N\ge 1$.

\subsection{Pressure Regularized QMOM}
To overcome the weak hyperbolicity present in classical QMOM, Chalons, Fox, and Massot \cite{article:Chalons10} proposed to replace the delta function ansatz
\eqref{eqn:delta} with a multi-Gaussian ansatz of the form:
\begin{align}
\label{eqn:multiguassian}
  {f}(t,x,v) \approx f^{\star}(t,x,v) := \frac{1}{\sqrt{2\pi \sigma}} \sum_{j=1}^N  \omega_j \exp\left[-\frac{\left( v-\mu_j \right)^2}{2\sigma}\right],
\end{align}
where the free parameters are now the quadrature weights, $\omega_k$, 
the abscissas, $\mu_k$, and the additional parameter $\sigma$.
 A similar approach using B-splines was also considered by Cheng and Rossmanith \cite{article:ChRo13}. The additional parameter $\sigma$ allows this closure to match an additional moment (i.e., a total of $2N+1$ moments can now be matched), but more importantly, it provides a pressure regularization that restores strong hyperbolicity.
Unfortunately, this closure exhibits a singularity in the limit of thermodynamic
equilibrium, since in that limit, all the quadrature points collapse to the macroscopic velocity:
\begin{equation}
\sigma \rightarrow T, \qquad
\sum_{j=1}^N \omega_j = \rho, \qquad
 \text{and} \qquad
\mu_j \rightarrow u \quad \forall \, j \in \left[ 1, N \right].
\end{equation}
This type of singularity is also evident in other closures, most notably the {\it maximum entropy} closure \cite{article:Junk98}.

\subsection{HyQMOM: Density Regularized QMOM}
As an alternative to the above-described pressure regularized QMOM, Fox et al. \cite{article:FoxLauVie2018,article:PaDeFo2019} developed a density regularized 
version, which they refer to as HyQMOM (hyperbolic quadrature-based method of moments). This approach was also studied by Johnson \cite{thesis:Johnson2017} and Wiersma \cite{thesis:Wiersma2019}.
The five-moment HyQMOM system is the subject of the current work, and we briefly review it in this section.

The idea of HyQMOM is as follows: approximate the distribution as a sum of delta functions (as in classical QMOM), but place one (or more) of these delta functions at known locations.
This converts the quadrature rule from classical Gaussian quadrature to something akin to Gauss-Radau quadrature.
The version of this idea relevant to the current work
is the case of three delta functions:
\begin{equation}
\label{eqn:HyQMOM_dist}
f\approx f^\star=\omega_1\delta(v-\mu_1)+\omega_2\delta(v-u)+\omega_3\delta(v-\mu_3),
\end{equation}
where two delta distributions are at unknown locations $\mu_1,\mu_3$ and the last delta distribution is fixed at the velocity, $u$. Each of the distributions is weighted by $\omega_1,\omega_2,\omega_3$. This results in the following
moment-inversion problem:
\begin{equation}
\label{eqn:hyqmom_inversion_1}
\begin{split}
  \omega_1+\omega_3 &= \widetilde\rho, \\
  \omega_1\mu_1+\omega_3\mu_3 &= \widetilde\rho u, \\
  \omega_1\mu_1^2+\omega_3\mu_3^2 &= \widetilde\rho u^2+p, \\
 \omega_1\mu_1^3+\omega_3\mu_3^3 &= \widetilde\rho u^3 +3pu+\heat,
\end{split}
\end{equation}
and 
\begin{equation}
\label{eqn:hyqmom_inversion_2}
\widetilde\rho:= \rho - \omega_2 = 
 \omega_1\left(\frac{\mu_1}{u} \right)^4 +
 \omega_3\left(\frac{\mu_3}{u} \right)^4
- \frac{6 \rho p^2 u^2 + 4\rho p \heat u + p^3 + \rho \heat^2 + 
\rho p k}{\rho p u^4}.
\end{equation}

System \eqref{eqn:hyqmom_inversion_1} can be solved in the same way that the $N=2$ classical QMOM system was solved, namely by constructing a quadratic polynomial:
\begin{equation}
\label{eqn:hyqmom_psi2}
\psi_2(v) = p^2 - p \widetilde\rho \left(v - u\right)^2 + \heat \widetilde\rho
   \left(v-u\right),
\end{equation}
which is just \eqref{eqn:qmom_psi2} with $\rho \rightarrow \widetilde\rho$. The roots of \eqref{eqn:hyqmom_psi2} provide $\mu_1$ and $\mu_3$,
and the corresponding weights, $\omega_1$ and $\omega_3$, can easily
be computed from \eqref{eqn:hyqmom_inversion_1} once $\mu_1$ and $\mu_3$ are known:
\begin{align}
  {\mu}_1, \, {\mu}_3 = u + \frac{\heat}{2p} \mp \sqrt{\frac{p}{\widetilde{\rho}} + \left(\frac{\heat}{2p}\right)^2},  
   \quad {\omega}_1, \, \omega_3 = \frac{\widetilde\rho}{2} \left[ 
   1 \pm \frac{\left(\frac{\heat}{2p}\right)}{\sqrt{\frac{p}{\widetilde\rho} + \left(\frac{\heat}{2p}\right)^2}} \right].
\end{align}
Finally, we can obtain $\omega_2$ and fully solve the moment inversion problem
by inserting the above expressions for $\omega_1$,
$\omega_3$, $\mu_1$, and $\mu_3$ into \eqref{eqn:hyqmom_inversion_2}:
\begin{equation}
\widetilde\rho = \rho - \omega_2 = \frac{p^3}{p\left(k+\frac{p^2}{\rho}+\frac{\heat^2}{p}\right)-\heat^2}.
\end{equation}
Putting all of these results together yields:
\begin{equation}
\label{eqn:qmom_full_eqns}
\begin{split}
\mom_5^\star 
= \, 
\rho u^5+10pu^3+10\heat u^2+5ru + \frac{2 \heat r}{p} - \frac{\heat^3}{p^2},
 \quad \text{where} \quad
r := \frac{p^2}{\rho} + \frac{\heat^2}{p} + k.
\end{split}
\end{equation}
From this, we can now assemble the full fluid approximation implied by the
5-moment HyQMOM approximation \eqref{eqn:HyQMOM_dist}.
\begin{definition}[5-moment HyQMOM fluid approximation]
The 5-moment HyQMOM approximation can be written either in conservative or
primitive form:
\begin{equation}
\label{eqn:HyQMOM_sys1}
\vec{q}_{,t}+\vec{f}\left(\vec{q}\right)_{,x}= \, \vec{0} \, 
\qquad \text{or} \qquad
\vec{\alpha}_{,t}+\mat{B}\left(\vec{\alpha}\right) \, 
\vec{\alpha}_{,x}= \, \vec{0} \, ,
\end{equation}
respectively, where
\begin{equation}
\label{eqn:HyQMOM_sys2}
\vec{q} =
\begin{bmatrix}
\rho \\ \rho u \\ \rho u^2+p \\ \rho u^3 + 3pu+\heat \\ \rho u^4+6pu^2+4\heat u + r
\end{bmatrix}, \quad
\vec{f}\left(\vec{q}\right) = \begin{bmatrix}
\rho u \\ \rho u^2+p \\ \rho u^3 + 3pu+\heat \\ \rho u^4+6pu^2+4 \heat u +r
\\ \mom^\star_5
\end{bmatrix},
\end{equation}
where $r$ and $\mom^{\star}_5$ are defined by \eqref{eqn:qmom_full_eqns},
and
\begin{equation}
\label{eqn:HyQMOM_prim}
\vec\alpha = 
\begin{bmatrix}[1.2]
\rho\\ u\\ p\\ \heat \\ k
\end{bmatrix}, \quad
\mat{B}\left(\vec{\alpha}\right) = 
\begin{bmatrix}[1.2]
 u   & \color{white}{-} &  \rho   & \color{white}{-} &  0                & \color{white}{-} &  0   & \color{white}{-} &  0 \\
 0   & \color{white}{-} &  u      & \color{white}{-} &  \frac{1}{\rho}   & \color{white}{-} &  0   & \color{white}{-} &  0 \\
 0   & \color{white}{-} &  3p     & \color{white}{-} &  u                & \color{white}{-} &  1   & \color{white}{-} &  0 \\
 -\frac{p^2}{\rho^2} & \color{white}{-} & 4\heat & \color{white}{-} & -\frac{\heat^2}{p^2}-\frac{p}{\rho} & \color{white}{-} & u + \frac{2\heat}{p} & \color{white}{-} & 1  \\
 0 & \color{white}{-} & 5k & \color{white}{-} & -\frac{2 k \heat}{p^2} & \color{white}{-} & \frac{2k}{p} & \color{white}{-} &  u
\end{bmatrix}.
\end{equation}
Furthermore, we note that that conservative flux Jacobian has the following form:
\begin{equation}
\label{eqn:cons_flux_jac}
\mat{A}\left(\vec{q}\right) := \vec{f}\left(\vec{q}\right)_{,\vec{q}} =
\begin{bmatrix}[1.2]
 0   & \color{white}{-} &  1   & \color{white}{-} &  0                & \color{white}{-} &  0   & \color{white}{-} &  0 \\
 0   & \color{white}{-} &  0      & \color{white}{-} &  1   & \color{white}{-} &  0   & \color{white}{-} &  0 \\
 0   & \color{white}{-} &  0     & \color{white}{-} & 0                & \color{white}{-} &  1   & \color{white}{-} &  0 \\
0 & \color{white}{-} & 0 & \color{white}{-} & 0 & \color{white}{-} & 0 & \color{white}{-} & 1  \\
 \frac{\partial\mom^\star_5}{\partial \mom_0} & \color{white}{-} & \frac{\partial\mom^\star_5}{\partial \mom_1} & \color{white}{-} & \frac{\partial\mom^\star_5}{\partial \mom_2} & \color{white}{-} & \frac{\partial\mom^\star_5}{\partial \mom_3} & \color{white}{-} &  \frac{\partial\mom^\star_5}{\partial \mom_4}
\end{bmatrix},
\end{equation}
where the details of the last row have been omitted for brevity.
The matrices $\mat{A}$ and $\mat{B}$ as defined in \eqref{eqn:cons_flux_jac}
and \eqref{eqn:HyQMOM_prim}, respectively, 
are {\it similar matrices}, meaning that they share
the same eigenvalues.
\end{definition}

\begin{proposition}[Hyperbolicity of HyQMOM]
The HyQMOM fluid model as expressed either in conservative or primitive form 
\eqref{eqn:HyQMOM_sys1}--\eqref{eqn:HyQMOM_prim}
 is strictly hyperbolic for all $\rho$, $p$, $k$ $>0$. 
 Note that $\rho$, $p$, $k$ $>0$ defines a convex set in the solution space: $\vec{q} \in \reals^5$.
The resulting wave structure includes one linearly degenerate wave and four nonlinear waves. 
\end{proposition}

Note that the essential contents of this theorem were already known in Fox, Laurent, and Vi\'e \cite{article:FoxLauVie2018},
but never presented as a formal proposition and proved. Therefore, we include the proof here.

\begin{proof}
The eigenvalues of the Jacobian matrix, $\mat{B}\left(\vec{\alpha}\right)$, in \eqref{eqn:HyQMOM_prim} 
can be computed explicitly:
%
\begin{equation}
\label{eqn:HyQMOM_evals1}
\begin{gathered}
\lambda_1=u+\frac{\heat}{2p}-\sqrt{a+b}, \quad
\lambda_2=u+\frac{\heat}{2p}-\sqrt{a-b}, \quad
\lambda_3=u, \\
\lambda_4=u+\frac{\heat}{2p}+\sqrt{a-b}, \quad
\lambda_5=u+\frac{\heat}{2p}+\sqrt{a+b},
\end{gathered}
\end{equation}
where
\begin{equation}
\label{eqn:HyQMOM_evals2}
a = \frac{p}{\rho} + \frac{k}{p} + \left(\frac{\heat}{2p}\right)^2 \qquad \text{and}
\qquad b = \sqrt{\frac{k^2}{p^2}+\frac{k}{\rho}}.
\end{equation}
One can show via simple calculations that for all $\rho$, $p$, $k$ $>0$:
\begin{gather}
 a>b>0, \quad  \sqrt{a+b} > \frac{|\heat|}{2p}, \quad 
\text{and} \quad \sqrt{a-b} > \frac{|\heat|}{2p}.
\end{gather}
Therefore, all five eigenvalues shown in 
\eqref{eqn:HyQMOM_evals1}--\eqref{eqn:HyQMOM_evals2} are real and distinct
for all $\rho$, $p$, $k$ $>0$, which is sufficient to show that system \eqref{eqn:HyQMOM_sys1}--\eqref{eqn:HyQMOM_prim}
is strictly hyperbolic.

The eigenvectors of the flux Jacobian \eqref{eqn:cons_flux_jac} can be written
as follows:
\begin{equation}
\vec{\rvec_\ell}=\left[1,\lambda_\ell,\lambda_\ell^2,\lambda_\ell^3,\lambda_\ell^4\right]^T
\qquad \text{for}
\qquad \ell=1,2,3,4,5. 
\end{equation}
To determine whether the corresponding waves are linearly degenerate or not, we need to compute the quantities:
\begin{equation}
\frac{\partial \lambda_\ell}{\partial \vec{q}} \cdot \vec{\rvec_\ell} =
\frac{\partial \lambda_\ell}{\partial \vec\alpha} \cdot 
\left( \frac{\partial \vec{q}}{\partial \vec{\alpha}} \right)^{-1} \cdot \vec{\rvec_\ell}, \quad
\forall \ell=1,2,3,4,5,
\end{equation}
where
\begin{equation}
\frac{\partial \lambda_\ell}{\partial \vec{q}} := 
\left[ 
\frac{\partial \lambda_\ell}{\partial q_1}, \, 
\frac{\partial \lambda_\ell}{\partial q_2}, \, 
\frac{\partial \lambda_\ell}{\partial q_3}, \, 
\frac{\partial \lambda_\ell}{\partial q_4}, \, 
\frac{\partial \lambda_\ell}{\partial q_5} 
\right],
\end{equation}
and
\begin{equation}
\frac{\partial  \vec{q}}{\partial \vec{\alpha}} =
\begin{bmatrix}[1.5]
 1 && 0 &&  0 &&  0 &&  0 \\
 u &&  \rho &&  0 &&  0 &&  0 \\
 u^2  &&  2 \rho u  &&  1 &&  0 &&  0 \\
 u^3 &&  3 \left(p + \rho u^2\right) &&  3 u &&  1 &&  0 \\
 \frac{\rho^2 u^4 - p^2}{\rho^2}  &&  
  4 \left(\heat + 3 p u + \rho u^3 \right) &&   \frac{6 \rho p^2 u^2 + 2 p^3 - \rho \heat^2}{\rho p^2} &&    \frac{2 \heat + 4 p u}{p}  &&  1
\end{bmatrix}.
%
\end{equation}
%
%
We note that one of the waves is linearly degenerate while the remaining
are nonlinear:
\begin{equation}
\frac{\partial \lambda_3}{\partial \vec{q}} \cdot \vec{\rvec_3}
 \equiv 0 \qquad \text{and} \qquad
 \frac{\partial \lambda_\ell}{\partial \vec{q}} \cdot \vec{\rvec_\ell} \not\equiv 0 \quad \text{for} \quad \ell=1,2,4,5.
\end{equation}
\qedsymbol
\end{proof}


%% file: numerics.tex
We consider generic one-dimensional conservation laws of the form:
\begin{equation}
\label{eqn:cons_law}
\vec{q}_{,t} + \vec{f}\left(\vec{q} \right)_{,x} = \vec{0},
\end{equation}
where $t$ is time, $x$ is space, $q(t,x): \in \reals^+ \times \reals \mapsto \reals^\meq$ is the vector of conserved variables, $\meq$ is the number
of equations, and
$\vec{f}\left(\vec{q} \right): \reals^\meq \mapsto \reals^\meq$ is the
flux function. We assume that this system is hyperbolic, meaning that the
flux Jacobian:
\begin{equation}
\vec{f}\left(\vec{q} \right)_{,\vec{q}}: \reals^\meq \mapsto \reals^{\meq \times \meq},
\end{equation}
has real eigenvalues and a complete set of eigenvectors over some convex 
region ${\mathcal D} \subset \reals^{\meq}$ in solution space inside of which we are interested in solving the equation.

The Lax-Wendroff method \cite{article:LxW60} is a time
discretization for
hyperbolic conservation laws  based on
the the Cauchy-Kovalevskaya 
\cite{article:Kovaleskaya1875} procedure to convert temporal derivatives into spatial derivatives.
We begin with a Taylor series in time:
\begin{equation}
 \vec{q}(t+\Delta t, x) = \vec{q}(t,x) + \Delta t \vec{q}_{,t}(t,x) + \frac{1}{2} \Delta t^2 \vec{q}_{,t,t}(t,x) +
 \ldots,
\end{equation}
and then replace all time derivatives by spatial derivatives:
\begin{equation}
\label{eqn:LxW_time_derivs}
 \vec{q}_{,t} = - \vec{f}\left(\vec{q}\right)_{,x}, \quad
 \vec{q}_{,t,t} = - \vec{f}\left(\vec{q}\right)_{,t,x} = -\left[ \mat{f'}\left(\vec{q}\right) \, \vec{q}_{,t} \right]_{,x} = 
 \left[ \mat{f'}\left(\vec{q}\right) \vec{f}\left(\vec{q}\right)_{,x} \right]_{,x}, \quad \ldots,
\end{equation}
which results in the following:
\begin{equation}
\vec{q}(t+\Delta t ,x) = \vec{q} - \Delta t \vec{f}\left(\vec{q}\right)_{,x} + \frac{1}{2} \Delta t^2 \left[ \mat{f'}\left(\vec{q}\right) \vec{f}\left(\vec{q}\right)_{,x} \right]_{,x}
+ \ldots,
\end{equation}
where on the right-hand side, we have suppressed the evaluation at $(t,x)$.
The final step is to truncate the Taylor series at some finite number of terms,
and then replace all spatial derivatives by some discrete spatial derivative operators.
The above Lax-Wendroff formalism \cite{article:LxW60} has been used in conjunction with
with a variety of spatial discretizations, including finite volume  \cite{article:Le97},
weighted essentially non-oscillatory (WENO) \cite{article:TitaToro02}, and discontinuous Galerkin \cite{article:Qiu05} operators.

In this work, we are concerned with the discontinuous Galerkin version 
of Lax-Wendroff \cite{article:Qiu05}; and in particular, we make use
of the reformulation of Gassner et al. \cite{article:GasDumHinMun2011} of the Lax-Wendroff discontinuous Galerkin (LxW-DG) scheme
in terms of a locally-implicit prediction step, followed by an explicit correction step.
The key advantage of this formulation is that we do not need to explicitly compute the
partial derivatives as shown in \eqref{eqn:LxW_time_derivs}; and instead, the locally-implicit
solver automatically produces discrete versions of these derivatives.
The next challenge is to efficiently solve the nonlinear algebraic equations arising from the 
locally-implicit prediction step; we solve these equations by following Gassner et al. \cite{article:GasDumHinMun2011} and making use of a Picard fixed point iteration. We will follow the notational conventions of Guthrey and Rossmanith \cite{article:GuRo2017} developed for locally-implicit and regionally-implicit LxW-DG schemes.
Note that the predictor-correct method is equivalent to the Lax-Wendroff DG method for linear constant-coefficient hyperbolic systems. For nonlinear systems, the predictor-correct method differs slightly from Lax-Wendroff DG in that the Picard iteration inside the prediction step approximates the direct computation of the nonlinear Taylor series expansion.

\subsection{Discontinuous Galerkin finite elements}
To discretize equation \eqref{eqn:cons_law} in space, we use the discontinuous Galerkin (DG) finite element method, which was first introduced by Reed and Hill \cite{article:ReedHill73}. It was fully developed for time-dependent hyperbolic conservation laws in a series of papers by Bernardo Cockburn, Chi-Wang Shu, and collaborators (see \cite{article:CoShu98} and references therein for details). 

We  define the broken finite element space 
\begin{equation}
\mathcal{W}^{\Delta x}:=\left\{w^{\Delta x}\in[L^\infty(\Omega)]^{M_\text{eqn}}:w^{\Delta x}\vert_{\mathcal{T}_i}\in[\mathbb{P}(M_\text{deg})]^{M_\text{eqn}} \quad \forall \mathcal{T}_i\right\},
\end{equation}
where $\Delta x=(x_\text{high}-x_\text{low})/M_\text{elem}$ is a uniform grid spacing with $M_\text{elem}$ being the number of elements. Additionally, $M_\text{eqn}$ is the number of conserved variables, 
$\mathbb{P}(M_\text{deg})$ is the set of all polynomials of degree at most $M_\text{deg}$, and the computational mesh is described by non-overlapping elements of width $\Delta x$ centered at the points $x_i$:
\begin{equation}
\mathcal{T}_i=\left[x_i-\frac{\Delta x}{2}, \, x_i+\frac{\Delta x}{2}\right]  \hspace{2em} \text{for } i=1,\dots,M_{\text{elem}}.
\end{equation}
On each element, we define the local spatial variable, $\xi$:
\begin{equation}
x=x_i+\left(\frac{\Delta x}{2}\right)\xi \quad \text{for} \quad \xi\in[-1,1].
\end{equation}

On each element, we approximate the solution by a finite expansion
in terms of the following orthonormal Legendre polynomial basis functions:
\begin{equation}
\label{eqn:legendre_basis}
\vec{\Phi} = \left( 1, \, \sqrt{3}\xi,\frac{\sqrt{5}}{2}\left(3\xi^2-1\right), \, \frac{\sqrt{7}}{2}\left(5\xi^3-3\xi\right), \, \frac{\sqrt{9}}{8}\left(35\xi^4-30\xi^2+3\right), \,\dots\right),
\end{equation}
with the orthonormality property:
\begin{equation}
 \frac{1}{2} \int_{-1}^{1} \vec{\Phi} \, \vec{\Phi}^T \, d\xi = \mat{ \, \, {\mathbb I} \, \, },
\end{equation}
where $\mat{ \, \, {\mathbb I} \, \, }$ is the identity matrix.
On each element at time $t=t^n$, we approximate the solution as follows:
\begin{equation}
\label{eqn:corr_ansatz}
\vec{q}^h\left(t^n,x_i+\frac{\Delta x}{2}\xi\right):=\vec{\Phi}(\xi)^T\mat{Q_i^n} \quad \text{for} \quad \xi\in[-1,1],
\end{equation}
where
\begin{align}
\vec{\Phi}(\xi):[-1,1]\mapsto \reals^{\mcorr} \quad \text{and} \quad \mat{Q_i^n}\in\reals^{\mcorr\times\meq}.
\end{align}
The number of basis functions in 1D is $\mcorr=\mdeg+1$, and the order of accuracy is $\morder = \mdeg + 1$.

\subsection{Prediction step}
\label{subsec:prediction}
The prediction step is completely local to each element, and therefore, the prediction step is inconsistent with the underlying conservation law.
This inconsistency allows us to freely choose updated variables; we use the primitive variables for this step: $\vec{\alpha}=(\rho, u, p, \heat, k)$. 

The prediction step is local on each space-element $\left[t^n,t^{n+1}\right] \times \Tm_i$, where $t^{n+1}=t^n+\Delta t$. Let $t=t^n+\frac{\Delta t}{2}(1+\tau)$, for $\tau\in[-1,1]$ and
\begin{equation}
\label{eqn:prediction_prim}
\vec{\alpha}_{,\tau}=\vec{\Theta}(\vec{\alpha}):=-\frac{\Delta t}{\Delta x} \, \mat{B}\left(\vec{\alpha}\right)\vec{\alpha}_{,\xi},
\end{equation}
where $\mat{B}(\vec{\alpha})$ is the (primitive variable) flux Jacobian matrix
defined by \eqref{eqn:HyQMOM_prim}. We introduce a space-time Legendre basis on each element:
\begin{equation}
\Psi_\ell(\tau,\xi)=\Phi_{\ell_1}(\tau)\Phi_{\ell_2}(\xi), \quad
\text{for} \quad
\ell_1 = 1,\ldots,\morder, \quad \ell_2 = 1,\ldots,\morder+1-\ell_1,
\end{equation}
where 
\begin{equation}
\ell = \morder \left( \ell_1 - 1 \right) - \frac{\left(\ell_1-1\right)\left(\ell_1-2 \right)}{2} + \ell_2 \quad \text{such that} \quad \ell = 1,\ldots,\mpred,
\end{equation}
and $\mpred = \morder(\morder+1)/2$ is the number of space-time basis functions. These space-time basis functions are 
orthonormal on $[-1,1]^2$:
\begin{equation}
\frac{1}{4} \iint_{-1}^{1} \vec{\Psi} \, \vec{\Psi}^T \, d\tau \, d\xi = 
  \mat{ \, \, \mathbb{I} \, \, }.
\end{equation}
 We write the predicted solution as follows:
\begin{equation}
\label{eqn:pred_ansatz}
\vec{\alpha}^{\text{ST}}\left(t^{n}+\frac{\Delta t}{2} (1+\tau), \, x_i +\frac{\Delta x}{2} \, \xi \right)   := \vec{\Psi}\left(\tau,\xi\right)^T \mat{W^{n+\half}_i},
\qquad \mat{W_i^{n+\half}} \in \reals^{\mpred \times \meq},
\end{equation}
for $(\tau,\xi) \in [-1,1]^2$, where $\mat{W}$ represents the matrix of unknown coefficients. 

We proceed by multiplying \eqref{eqn:prediction_prim} by the test function $\vec{\Psi}$, then integrating over $(\tau,\xi) \in [-1,1]^2$:
	\begin{equation}
\frac{1}{4} \iint_{-1}^{1} {\alpha}_{m,\tau} \, \vec{\Psi} \, d\tau \, d\xi =
\frac{1}{4} \, \iint_{-1}^{1} {\Theta}_m(\vec{\alpha}) \,
\vec{\Psi} \, d\tau \, d\xi,
\end{equation}
 and applying integration-by-parts in $\tau$ only:
 \begin{equation}
 \begin{gathered}
 	\frac{1}{4}\int_{-1}^{1} {\alpha}^{\star}_{m}\left(\tau=1,\xi\right) \, \vec{\Psi}\left(\tau=1,\xi\right) \,  d\xi -
	\frac{1}{4}\int_{-1}^{1} {\alpha}^{\star}_{m}\left(\tau=-1,\xi\right) \, \vec{\Psi}\left(\tau=-1,\xi\right) \,  d\xi \\
-\frac{1}{4}\iint_{-1}^{1} {\alpha}_{m} \, \vec{\Psi}_{,\tau} \, d\tau \, d\xi =
\frac{1}{4} \, \iint_{-1}^{1} {\Theta}_m(\vec{\alpha}) \,
\vec{\Psi} \, d\tau \, d\xi,
\end{gathered}
\end{equation}
where the choice of values for ${\alpha}^{\star}_{m}$ at $\tau=1$ and $\tau=-1$ still remains to be made.
Before making this choice, however, let us reverse integrate-by-parts, such that the newly introduced boundary
terms are always the internal values of the current space-time element:
	\[
	\begin{gathered}
	\frac{1}{4}\int_{-1}^{1} \Bigl[ {\alpha}^{\star}_{m}\left(\tau=1,\xi\right) - 
	{\alpha}_{m}\left(\tau=1,\xi\right) \Bigr] \, \vec{\Psi}\left(\tau=1,\xi\right) \,  d\xi \\
	-
	\frac{1}{4}\int_{-1}^{1} \Bigl[ {\alpha}^{\star}_{m}\left(\tau=-1,\xi\right) - {\alpha}_{m}\left(\tau=-1,\xi\right)\Bigr] \, \vec{\Psi}\left(\tau=-1,\xi\right) \,  d\xi \\
+\frac{1}{4}\iint_{-1}^{1}  \vec{\Psi} \, {\alpha}_{m,\tau} \, d\tau \, d\xi =
\frac{1}{4} \, \iint_{-1}^{1} {\Theta}_m(\vec{\alpha}) \,
\vec{\Psi} \, d\tau \, d\xi.
\end{gathered}
\]
Now we plug-in the following values for ${\alpha}^{\star}_{m}$ at $\tau=1$ and $\tau=-1$:
\begin{equation}
{\alpha}^{\star}_{m}\left(\tau=1,\xi\right) = {\alpha}_{m}\left(\tau=1,\xi\right)
\quad \text{and} \quad
{\alpha}^{\star}_{m}\left(\tau=-1,\xi\right) = \vec{\Phi} \left(\xi \right)^T \, \vec{A^n_{i \, (:,m)}},
\end{equation}
as well as the following values for ${\alpha}_{m}$ on the interior of the space-time element:
\begin{equation}
{\alpha}_{m}\left(\tau,\xi\right) = \vec{\Psi}\left(\tau,\xi\right)^T \vec{W^{n+\half}_{i(:,m)}}.
\end{equation}
 The result is
\begin{equation}
\label{eqn:pred_step}
\begin{split}
\mat{L} \, \, \vec{W^{n+\half}_{i \, (:,m)}} &= \frac{1}{4}\iint_{-1}^{1} 
  {\Theta}_m\left( \vec{\Psi}^T \mat{W^{n+\half}_i} \right) \vec{\Psi} \, d\tau \, d\xi \\ &+ 
  \left[ \frac{1}{4} \int_{-1}^{1} \, \vec{{\Psi}} \left(-1,\xi \right)  \, \vec{\Phi} \left(\xi \right)^T \, d\xi
 \right] \, \vec{A^n_{i \, (:,m)}},
 \end{split}
\end{equation}
for each equation $m=1,\ldots,\meq$,
 where
\begin{align}
\mat{L} &= \frac{1}{4} \iint_{-1}^{1} \vec{\Psi} \, \vec{\Psi}^T_{,\tau} \, d\tau \, d\xi
+ \frac{1}{4} \int_{-1}^{1} \vec{\Psi}_{|_{\tau=-1}} \vec{\Psi}_{|_{\tau=-1}}^T \, d\xi, \\
\mat{A_i^n} &= \frac{1}{2} \sum_{a=1}^{\morder} \omega_a \, \vec{\Phi}\left(\mu_a \right)  \left[ \vec{\alpha}\left( 
\vec{\Phi}\left(\mu_a \right)^T \mat{Q_i^n} \right) \right]^T ,
\end{align}
and $\vec{\alpha}(\vec{q})$ gives the relationship between conservative
and primitive variables.
Equation \eqref{eqn:pred_step} is a nonlinear algebraic equation
that must be solved on each space-time element for the matrix of unknown coefficients:
 $ \mat{W}$.

Following Gassner et al. \cite{article:GasDumHinMun2011}, instead of using Newton's method to solve the resulting non-linear equation, which involves inverting a Jacobian matrix at every step, we used the much simpler Picard iteration. There are two key advantages of the Picard iteration over Newton's method. First, since $\mat{L}$ is independent of the solution and the same on each space-time element, we only invert this relatively small matrix once. Second, the Picard iteration converges to sufficiently high order accuracy after $\morder-1$ iterations, so the need to compute residuals is eliminated (see
justification in \cite{article:GasDumHinMun2011}). We can write the Picard iteration as
\begin{equation}
\label{eqn:picard}
\begin{split}
\vec{W^{n+\half (j)}_{i \, (:,m)}} = & \frac{1}{4} \sum_{a=1}^{\morder}
\sum_{b=1}^{\morder} \, \omega_a  \, \omega_b \, 
  \vec{\hat{\Psi}}\left(\mu_b, \, \mu_a \right)
   \, {\Theta}_m\left( \vec{\Psi}\left(\mu_b, \mu_a \right)^T \mat{W^{n+\half (j-1)}_i} \right) \\
   + & \frac{1}{4} \sum_{b=1}^{\morder} \omega_{b} \, \vec{\hat{\Psi}} \left(-1,\xi_b \right)  \, \vec{\Phi} \left(\xi_b \right)^T \vec{A^{n}_{i \, (:,m)}},
\end{split}
\end{equation}
where $j=1,\ldots,\morder-1$ is the iteration count,
 $m=1,\ldots,\meq$ is the equation index, $\vec{\hat{\Psi}} = \mat{L}^{-1} \vec{\Psi}$, and
 $\omega_a$ and $\mu_a$ for $a=1,\ldots,\morder$ are the weights and abscissas of the $\morder$-point Gauss-Legendre quadrature rule.
This gives a solution for the prediction step, which we know is inconsistent with the conservation law. To make the final solution consistent (and high-order) with the original conservation law, we next need to add a correction step.

\subsection{Correction step}
\label{subsec:correction}
The correction step is designed to work like a single forward Euler-like step that uses the predicted solution. To perform this step, we begin with the hyperbolic conservation law \eqref{eqn:cons_law} and multiply by the spatial basis functions 
defined in \eqref{eqn:legendre_basis}.
Next, we integrate over the space-time element $\left(\tau, \xi \right) \in 
\left[-1,1\right]^2$:
\begin{equation}
        \frac{1}{2}\iint_{-1}^{1} 
        \left[ \vec{\Phi}(\xi) \, \vec{q}_{,\tau} + \frac{\Delta t}{\Delta x} \, \vec{\Phi}(\xi) \,\vec{f}\left(\vec{q}\right)_{,\xi} \right] \,  d\xi \, d\tau = \vec{0},
\end{equation}
which can be written as 
\begin{equation}
\label{eqn:corr_step}
\begin{split} 
        \frac{1}{2}\int_{-1}^{1}  
         \vec{\Phi}(\xi) \, \vec{q}\left(t^{n+1},x_i + 
         \frac{\Delta x}{2} \xi \right) \, \, d\xi &=\frac{1}{2}\int_{-1}^{1}  
         \vec{\Phi}(\xi) \, \vec{q}\left(t^{n},x_i + 
         \frac{\Delta x}{2} \xi \right) \, \, d\xi \\
        &-\frac{\Delta t}{2 \Delta x}\iint_{-1}^{1}  
        \vec{\Phi}(\xi) \,\vec{f}\left(\vec{q}\right)_{,\xi}  \, d\xi \, d\tau.
        \end{split}
    \end{equation}
We approximate $\vec{q}\left(t^{n+1}, \cdot \right)$ and
$\vec{q}\left(t^{n}, \cdot \right)$ in \eqref{eqn:corr_step}
via appropriate versions of ansatz \eqref{eqn:corr_ansatz}.
For the remaining term, we first apply integration-by-parts in space, then
replace the true solution $q$ by its space-time predicted
solution: \eqref{eqn:pred_ansatz}, and replace exact integration by numerical quadrature.
This results in the following expression:
 \begin{equation}
 \label{eqn:corr_step_update_1}
 \begin{split}
        \mat{Q_i^{n+1}} &= 
 \mat{Q^{n}_i} +  \frac{\Delta t}{2 \Delta x} \sum_{a=1}^{\morder}
\sum_{b=1}^{\morder} \, \omega_a  \, \omega_b \, \vec{\Phi}_{,\xi}\left(\mu_a \right) \left[
 \vec{f} \left( \vec{\Psi}\left(\mu_b,\mu_a\right)^T \mat{W^{n+\half}_{i}} \right) \right]^T \\
&- {\frac{\Delta t}{\Delta x}}\left( \, \vec{\Phi}(1) \left[ \vec{{\mathcal F}^{n+\half}_{i+\half}} \right]^T
-  \vec{\Phi}(-1) \left[ \vec{{\mathcal F}^{n+\half}_{i-\half}} \right]^T \, \right).
\end{split}
    \end{equation}
     In the expressions after the approximation symbol, $\approx$, we
       replaced all exact integration by a Gauss-Legendre quadrature, 
where $\omega_a$ and $\mu_a$ for $a=1,\ldots,\morder$ are the weights and abscissas of the $\morder$-point Gauss-Legendre quadrature rule.

The time-integrated numerical fluxes are defined using the predicted solution and the Rusanov \cite{article:Ru61} time-averaged flux: 
\begin{gather}
 \label{eqn:corr_step_update_2}
  \vec{{\mathcal F}^{n+\half}_{i-\half}} := \frac{1}{2} \sum_{a=1}^{\morder}
  \omega_a \, \vec{\mathcal F} \left( \mu_a \right),
\end{gather}
where the numerical flux at each temporal quadrature point is given by
\begin{equation}
 \label{eqn:corr_step_update_3}
   \vec{\mathcal F} \left( \tau \right) := \frac{1}{2} \left\{ \vec{f}\left( \vec{W_{\text{R}}}(\tau) \right) + 
  \vec{f} \left( \vec{W_{\text{L}}}(\tau) \right) \right\}   -  
       \frac{1}{2} \lmax(\tau) \left\{ \vec{q} \left( \vec{W_{\text{R}}}(\tau) \right) 
         - \vec{q} \left( \vec{W_{\text{L}}}(\tau) \right) \right\},
\end{equation}
where
\begin{gather}
 \label{eqn:corr_step_update_4}
  \vec{W_{\text{L}}}(\tau) := \vec{\Psi}\left(\tau,1\right)^T \mat{W^{n+\half}_{i-1}}, \qquad
  \vec{W_{\text{R}}}(\tau) := \vec{\Psi}\left(\tau,-1\right)^T \mat{W^{n+\half}_{i}},
\end{gather}
and $\lmax(\tau)$ is a local bound on the spectral radius of 
the flux Jacobian, $\vec{f}(\vec{q})_{,\vec{q}} = \mat{A}\left(\vec{q} \right)$, in the neighborhood of interface $x=x_{i-\half}$ and at time $t = t^{n} + (\tau+1)\Delta t/2$.

\begin{table}[!t]
\begin{center}
\begin{tabular}{|r||c||c||c||c|}
\hline
\text{order} & $\morder=1$ & $\morder=2$ & $\morder=3$ & $\morder=4$ \\
\hline\hline
\text{CFL \# used in practice} & $0.90$ & $0.30$ & $0.14$ & $0.09$ \\
\hline
\end{tabular}
\caption{The CFL numbers used in all simulations as a function of the method order.}
\label{table:CFL_numbers}
\end{center}
\end{table}

These steps are all it takes to regain the coupling neglected in the prediction step. We now have a solution that is not only consistent with the conservation law but is also high order. 
The Courant-Friedrichs-Lewy (CFL) number we use for each order is given in \cref{table:CFL_numbers}.
However, we still have some work to do to ensure that the solution is physical. We must be careful to maintain the positivity of the primitive variables $\rho$, $p$, $k$, as was necessary for the system's hyperbolicity and moment-realizability. We address the limiters utilized to accomplish this in the next section.

%% file: limiters.tex

The high-order numerical method as described in \S\ref{sec:numerics} does not guarantee that density, pressure, and modified kurtosis remain positive throughout a time-step:
\begin{equation}
\rho^n > 0, \quad p^n > 0, \quad k^n > 0 \quad \centernot\Longrightarrow \quad
\rho^{n+1} > 0, \quad p^{n+1} > 0, \quad k^{n+1} > 0.
\end{equation}
Recall that positivity of these quantities is needed to guarantee the moment-realizability of the moment-closure and strict hyperbolicity of the resulting evolution equations.
 If we want positivity over a time step, we will need to introduce {\it positivity-preserving limiters}. Additionally, if we want to control unphysical oscillations near large gradients, shocks, and rarefactions, we will also need {\it non-oscillatory limiters}. In this section, we derive all of these limiters. In particular, we first need to establish that a simple first-order scheme is positivity-preserving under some appropriate time-step restriction; this is done in \S\ref{subsec:positivity_LLF}. Using this result we derive a suite of limiters that ensure positivity: \S\ref{subsec:prediction_positivity_at_points} (positivity at select points in the prediction step), \S\ref{subsec:correction_positivity_average} (positivity of the average density, pressure, and modified kurtosis in each element in the correction step), and \S\ref{subsec:correction_positivity_at_points} (positivity at select points in the correction step). We then develop an unphysical oscillation limiter in \S\ref{subsec:oscillation_limiter}.

\subsection{Positivity of the Rusanov scheme}
\label{subsec:positivity_LLF}
Before considering positivity limiters for the high-order method, we must
first establish that simple first-order schemes, in this case, we consider the Rusanov (aka local Lax-Friedrichs) scheme \cite{article:Ru61}, are positivity-preserving under some appropriate time-step restriction. This is established in the theorem below, which is an extension of the result of Zhang and Shu \cite{article:ZhangShu2011} for the compressible Euler equations.

\begin{theorem}
\label{llf_pos}
Let $\vec{Q_i^n}$ be some
approximation of the element averages of the conserved variables 
on element $\Tm_i$ at time $t^n$,
and let $\vec{Q_i^{n+1}}$ be the element averages produced by the Rusanov scheme \cite{article:Ru61} 
at time $t^{n+1} = t^n + \Delta t$. Then 
 \begin{equation}
 \label{eqn:LLF_pos_result}
 \quad \rho_i^n > 0,  \, \, \, \, p_i^n > 0, \, \, \, \, k_i^n > 0 \, \, \, \, \, \forall i \quad \Longrightarrow \quad
\rho^{n+1}_i > 0, \, \, \, \, p^{n+1}_i > 0, \, \, \, \, k^{n+1}_i > 0 \, \, \, \, \, \forall i,
\end{equation}
 under the CFL condition:
 \begin{equation}
 \label{LLF_CFL_condition}
 \frac{\Delta t}{\Delta x}  \max_i\left(\lph\right)<1,
 \end{equation}
where 
\begin{equation}
\label{eqn:lambda_plusminus}
\lampm=\max\left\{\lmax\left(\vec{Q_i^n}\right), \, \lmax\left(\vec{Q^n_{i\pm1}}\right), \, \lmax\left(\frac{1}{2}\left(\vec{Q_i^n}+\vec{Q^n_{i\pm1}}\right)\right)\right\},
\end{equation}
and $\lmax\left(\vec{Q_i^n}\right)$ is a bound on the spectral radius of 
the flux Jacobian \eqref{eqn:cons_flux_jac} at state $\vec{Q_i^n}$.

\end{theorem}

\begin{proof}
Recall that the Rusanov scheme can be written as 
\begin{equation}
\label{eqn:update_LLF}
\vec{Q_i^{n+1}}=\vec{Q_i^n}-\frac{\Delta t}{\Delta x}\left(\vec{\NF_{i+\frac{1}{2}}}-\vec{\NF_{i-\frac{1}{2}}}\right),
\end{equation}
where the numerical fluxes are given by 
\begin{equation}
\label{eqn:LLF_flux}
\vec{\NF_{i\pm\frac{1}{2}}} = \frac{1}{2}\left[\vec{f}(\vec{Q_i^n})+\vec{f}(\vec{Q_{i\pm1}^n})\right]-\frac{1}{2} \lampm \left(\vec{Q_i^n}-\vec{Q^n_{i\pm1}}\right),
\end{equation}
where the flux function, $\vec{f}\left(\vec{q}\right)$, is defined by \eqref{eqn:HyQMOM_sys2}
and the local wave speed, $\lampm$, is defined by \eqref{eqn:lambda_plusminus}.
We can rewrite the above expression into the following numerical update:
\begin{equation}
\label{eqn:LLF_proof_1}
\begin{split}
\vec{Q_i^{n+1}} 
&=\left[ 1-\frac{\Delta t}{2\Delta x}\left(\lph+\lmh\right)\right] \vec{Q_i^n}+
\left[\frac{\Delta t \lph}{2 \Delta x} \right] \vec{M^+_i} 
+\left[\frac{\Delta t \lmh}{2 \Delta x}\right] \vec{M^-_i},
\end{split}
\end{equation}
where
\begin{align}
 \vec{M^+_i} = \vec{Q_{i+1}^n}-\left( \lph \right)^{-1} {\vec{f}\left(\vec{Q_{i+1}^n}\right)}
 \quad \text{and} \quad
 \vec{M^-_i} = \vec{Q_{i-1}^n}+\left( \lmh \right)^{-1} {\vec{f}\left(\vec{Q_{i-1}^n}\right)}.
\end{align}
Under the CFL condition \eqref{LLF_CFL_condition}, we note that
\begin{equation}
\label{eqn:LLF_proof_3}
1-\frac{\Delta t}{2 \Delta x}\left(\lph+\lmh\right)\ge0, \quad
\frac{\Delta t \lph}{2 \Delta x}  \ge 0, \quad \text{and} \quad
\frac{\Delta t \lmh}{2 \Delta x} \ge 0.
\end{equation}
Additionally, note that the coefficients in \eqref{eqn:LLF_proof_1} sum to unity:
\begin{equation}
\label{eqn:LLF_proof_4}
\left[ 1-\frac{\Delta t}{2 \Delta x}\left(\lph+\lmh\right)\right] +
\left[\frac{\Delta t \lph}{2 \Delta x} \right] 
+\left[\frac{\Delta t \lmh}{2 \Delta x}\right] = 1.
\end{equation}
Now let ${\mathcal C}_{\alpha}$ be any convex function of the conserved variables: $\vec{q} = \left( \mom_0, \mom_1, \mom_2, \mom_3, \mom_4 \right)$. Then, applying the convex function to
both sides of  \eqref{eqn:LLF_proof_1} we see that
\begin{equation}
\begin{split}
{\mathcal C}_{\alpha}\left(\vec{Q_i^{n+1}}\right) \le &\left[ 1-\frac{\Delta t}{2 \Delta x}\left(\lph+\lmh\right)\right] {\mathcal C}_{\alpha}\left(\vec{Q_i^n} \right) \\ +
&\left[\frac{\Delta t \lph}{2 \Delta x} \right] {\mathcal C}_{\alpha}\left(\vec{M^+_i} \right)
+\left[\frac{\Delta t \lmh}{2 \Delta x}\right] {\mathcal C}_{\alpha}\left(\vec{M^-_i} \right),
\end{split}
\end{equation}
which follows from conditions \eqref{eqn:LLF_proof_3} and \eqref{eqn:LLF_proof_4}, as well as the 
property of convex functions shown in lemma \ref{lemma:convexity}.

Furthermore, we note that density, pressure, and modified kurtosis ($\rho$, $p$, and $k$) 
are all convex functions of $\vec{q} = \left( \mom_0, \mom_1, \mom_2, \mom_3, \mom_4 \right)$; the density is trivially convex, while the pressure is convex if $\rho>0$, and the modified kurtosis is convex if $\rho>0$ and $p>0$ (see definitions \eqref{eqn:prim_vars_2}).
Therefore, to prove result \eqref{eqn:LLF_pos_result}, all we need to show is that
\begin{equation}
{\mathcal C}_{\alpha}\left(\vec{M^+_i} \right) > 0 \qquad \text{and} \qquad
{\mathcal C}_{\alpha}\left(\vec{M^-_i} \right) > 0,
\end{equation}
with ${\mathcal C}_{\alpha}$ chosen as the density, pressure, and modified kurtosis:
\begin{equation}
\label{eqn:convex_functions}
\begin{gathered}
{\mathcal C}_{\rho}\left(\vec{q}\right) := q_1, \quad
{\mathcal C}_{p}\left(\vec{q}\right) := q_3 - \frac{q_2^2}{q_1}, \\
{\mathcal C}_{k}\left(\vec{q}\right) := \frac{q_3^3 - 2 q_2 q_3 q_4 + q_1 q_4^2 + q_2^2 q_5 - q_1 q_3 q_5}{q_2^2 - q_1 q_3}.
\end{gathered}
\end{equation}

\begin{enumerate}
\item {\bf Density:} We take $\alpha\equiv\rho$ and note that
\begin{equation}
\label{eqn:LLF_rho_pos}
 {\mathcal C}_{\rho}\left(\vec{Q}\pm\ \lambda^{-1} \vec{f}(\vec{Q})\right)= 
 \frac{\left( \lambda \pm u \right) \rho}{\lambda}.
\end{equation}
Positivity of \eqref{eqn:LLF_rho_pos} follows from the fact that the wave speed, $\lambda$, always exceeds the local fluid speed, $|u|$.

\smallskip

\item {\bf Pressure:} We take $\alpha \equiv p$ and note that
\begin{equation}
\label{eqn:LLF_press_pos}
 {\mathcal C}_p\left(\vec{Q}\pm {\lambda^{-1}} \vec{f}(\vec{Q})\right) =
\frac{p \rho \left(\lambda \pm u\right)^2 + \heat \rho \left(u \pm \lambda \right) - p^2}{\lambda\left(\lambda \pm u \right) \rho},
\end{equation}
for which we note that the numerator is a quadratic polynomial in $\lambda$. 
The roots  of this quadratic polynomial can easily be computed:
\begin{equation*}
\begin{split}
z_1=\mp \left( u + \frac{\heat}{2p} \right) -  \sqrt{\frac{p}{\rho} + \left(\frac{\heat}{2p}\right)^2}, \quad
z_2=\mp \left( u + \frac{\heat}{2p} \right) +  \sqrt{\frac{p}{\rho} + \left(\frac{\heat}{2p}\right)^2}.
\end{split}
\end{equation*}
Positivity of \eqref{eqn:LLF_press_pos} follows from the fact that $\lambda >
\max\left\{ z_1, \, z_2 \right\}$ (i.e., $\lambda$ is always to the right of the roots) and $p \rho>0$ (i.e., the quadratic is concave up).

\smallskip

\item {\bf Modified kurtosis:} We take $\alpha \equiv k$  and note that
\begin{equation}
\label{eqn:LLF_kurt_pos}
 {\mathcal C}_k\left(\vec{Q}\pm {\lambda^{-1}}\vec{f}(\vec{Q})\right) =
 \left(\frac{k (\lambda \pm u)}{\lambda} \right) \left(\frac{p \rho \left(\lambda \pm u\right)^2  + \heat \rho \left(u \pm \lambda \right) -p^2 - k \rho}{p \rho \left(\lambda \pm u\right)^2 + \heat \rho \left(u \pm \lambda \right) - p^2} \right),
\end{equation}
where the numerator of the second fraction is again a quadratic polynomial in $\lambda$. Showing that this quadratic is positive is sufficient to show that the whole expression is positive since the remaining pieces are already
positive due to the previously established positivity of \eqref{eqn:LLF_rho_pos} and \eqref{eqn:LLF_press_pos}. The roots of the quadratic are:
\begin{equation*}
\begin{split}
z_1=\mp \left( u + \frac{\heat}{2p} \right) -  \sqrt{\frac{k}{p} + \frac{p}{\rho} + \left(\frac{\heat}{2p}\right)^2}, \quad
z_2=\mp \left( u + \frac{\heat}{2p} \right) +  \sqrt{\frac{k}{p} + \frac{p}{\rho} + \left(\frac{\heat}{2p}\right)^2}.
\end{split}
\end{equation*}
Positivity of \eqref{eqn:LLF_kurt_pos} follows from the fact that $\lambda >
\max\left\{ z_1, \, z_2 \right\}$ (i.e., $\lambda$ is always to the right of the roots) and $p \rho>0$ (i.e., the quadratic is concave up).
\end{enumerate}
\qedsymbol
\end{proof}

We have now shown that the first-order method will maintain positivity from one time step to the next under an appropriate CFL condition. However, higher-order methods will not automatically guarantee positivity; we address this issue in the subsequent subsections.

\subsection{Limiter I: Positivity-at-points in the prediction step}
\label{subsec:prediction_positivity_at_points}
The prediction step, as described in \eqref{eqn:picard} requires numerical quadrature in space-time in each Picard iteration. Furthermore, once the
predicted solution has been computed, it will again be integrated in
space-time in the correction step (i.e., see \eqref{eqn:corr_step_update_1}--\eqref{eqn:corr_step_update_4}). To guarantee that all of the
numerical quadratures in both the prediction and correction steps only use positive values of density, pressure, and modified kurtosis, we introduce a prediction-step positivity limiter.

Let the 1D Gauss-Legendre points internal to each element, augmented by the element 
end-points (i.e., the location of the element faces), be defined as follows:
\begin{equation}
\label{eqn:space_pos_points}
{\mathbb X}_{\morder} := \Bigl\{ -1, 1 \Bigr\} \cup \Bigl\{ \text{roots of the $\morderth$ degree Legendre polynomial} \Bigr\},
\end{equation}
where $\morder$ is the desired order of accuracy. Note that ${\mathbb X}_{\morder}$
contains a total of $\morder+2$ points.
We note that all of the quadratures in the prediction update \eqref{eqn:picard}
and correction update \eqref{eqn:corr_step_update_1}--\eqref{eqn:corr_step_update_4}
only depends on the predicted solution at a small number of quadrature points,
which are fully contained in the Cartesian product of ${\mathbb X}_{\morder}$ with itself:
\begin{equation}
\label{eqn:spacetime_pos_points}
{\mathbb X}^2_{\morder} := {\mathbb X}_{\morder} \otimes {\mathbb X}_{\morder}.
\end{equation}
Therefore, ${\mathbb X}^2_{\morder}$
contains a total of $(\morder+2)^2$ points.
Our goal is thus to enforce positivity at all the space-time points
$(\tau,\xi) \in {\mathbb X}^2_{\morder}$.

Following the strategy developed by
Zhang and Shu \cite{article:ZhangShu11} for the Runge-Kutta discontinuous Galerkin scheme, we apply the following procedure, which is applied element-by-element.
\begin{description}
\item[{\bf Step 1.}] On the current space-time element defined over 
\[(t,x) \in
\left[t^n, \, t^n + \Delta t\right] \times  \left[x_i - \frac{\Delta x}{2}, \, 
x_i + \frac{\Delta x}{2}\right],
\]
 the solution is given by \eqref{eqn:pred_ansatz}. Find the 
minimum density, pressure, and modified kurtosis of this solution
over the points $(\tau, \xi) \in {\mathbb X}^2_{\morder}$: 
\begin{equation}
\alpha_{i}^{(m)} := \min_{(\tau, \xi) \in {\mathbb X}^2_{\morder}}
   \left\{ \vec{\Psi}\left(\tau,\xi\right)^T \vec{W^{n+\half}_{i \, (:,m)}} \right\},
\end{equation}
for $m=1$ (density), $m=3$ (pressure), and $m=5$ (modified kurtosis).

\item[{\bf Step 2.}] Rewrite the solution as
\begin{equation}
\label{eqn:pred_step_limiter_damping}
\begin{split}
\vec{\alpha}^{\text{ST}}\left(t^{n}+\frac{\Delta t}{2} (1+\tau), \, x_i +\frac{\Delta x}{2} \, \xi; \, \theta \right) := 
 (1-\theta) \, \vec{W^{n+\half}_{i \, (1, :)}} + \theta \, \vec{\Psi}\left(\tau,\xi\right)^T \mat{W^{n+\half}_i},
\end{split}
\end{equation}
where $\theta \in [0,1]$, such that $\theta=1$ recovers the original solution
\eqref{eqn:pred_ansatz} and $\theta=0$ results in reducing the entire solution
on to its space-time average. We now choose the largest possible $\theta \in [0,1]$
so that \eqref{eqn:pred_step_limiter_damping} is positive for components
$m=1 \, (\text{density}), \, 3 \, (\text{pressure}), \, 5 \, (\text{modified kurtosis})$ at all the space-time points in ${\mathbb X}^2_{\morder}$.
This is achieved by taking
\begin{equation}
\theta = \min_{m \in \left\{1,3,5\right\}}\min\left\{ 1, \frac{{W^{n+\half}_{i \, (1, m)}} - \varepsilon}{{W^{n+\half}_{i \, (1, m)}} - \alpha_{i}^{(m)}} \right\},
\end{equation}
where $\varepsilon>0$ is a preselected small constant (e.g., in this work
we select $\varepsilon = 
10^{-14}$).
\end{description}

\subsection{Limiter II: Positivity-in-the-mean in the correction step}
\label{subsec:correction_positivity_average}
One of the key challenges in the correction step, as described by
\eqref{eqn:corr_step_update_1}--\eqref{eqn:corr_step_update_4} 
is to make sure that the element
averages of the density, pressure, and modified kurtosis remain positive
at the end of the time-step:
$\overbar{\rho}_i^{n+1}>0$,  $\overbar{p}_i^{n+1}>0$, and $\overbar{k}_i^{n+1}>0$,
where the bar over each variable refers to the element average.
The prediction step limiter described in the previous subsection, \S\ref{subsec:prediction_positivity_at_points}, helps with this positivity-in-the-mean but cannot guarantee it. Furthermore, without positivity-in-the-mean, we cannot achieve positivity of the higher-order polynomial inside the element (i.e., if the polynomial average is negative, a significant portion of the polynomial must be negative inside the element). To overcome this challenge, we extend the approach developed by Moe et al. \cite{article:MoRoSe17}, which, for the element averages, blends the high-order
 update described by \eqref{eqn:corr_step_update_1}--\eqref{eqn:corr_step_update_4}
 with a first-order Rusanov scheme. We have already proved in Theorem \ref{llf_pos} that the Rusanov scheme is guaranteed to preserve positivity. 

We begin by defining the  Rusanov \cite{article:Ru61} (aka local Lax-Friedrichs)
update based on the element averages at $t=t^n$:
\begin{equation}
\label{eqn:update_LxF_limiter}
\vec{Q_i^{\text{Rus}}} \, := \, \vec{Q^n_{i(1,:)}}-\frac{\Delta t}{\Delta x} \left( \, \vec{\NF_{i+\half}^{\text{Rus}}}-
\vec{\NF_{i-\half}^{\text{Rus}}} \, \right),
\end{equation}
where the numerical flux is given by
\begin{equation}
\vec{\NF_{i-\half}^{\text{Rus}}}:=\frac{1}{2}\left[\vec{f}\left(\vec{Q^n_{i(1,:)}}\right)+\vec{f}\left(\vec{Q^n_{i-1(1,:)}}\right)\right]-\frac{1}{2}\Bigl|\lmh\Bigr|\left(\vec{Q^n_{i(1,:)}}-\vec{Q^n_{i-1(1,:)}}\right),
\end{equation}
and $\Bigl|\lmh\Bigr|$ is a local bound of the flux Jacobian spectral radius.
 Recall that $\vec{Q_i^{\text{Rus}}}$ is guaranteed to have positive
 density, pressure, and modified kurtosis under a time-step restriction (see Theorem \ref{llf_pos}).

Next, we write the update for the element averages of the high-order
method in terms of the low-order update \eqref{eqn:update_LxF_limiter}:
\begin{equation}
\label{eqn:pos_in_mean_updated}
  \vec{Q^{n+1}_{i \, (1,:)}} = \vec{Q^\text{Rus}_{i \, (1,:)}} - \frac{\Delta t}{\Delta x} \left( \, \theta_{i+\half} \, \vec{\Delta {\mathcal F}_{i+\half}}  
  - \theta_{i-\half} \, \vec{\Delta {\mathcal F}_{i-\half}} \, \right),
\end{equation}
where the difference between the high and low-order fluxes is given by
\begin{equation}
\label{eqn:flux_differences}
\vec{\Delta {\mathcal F}_{i-\half}} := \vec{{\mathcal F}^{n+\half}_{i-\half}} 
   - \vec{{\mathcal F}^{\text{Rus}}_{i-\half}},
\end{equation}
and $\theta_{i+\half}\in[0,1]$ measures the amount of flux limiting, where $\theta_{i+\half}=0$
represents maximal limiting (i.e., no high-order flux contributions) and
$\theta_{i+\half}=1$ represents no limiting (i.e., no low-order flux contributions). 

The strategy for the positivity-in-the-mean limiter is then to find the
maximum $\theta_{i+\half}\in[0,1]$ such that $\forall i$
\begin{equation}
{\mathcal C}_{\alpha} \left( \vec{Q^{n+1}_{i \, (1,:)}} \right) > 0,
\end{equation}
for $\alpha\equiv \rho$ (density), $\alpha\equiv p$ (pressure), 
and $\alpha\equiv k$ (modified kurtosis), as defined in \eqref{eqn:convex_functions}.
The strategy for achieving this is outlined below and
is applied element-by-element.  The process begins
by initializing $\theta_{i+\half}=1$ $\forall i$.

\begin{description}
\item[{\bf Step 1: (density)}] 
Define
\begin{equation}
\Gamma:= \frac{\Delta x}{\Delta t} \left( {Q^\text{Rus}_{i(1)}-\varepsilon} \right).
\end{equation}
Set $\Lambda_{\text{left}}=\Lambda_{\text{right}}=1$ (i.e., full high-order flux), but modify these if there is any potential for the density to decrease below
zero.

\smallskip

\begin{description}
\item[{\bf Case 1}.] If $\Delta\NF_{i-\half (1)}<0$ and $\Delta\NF_{i+\half (1)}<0$, then 
\begin{equation}
\Lambda_{\text{left}}=\Lambda_{\text{right}}=\min\left\{1,\frac{\Gamma}{\left\vert\Delta\NF_{i-\half (1)}\right\vert+\left\vert\Delta\NF_{i+\half (1)}\right\vert}\right\}.
\end{equation}

\smallskip

\item[{\bf Case 2}.]  If $\Delta\NF_{i-\half (1)}<0$ and $\Delta\NF_{i+\half (1)}>0$ then 
\begin{equation}
\Lambda_{\text{left}}=\min\left\{1,\frac{\Gamma}{\left\vert\Delta\NF_{i-\half (1)}\right\vert}\right\}.
\end{equation}

\smallskip

\item[{\bf Case 3}.]  If $\Delta\NF_{i+\half (1)}<0$ and $\Delta\NF_{i-\half (1)}>0$ then 
\begin{equation}
\Lambda_{\text{right}}=\min\left\{1,\frac{\Gamma}{\left\vert\Delta\NF_{i+\half (1)}\right\vert}\right\}.
\end{equation}
\end{description}

\item[{\bf Step 2: (pressure)}] 
 Compute the Rusanov pressure (which is guaranteed to be positive):
\begin{equation}
p^\text{Rus} := {\mathcal C}_{p} \left( \vec{Q^{\text{Rus}}_{i}}\right).
\end{equation}
Set $\mu_{11}=\mu_{10}=\mu_{01}=1$, but modify these if there is any potential for the pressure to decrease below zero.

\medskip

\begin{description}
\item[{\bf Part 2A.}] Set 
\begin{equation}
\vec{Q^\star}=\vec{Q^{n}_{i \, (1,:)}}-\frac{\Delta t}{\Delta x} \left(\Lambda_{\text{right}} \, \vec{\NF_{i+\half}^{n+\half}}-\Lambda_{\text{left}} \, \vec{\NF_{i-\half}^{n+\half}} \right),
\quad
p^\star = {\mathcal C}_{p} \left( \vec{Q^\star} \right).
\end{equation}
If $p^\star<\varepsilon$, then we set $\mu_{11}=\left({p^\text{Rus}-\varepsilon}\right)\bigl/\left(p^\text{Rus}-p^\star \right)$. 

\item[{\bf Part 2B.}] Set 
\begin{equation}
\vec{Q^\star}=\vec{Q^{n}_{i \, (1,:)}} + \frac{\Delta t}{\Delta x} \Lambda_{\text{left}} \, \vec{\NF_{i-\half}^{n+\half}},
\quad
p^\star = {\mathcal C}_{p} \left( \vec{Q^\star} \right).
\end{equation}
If $p^\star<\varepsilon$, then  $\mu_{10}=\left({p^\text{Rus}-\varepsilon}\right)/\left(p^\text{Rus} - p^\star \right)$. 

\medskip

\item[{\bf Part 2C.}] Set
\begin{equation}
\vec{Q^\star}=\vec{Q^{n}_{i \, (1,:)}}-\frac{\Delta t}{\Delta x} \Lambda_{\text{right}} \, \vec{\NF_{i+\half}^{n+\half}},
\quad
p^\star = {\mathcal C}_{p} \left( \vec{Q^\star} \right).
\end{equation}
If $p^\star<\varepsilon$, then  $\mu_{01}=\left(p^\text{Rus}-\varepsilon\right)/\left(p^\text{Rus} - p^{\star}\right)$.

\medskip

\item [{\bf Part 2D.}] Set
\begin{equation}
\mu = \min\bigl\{\mu_{11},\mu_{10},\mu_{01}\bigr\}, \quad
\Lambda_{\text{left}}  \leftarrow \mu \, \Lambda_{\text{left}}, \quad
\Lambda_{\text{right}} \leftarrow \mu \, \Lambda_{\text{right}}.
\end{equation}

\end{description}

\smallskip

\item[{\bf Step 3: (modified kurtosis)}] 
 Compute the Rusanov modified kurtosis (which is guaranteed to be positive):
\begin{equation}
k^\text{Rus} := {\mathcal C}_{k} \left( \vec{Q^{\text{Rus}}_{i}}\right).
\end{equation}
Set $\mu_{11}=\mu_{10}=\mu_{01}=1$, but modify these if there is any potential for the pressure to decrease below zero.

\medskip

\begin{description}
\item[{\bf Part 3A.}] Set 
\begin{equation}
\vec{Q^\star}=\vec{Q^{n}_{i \, (1,:)}}-\frac{\Delta t}{\Delta x}\left(\Lambda_{\text{right}} \, \vec{\NF_{i+\half}^{n+\half}}-\Lambda_{\text{left}} \, \vec{\NF_{i-\half}^{n+\half}} \right),
\quad
k^\star = {\mathcal C}_{k} \left( \vec{Q^\star} \right).
\end{equation}
If $k^\star<\varepsilon$, then we set $\mu_{11}=\left({k^\text{Rus}-\varepsilon}\right)\bigl/\left(k^\text{Rus}-k^\star \right)$. 

\item[{\bf Part 3B.}] Set 
\begin{equation}
\vec{Q^\star}=\vec{Q^{n}_{i \, (1,:)}} + \frac{\Delta t}{\Delta x} \Lambda_{\text{left}} \, \vec{\NF_{i-\half}^{n+\half}},
\quad
k^\star = {\mathcal C}_{k} \left( \vec{Q^\star} \right).
\end{equation}
If $k^\star<\varepsilon$, then  $\mu_{10}=\left({k^\text{Rus}-\varepsilon}\right)/\left(k^\text{Rus} - k^\star \right)$. 

\medskip

\item[{\bf Part 3C.}] Set
\begin{equation}
\vec{Q^\star}=\vec{Q^{n}_{i \, (1,:)}}-\frac{\Delta t}{\Delta x} \Lambda_{\text{right}} \, \vec{\NF_{i+\half}^{n+\half}},
\quad
k^\star = {\mathcal C}_{k} \left( \vec{Q^\star} \right).
\end{equation}
If $k^\star<\varepsilon$, then  $\mu_{01}=\left(p^\text{Rus}-\varepsilon\right)/\left(k^\text{Rus} - k^{\star}\right)$.

\medskip

\item [{\bf Part 3D.}] Set
\begin{equation}
\mu = \min\bigl\{\mu_{11},\mu_{10},\mu_{01}\bigr\}, \quad
\Lambda_{\text{left}}  \leftarrow \mu \, \Lambda_{\text{left}}, \quad
\Lambda_{\text{right}} \leftarrow \mu \, \Lambda_{\text{right}}.
\end{equation}

\end{description}

\item[{\bf Step 4:}] Set
\begin{equation}
  \theta_{i-\half} \leftarrow \min \left\{ \theta_{i-\half}, \, 
  \Lambda_{\text{left}} \right\} \quad \text{and} \quad
  \theta_{i+\half} \leftarrow \min \left\{ \theta_{i+\half}, \, 
  \Lambda_{\text{right}} \right\}.
\end{equation}

\end{description}
In all of the above formulas, we select in this work: $\varepsilon = 
10^{-14}$.


\subsection{Limiter III: Positivity-at-points in the correction step}
\label{subsec:correction_positivity_at_points}
Once we have ensured that the element averages are positive, we then look to enforce positivity of the corrected solution at spatial quadrature points:  $\xi \in {\mathbb X}_{\morder}$ as defined by \eqref{eqn:space_pos_points}. Following the ideas developed by
Zhang and Shu \cite{article:ZhangShu11} for the Runge-Kutta discontinuous Galerkin scheme, we aim to find the maximum $\theta\in[0,1]$ such that
\begin{equation}
\label{eqn:corr_limiter_scaled}
\vec{q}^{h}\left( t^{n+1}, \, x_i +\frac{\Delta x}{2} \, \xi; \theta \right) := \left(1-\theta \right) \, \vec{Q^{n+1}_{i \, (1, :)}} + \theta \,  \vec{\Phi} (\xi)^T \, 
\mat{Q^{n+1}_{i}}
\end{equation}
is positive at all points $\xi \in {\mathbb X}_{\morder}$ for every space 
element $\Tm_i$.
As in the prediction step limiter from \S\ref{subsec:prediction_positivity_at_points}, $\theta=0$ means that the solution is limited fully down to its element average, while $\theta=1$ means that no limiting is needed and the full high-order approximation can be used.
We apply the following procedure element-by-element.
\begin{description}
\item[{\bf Step 1.}] On the current element defined over 
\[x \in \left[x_i - \frac{\Delta x}{2}, \, 
x_i + \frac{\Delta x}{2}\right],
\]
 the solution is given by \eqref{eqn:corr_ansatz}. Find the 
minimum density $\forall \xi \in {\mathbb X}_{\morder}$ (see \eqref{eqn:space_pos_points}) and compute the corresponding damping parameter
($\theta$): 
\begin{equation}
\rho_{i}^{\text{min}} := \min_{\xi \in {\mathbb X}_{\morder}}
   \left\{ \vec{\Phi}\left(\xi\right)^T \vec{Q^{n+1}_{i \, (:,1)}} \right\},
   \quad
   \theta = \min\left\{ 1, \frac{{Q^{n+1}_{i \, (1, 1)}} - \varepsilon}{{Q^{n+1}_{i \, (1, 1)}} - \rho_{i}^{{\text{min}}}} \right\}.
\end{equation}
Finally, rescale the higher-order coefficients using the above
calculated damping parameter ($\theta$):
\begin{equation}
\vec{Q^{n+1}_{i \, (:,\ell)}} \leftarrow \theta \, \vec{Q^{n+1}_{i \, (:,\ell)}}
\quad  \forall\ell=2,\ldots,\mcorr.
\end{equation}

\item[{\bf Step 2.}] Now that the density is positive $\forall \xi \in {\mathbb X}_{\morder}$, we repeat {\bf Step 1} for the pressure. That is,
we find the average pressure and the minimum pressure $\forall \xi \in {\mathbb X}_{\morder}$: 
\begin{equation}
\overbar{P}_i := {\mathcal C}_p \left(\vec{Q^{n+1}_{i \, (1, :)}}\right),
   \quad
p_{i}^{\text{min}} := \min_{\xi \in {\mathbb X}_{\morder}}
   \left\{ {\mathcal C}_p \left(\vec{\Phi}\left(\xi\right)^T \mat{Q^{n+1}_{i}}\right)\right\},
   \end{equation}
 where ${\mathcal C}_p\left(\vec{q}\right)$ is defined in \eqref{eqn:convex_functions}.
From here, we compute the corresponding damping parameter and rescale the higher-order coefficients:
\begin{equation}
   \theta = \min\left\{ 1, \frac{\overbar{P}_i - \varepsilon}{\overbar{P}_i - p_{i}^{{\text{min}}}} \right\}, \quad
   \vec{Q^{n+1}_{i \, (:,\ell)}} \leftarrow \theta \, \vec{Q^{n+1}_{i \, (:,\ell)}} \quad \forall\ell=2,\ldots,\mcorr.
\end{equation}

\item[{\bf Step 3.}] Now that both density and pressure are positive $\forall \xi \in {\mathbb X}_{\morder}$, we repeat {\bf Step 2} for the modified kurtosis. That is, we find the average modified kurtosis and the minimum modified kurtosis $\forall \xi \in {\mathbb X}_{\morder}$: 
\begin{equation}
\overbar{K}_i := {\mathcal C}_k \left(\vec{Q^{n+1}_{i \, (1, :)}}\right),
   \quad
k_{i}^{\text{min}} := \min_{\xi \in {\mathbb X}_{\morder}}
   \left\{ {\mathcal C}_k \left(\vec{\Phi}\left(\xi\right)^T \mat{Q^{n+1}_{i}}\right)\right\},
   \end{equation}
 where ${\mathcal C}_k\left(\vec{q}\right)$ is defined in \eqref{eqn:convex_functions}.
From here, we compute the corresponding damping parameter and rescale the higher-order coefficients:
\begin{equation}
   \theta = \min\left\{ 1, \frac{\overbar{K}_i - \varepsilon}{\overbar{K}_i - k_{i}^{{\text{min}}}} \right\}, \quad
   \vec{Q^{n+1}_{i \, (:,\ell)}} \leftarrow \theta \, \vec{Q^{n+1}_{i \, (:,\ell)}}
 \quad \forall\ell=2,\ldots,\mcorr.
\end{equation}
\end{description}
In all the formulas presented above we use $\varepsilon = 10^{-14}$.

\subsection{Limiter IV: Unphysical oscillation limiter}
\label{subsec:oscillation_limiter}
The previously described limiters guarantee positivity for $\rho$, $p$, and $k$, but there still may be unphysical oscillations near shocks, rarefactions, or large gradients. We augment the method with one more limiter to eliminate these oscillations: a variant of the strategy developed in Moe et al. \cite{article:Moe15}. This limiter is applied once per time step and can remove unphysical oscillations without overly diffusing the numerical solution. 
We apply the following procedure.

\begin{description}
\item[{\bf Step 1.}] Loop over each element $\Tm_i$ and compute the minimum and maximum values of all of the following variables: $w^\ell \in \left\{\rho, u, p, \heat, r\right\}$:
\begin{equation}
    w^\ell_{M_i} := \max_{\xi\in{\mathbb X}_{\morder}} \left\{ w^\ell\left(\vec{q^{\Delta x}}(\xi )\right)\Bigl|_{\Tm_i} \right\},
     \quad
    w^\ell_{m_i} := \min_{\xi\in{\mathbb X}_{\morder}} \left\{ w^\ell\left(\vec{q^{\Delta x}}(\xi )\right)\Bigl|_{\Tm_i} \right\},
\end{equation}
for all $\ell=1,2,3,4,5$. Here ${\mathbb X}_{\morder}$ is taken to be the $\morder$ roots of the $\morder^{\text{th}}$ 
Legendre polynomial (i.e., Gauss-Legendre points) plus the element ends points (see
\cref{eqn:space_pos_points}).

\medskip

\item[{\bf Step 2.}] Compute upper and 
lower bounds over all neighborhoods, $N_{\Tm_i} := \left\{ \Tm_{i-1}, \Tm_{i}, \, \Tm_{i+1}\right\}$:
\begin{equation}
\label{eqn:limiter_min_max_bounds}
\begin{split}
    {M^{\ell}_i} &= {\max \left\{ \overbar{w}^{\ell}_i + {\mathcal A}_0 h^{1.5}, \, \displaystyle\max_{j \in N_{\Tm_i}}
     \left\{ w^{\ell}_{M_j}\right\} \right\}}, \\
     {m^{\ell}_i} &= {\min\left\{\overbar{w}^{\ell}_i-{\mathcal A}_0 h^{1.5}, \,  \displaystyle\min_{j \in N_{\Tm_i}}
     \left\{w^{\ell}_{m_j}\right\} \right\}},
     \end{split}
 \end{equation}
 where $\overbar{w}^{\ell}_i$ are the element averages of each variable, and ${\mathcal A}_0 h^{1.5}$ is used to offset these averages
 to recover high-order accuracy for smooth solutions in the limit $h\rightarrow 0$ (see Moe et al. \cite{article:Moe15} for more details).
 
 \medskip
 
\item[{\bf Step 3.}] On each element $\Tm_i$, compute the largest damping parameters between $[0,1]$ that guarantee that the high-order solution in $\Tm_i$
does not violate the maximum and minimum bounds defined by \eqref{eqn:limiter_min_max_bounds}:
 \begin{equation}
 	\theta = \min\left\{ 1, \, \mu \cdot \min_{\ell} \left\{  \frac{ M^{\ell}_i-\bar{w}^{\ell}_i }{ w^{\ell}_{M_i}-\bar{w}^{\ell}_i } \right\}, \, \mu \cdot \min_{\ell} \left\{  \frac{ m^{\ell}_i-\bar{w}^{\ell}_i }{ w^{\ell}_{m_i}-\bar{w}^{\ell}_i } \right\} \right\},
 \end{equation}
 where the factor $\mu = 10/11$ is introduced to slightly increase the aggressiveness of the limiter (again, see Moe et al. \cite{article:Moe15} for more details).
 
 \medskip
 
 \item[{\bf Step 4.}] On each element $\Tm_i$, limit the conserved variables:
 \begin{equation}
     \vec{Q^{n+1}_{i \, (\ell,:)}} \leftarrow \theta \, \vec{Q^{n+1}_{i \, (\ell,:)}} \quad \forall \ell=2,\ldots,\mcorr.
 \end{equation}
 %
\end{description}

%% file: examples.tex

In this section, we apply the proposed scheme and the corresponding limiters to several test cases.
In \S\ref{sub:conv_test}, we verify the claimed orders of accuracy on a smooth exact solution. In \S\ref{sub:ex1} and \S\ref{sub:ex2} we apply the scheme to {\it shock tube} initial data. These results demonstrate the ability of the
non-oscillatory limiter to control unphysical oscillations.
Finally, in \S\ref{sub:vacuum}, we fully validate the positivity limiters by applying the scheme to piecewise constant initial data that lead to the formation of a vacuum.
This example demonstrates the ability of the positivity limiters to prevent negative states in density, pressure, and modified kurtosis, both on the element average and on the solution values internal to the element. 
In all the cases presented in \S\ref{sub:ex1},
\S\ref{sub:ex2}, and \S\ref{sub:vacuum}, we compare the high-order scheme against a  highly-resolved first-order Rusanov scheme that is guaranteed to be positivity-preserving without the need for any limiters.

\subsection{Smooth solution convergence test}
\label{sub:conv_test}
Consider the following exact solution to the 1D HyQMOM system
\eqref{eqn:HyQMOM_sys1}--\eqref{eqn:HyQMOM_sys2} with periodic boundary conditions on $x \in \bigl[-1,1\bigr]$:
\begin{equation}
\begin{split}
\rho(t,x)&=2+\sin(2\pi(x-t)), \\
 u(t,x)=1,\quad  p(t,x)&=2, \quad
\heat(t,x)=4, \quad k(t,x)=8 - 4 \left[{\rho(t,x)}\right]^{-1}.
\end{split}
\end{equation}
The numerical solution is computed with grid resolutions of
\begin{equation}
\melems=10\times2^{\ell}, \quad \text{for} \quad \ell=0,1,2,3,4,5,
\end{equation}
up to a final time of $t=1$. We verify the order of accuracy
for the schemes with orders of accuracy $\morder=2,3,4$. 

The errors we report are based on the following error measure:
\begin{equation}
\sum_{\ell=1}^{\meq} \frac{\Bigl\|f_{\ell}-g_{\ell}\Bigr\|_{\scriptstyle  L^2[-1,1]}}{\Bigl\|g_{\ell}\Bigr\|_{\scriptstyle L^2[-1,1]}} =\sum_{\ell=1}^{\meq}\sqrt{\frac{\displaystyle\int_{-1}^1\Bigl\vert f_\ell(x)-g_\ell(x)\Bigr\vert^2\,dx}{\displaystyle\int_{-1}^1\Bigl\vert g_\ell(x)\Bigr\vert^2\,dx}},
\end{equation}
where $\vec{f}(x):  [-1,1] \mapsto \reals^{\meq}$ 
is the approximate solution and 
$\vec{g}(x):  [-1,1] \mapsto \reals^{\meq}$ is the
exact solution. 
In practice, however, we replace the exact solution with a piecewise Legendre polynomial approximation of degree $\morder+1$ on the computational mesh. Repeated
use of the orthonormality of the Legendre basis functions yields the following (approximate) relative error on a
mesh with $N$ elements and a numerical method of order $\morder$:
\begin{equation}
\label{eqn:error}
e_N:= \sum_{\ell=1}^{\meq}\sqrt{\frac{\displaystyle\sum_{i=1}^N\left\{\sum_{j=1}^{M_C}\Bigl(Q_{i(j,\ell)}-Q^\star_{i(j,\ell)}\Bigr)^2+\Bigl(Q^\star_{i(M_C+1,\ell)}\Bigr)^2\right\}}{\displaystyle\sum_{i=1}^N\sum_{j=1}^{M_C+1}\Bigl(Q^\star_{i(j,\ell)}\Bigr)^2}},
\end{equation}
where $Q$ and $Q^\star$ are the Legendre coefficients of the numerical and exact solutions at the final time, respectively. The exact solution coefficients are computed using Gaussian quadrature with $20$ quadrature points per element:
\begin{equation}
\vec{Q^\star_{i(k,:)}}:=\frac{1}{2}\sum_{a=1}^{20}\omega_a^\star \, \phi_k\left(\mu_a^\star\right) \vec{q}^{\star}\left(t=1, \, x_i+\frac{\Delta x}{2}\mu_a^\star\right),
\end{equation}
where $\omega_a^\star$ and $\mu_a^\star$ for $a=1,\dots,20$ are the weights and abscissas of the $20$ point quadrature rule, and $\vec{q}^{\star}$ is the exact solution. Gaussian quadrature rules have been tabulated in many books and websites; we obtained our data from \cite{web:gauss_quad}. 

The errors as defined by \eqref{eqn:error}, as well as the base-2 logarithms of the ratio of consecutive errors,
\begin{equation}
\label{eqn:error_ratio}
\log_2 \left( \frac{e_{N/2}}{e_{N}} \right) \approx
\log_2 \left( \frac{\left(N/2\right)^{-\morder}}{N^{-\morder}} \right) = 
\log_2 \left( 2^{\morder} \right) =  \morder,
\end{equation}
are shown in Tables \ref{table:conv_test_nolimiters} (all limiters are turned off) 
and \ref{table:conv_test_fulllimiters} (all limiters are turned on).  
For the simulations that result in Table \ref{table:conv_test_fulllimiters}, 
none of the three positivity limiters (i.e., Limiters I, II, and III) are active because the solution is far away from positivity violations. In Table \ref{table:conv_test_fulllimiters}, the 
values affected by Limiter IV are highlighted in \textcolor{red}{red}.
Note that at low resolutions, Limiter IV is active, and the results in Tables \ref{table:conv_test_nolimiters} and \ref{table:conv_test_fulllimiters} differ slightly, but that at higher resolutions, the effect of the limiter disappears.
Note that for all the simulations with limiters turned on, we used the value of ${\mathcal A}_0 = 5$ in formula \cref{eqn:limiter_min_max_bounds}.

\begin{table}
\begin{center}
\begin{Large}
\begin{tabular}{|c||c|c||c|c||c|c|}
\hline
{\normalsize $N$} & {\normalsize $M_{\text{O}}=2$} & {\normalsize Eq. \eqref{eqn:error_ratio}}  & {\normalsize $M_{\text{O}}=3$} & {\normalsize Eq. \eqref{eqn:error_ratio}}  & {\normalsize $M_{\text{O}}=4$} & {\normalsize Eq. \eqref{eqn:error_ratio}} \\
\hline\hline
{\normalsize 10} & {\normalsize 1.143e-01} & -- & {\normalsize 1.171e-02} & -- & {\normalsize 4.924e-03} & -- \\\hline
{\normalsize 20} & {\normalsize 2.005e-02} & {\normalsize $2.511$} & {\normalsize 2.260e-03} & {\normalsize $2.374$} & {\normalsize 4.617e-04} & {\normalsize $3.415$} \\\hline
{\normalsize 40} & {\normalsize 3.759e-03} & {\normalsize $2.415$} & {\normalsize 4.032e-04} & {\normalsize $2.487$} & {\normalsize 5.337e-06} & {\normalsize $6.435$} \\\hline
{\normalsize 80} & {\normalsize 8.802e-04} & {\normalsize $2.095$} & {\normalsize 6.077e-05} & {\normalsize $2.730$} & {\normalsize 1.962e-07} & {\normalsize $4.765$} \\\hline
{\normalsize 160} & {\normalsize 2.192e-04} & {\normalsize $2.006$} & {\normalsize 8.127e-06} & {\normalsize $2.903$} & {\normalsize 1.203e-08} & {\normalsize $4.028$} \\\hline
{\normalsize 320} & {\normalsize 5.485e-05} & {\normalsize $1.999$} & {\normalsize 1.040e-06} & {\normalsize $2.966$} & {\normalsize 7.500e-10} & {\normalsize $4.004$} \\\hline
\end{tabular} 
\caption{(\S\ref{sub:conv_test}: smooth solution convergence test with limiters turned off) 
Relative $L^2$ errors for the HyQMOM equations with variable density, constant fluid velocity, pressure, heat flux,  and fourth primitive moment, and periodic boundary conditions. In these simulations all four limiters were turned off.
\label{table:conv_test_nolimiters}}
\end{Large}
\end{center}
\end{table}

\begin{table}
\begin{center}
\begin{Large}
\begin{tabular}{|c||c|c||c|c||c|c|}
\hline
{\normalsize $N$} & {\normalsize $M_{\text{O}}=2$} & {\normalsize Eq. \eqref{eqn:error_ratio}}  & {\normalsize $M_{\text{O}}=3$} & {\normalsize Eq. \eqref{eqn:error_ratio}}  & {\normalsize $M_{\text{O}}=4$} & {\normalsize Eq. \eqref{eqn:error_ratio}} \\
\hline\hline
{\normalsize 10} & \textcolor{red}{\normalsize 3.154e-01} & -- & \textcolor{red}{\normalsize 5.360e-02} & -- & {\normalsize 4.924e-03} & -- \\\hline
{\normalsize 20} & \textcolor{red}{\normalsize 4.887e-02} & {\normalsize $2.690$} & {\normalsize 2.260e-03} & {\normalsize $4.568$} & {\normalsize 4.617e-04} & {\normalsize $3.415$} \\\hline
{\normalsize 40} & {\normalsize 3.759e-03} & {\normalsize $3.700$} & {\normalsize 4.032e-04} & {\normalsize $2.487$} & {\normalsize 5.337e-06} & {\normalsize $6.435$} \\\hline
{\normalsize 80} & {\normalsize 8.802e-04} & {\normalsize $2.095$} & {\normalsize 6.077e-05} & {\normalsize $2.730$} & {\normalsize 1.962e-07} & {\normalsize $4.765$} \\\hline
{\normalsize 160} & {\normalsize 2.192e-04} & {\normalsize $2.006$} & {\normalsize 8.127e-06} & {\normalsize $2.903$} & {\normalsize 1.203e-08} & {\normalsize $4.028$} \\\hline
{\normalsize 320} & {\normalsize 5.485e-05} & {\normalsize $1.999$} & {\normalsize 1.040e-06} & {\normalsize $2.966$} & {\normalsize 7.500e-10} & {\normalsize $4.004$} \\\hline
\end{tabular} 
\caption{(\S\ref{sub:conv_test}: smooth solution convergence test with limiters turned on) 
Relative $L^2$ errors for the one-dimensional HyQMOM equations with variable density, constant fluid velocity, pressure, heat flux,  and fourth primitive moment, and periodic boundary conditions. In these simulations, all four limiters were turned on. 
\label{table:conv_test_fulllimiters}}
\end{Large}
\end{center}
\end{table}

\subsection{Shock tube problem \#1}
\label{sub:ex1}
Consider the Riemann problem for  \eqref{eqn:HyQMOM_sys1}--\eqref{eqn:HyQMOM_prim} with the following initial data at $t=0$:
\begin{equation}
\label{eqn:shock1_init}
\bigl( \rho, u, p, \heat, k \bigr)(t=0,x) =
\begin{cases}
\bigl(1.5, \, -0.5, \, 1.5, \, 1.0, \, 2.3\bar{3}\bigr) \quad x < 0, \\
\bigl(1.0, \, -0.5, \, 1.0, \, 0.5, \, 1.75\bigr) \quad x > 0,
\end{cases}
\end{equation}
on $x\in[-1.2, 1.2]$ with extrapolation boundary
conditions. 

Shown in Figure \ref{fig:shocktube1} are results from a simulation run with 
two distinct methods: (1) the $\morder=4$ scheme
with 200 elements and full limiters (shown as blue dots), and (2)
the first-order Rusanov  scheme
with 20,000 elements (shown as a solid red line). For the
$\morder=4$ scheme, we are plotting four points per element in order
to show the intra-element solution structure. 
The panels show the primitive variables: (a) density: $\rho(t,x)$, (b) macroscopic velocity: $u(t,x)$, (c) pressure: $p(t,x)$, (d) heat flux: $\heat(t,x)$, (e) modified kurtosis: $k(t,x)$, and (f) primitive fourth-moment: $r(t,x)$. Note that we used the value of ${\mathcal A}_0 = 5$ in formula \cref{eqn:limiter_min_max_bounds}.
These results clearly demonstrate the non-oscillatory limiters' ability to adequately control unphysical oscillations and produce accurate solutions.

\begin{figure}
\begin{tabular}{cc}
(a)\includegraphics[width=.44\linewidth]{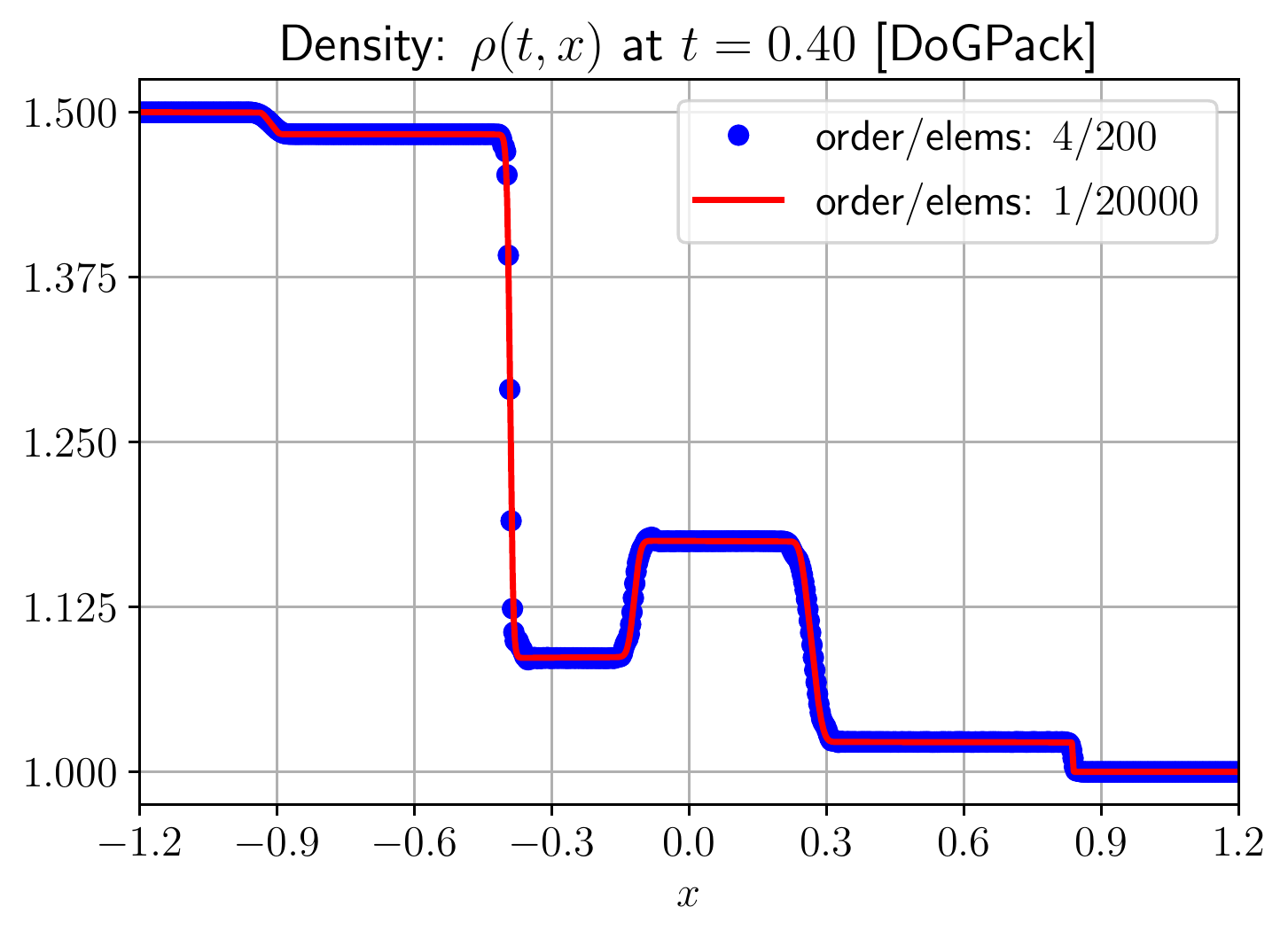} &
(b)\includegraphics[width=.44\linewidth]{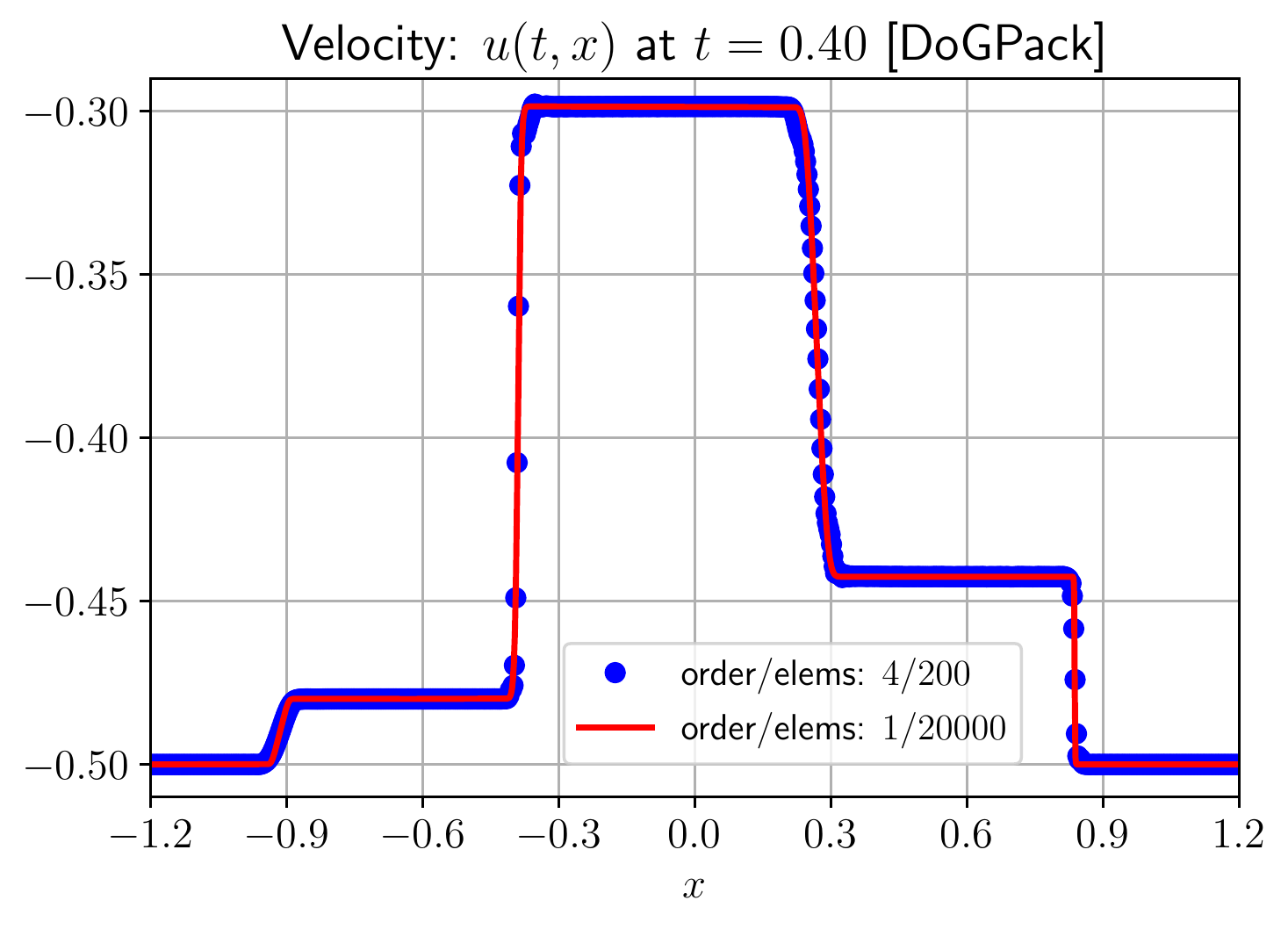} \\
(c)\includegraphics[width=.44\linewidth]{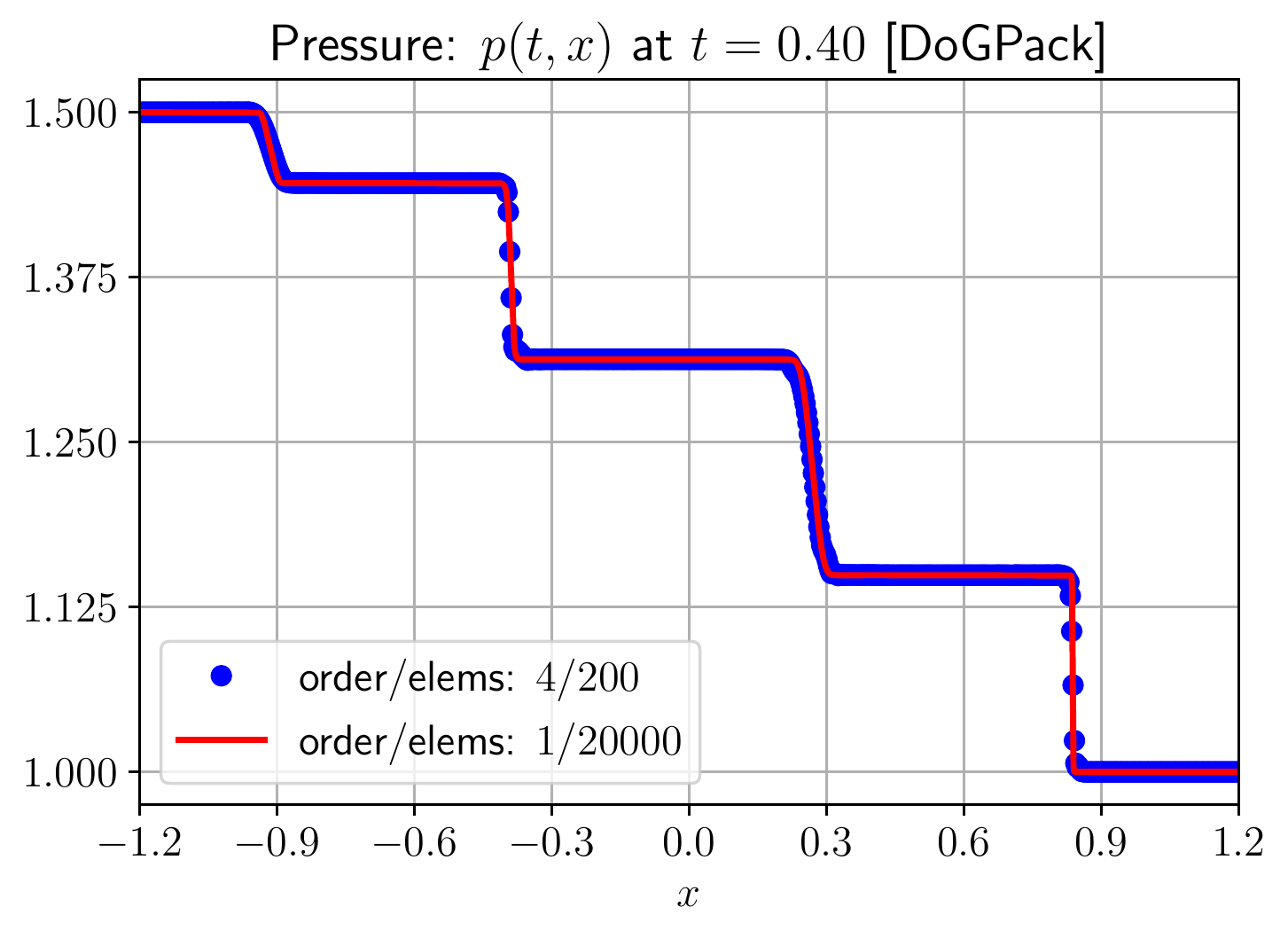} &
(d)\includegraphics[width=.44\linewidth]{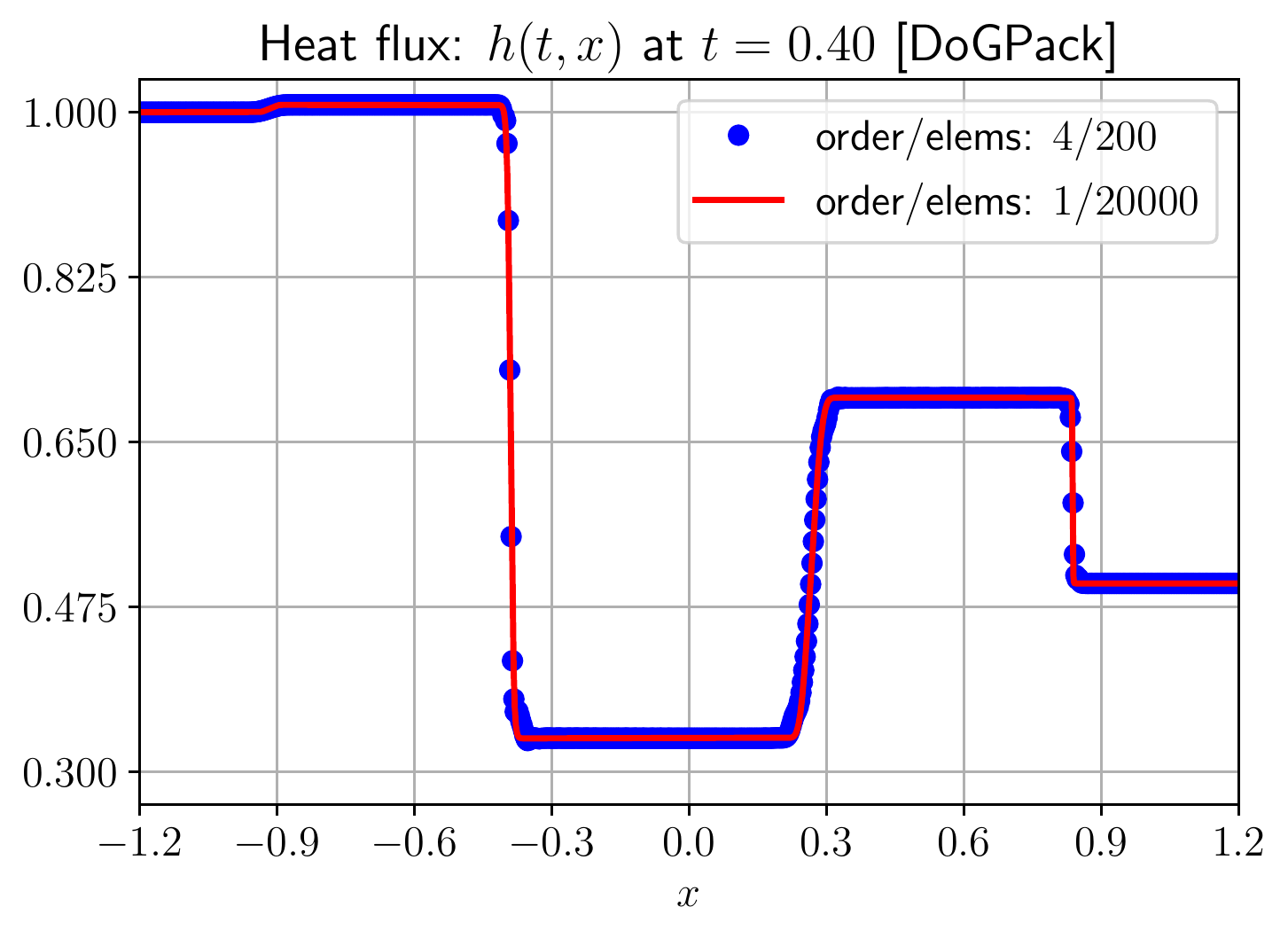} \\
(e)\includegraphics[width=.44\linewidth]{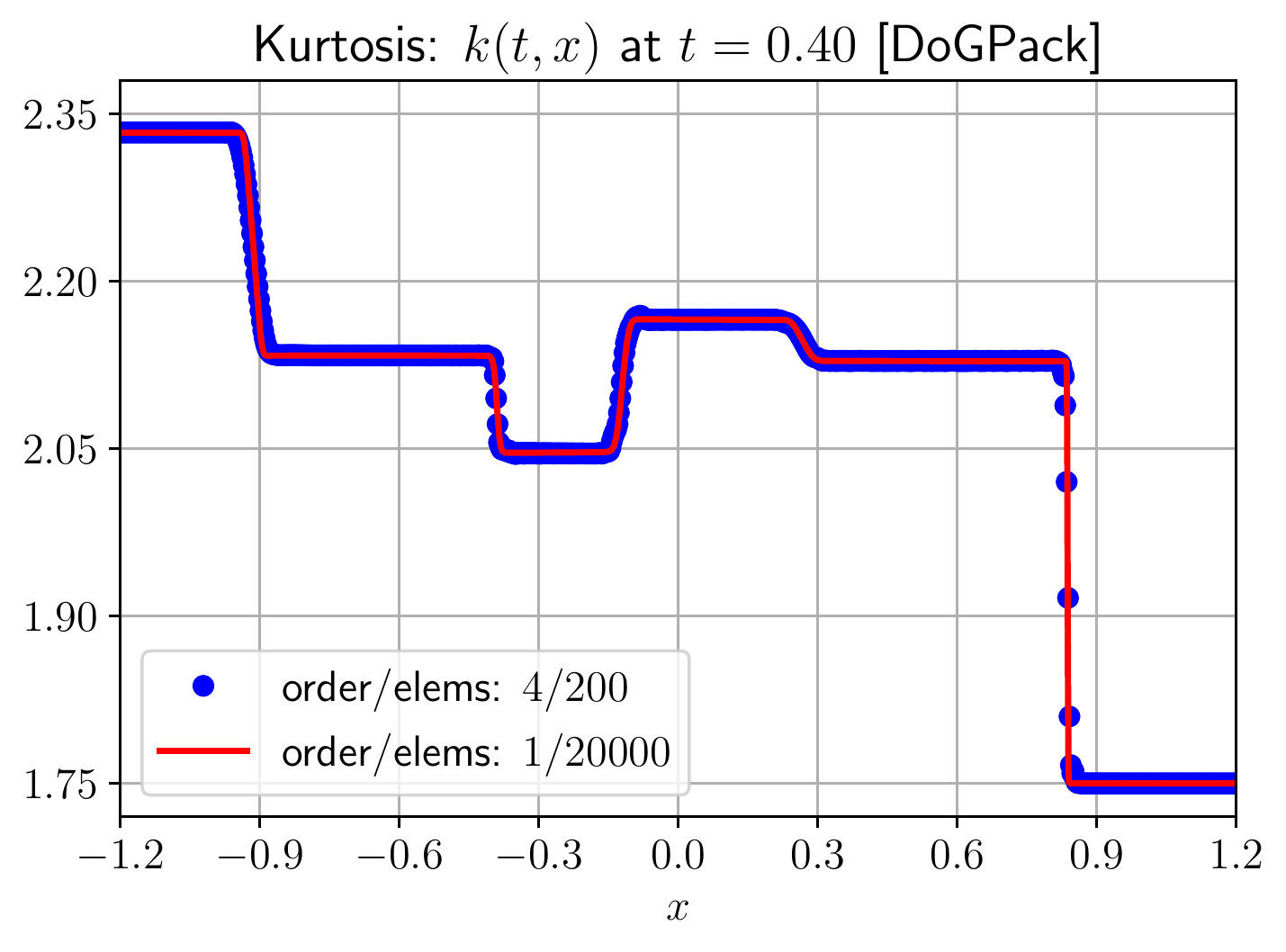} &
(f)\includegraphics[width=.44\linewidth]{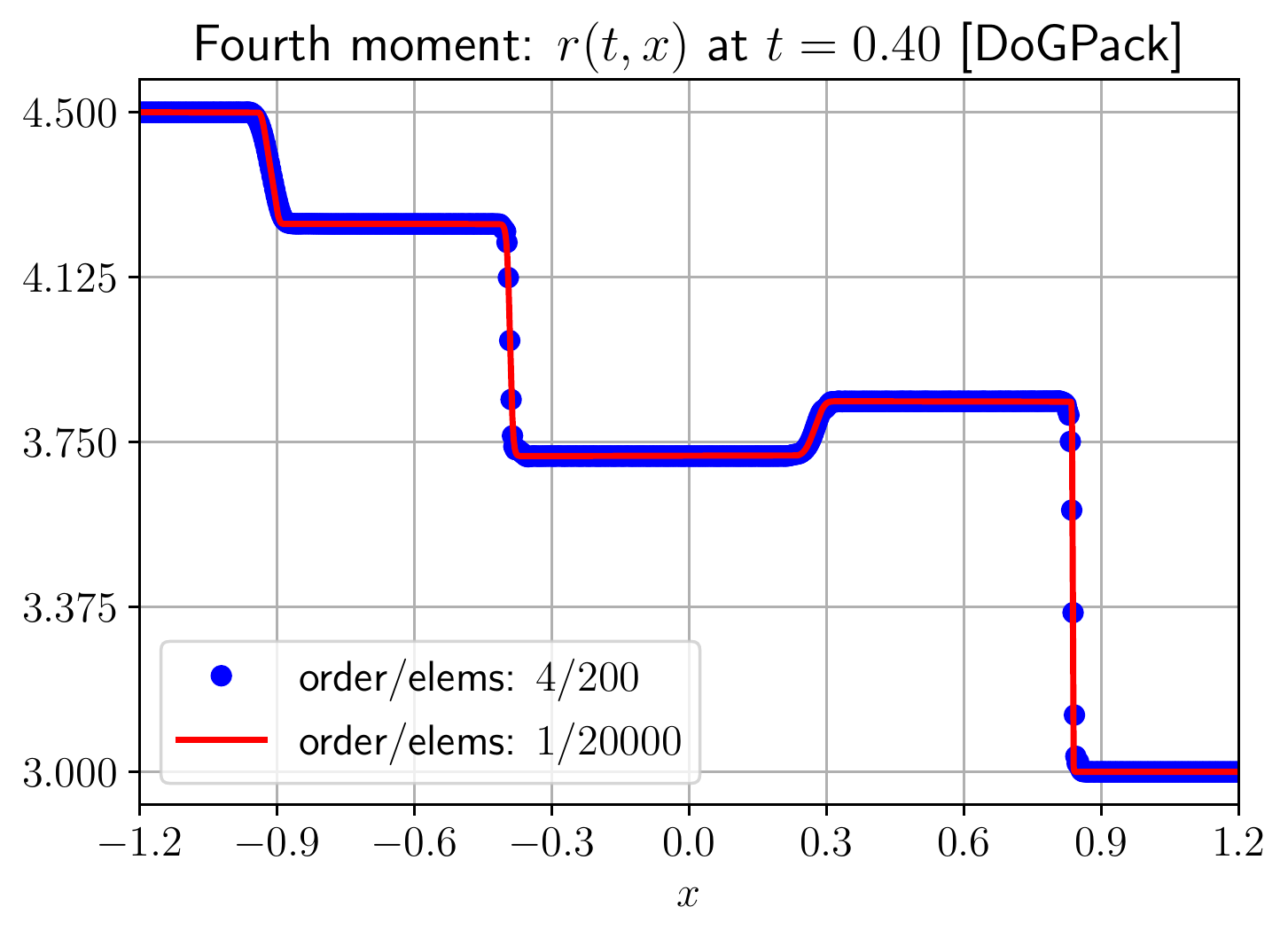}
\end{tabular}
\caption{(\S\ref{sub:ex1}: shock tube problem \#1) 
Numerical solution of shock tube problem \#1
on $x \in [-1.2, 1.2]$ with initial conditions given
by \eqref{eqn:shock1_init}. 
Shown are the results from a simulation run with 
two distinct methods: (1) the $\morder=4$ scheme
with 200 elements and full limiters (shown as blue dots), and (2)
the first-order Rusanov  scheme
with 20,000 elements (shown as a solid red line). For the
$\morder=4$ scheme, we are plotting four points per element in order
to show the intra-element solution structure. 
The panels show the primitive variables: (a) density: $\rho(t,x)$, (b) macroscopic velocity: $u(t,x)$, (c) pressure: $p(t,x)$, (d) heat flux: $\heat(t,x)$, (e) modified kurtosis: $k(t,x)$, and (f) primitive fourth-moment: $r(t,x)$.}
\label{fig:shocktube1}
\end{figure}

\subsection{Shock tube problem \#2}
\label{sub:ex2}
Consider the Riemann problem for  \eqref{eqn:HyQMOM_sys1}--\eqref{eqn:HyQMOM_prim} with the following initial data at $t=0$:
\begin{equation}
\label{eqn:shock2_init}
\bigl(\rho, u, p, \heat, k \bigr)(t=0,x) = 
\begin{cases}
\bigl(1.0, \, -0.7, \, 1.5, \, 1.5, \, 1.75 \bigr) & \quad x < 0, \\
\bigl(0.5, \, -0.9, \, 1.0, \, 1.0, \, 1.0 \, \bigr) & \quad x > 0,
\end{cases}
\end{equation}
on $x\in[-1.2, 1.2]$ with extrapolation boundary
conditions. 

Shown in Figure \ref{fig:shocktube2} are results from a simulation run with 
two distinct methods: (1) the $\morder=4$ scheme
with 200 elements and full limiters (shown as blue dots), and (2)
the first-order Rusanov  scheme
with 20,000 elements (shown as a solid red line). For the
$\morder=4$ scheme, we are plotting four points-per-element in order
to show the intra-element solution structure. 
The panels show the primitive variables: (a) density: $\rho(t,x)$, (b) macroscopic velocity: $u(t,x)$, (c) pressure: $p(t,x)$, (d) heat flux: $\heat(t,x)$, (e) modified kurtosis: $k(t,x)$, and (f) primitive fourth-moment: $r(t,x)$. Note that we used the value of ${\mathcal A}_0 = 5$ in formula \cref{eqn:limiter_min_max_bounds}.

Again, just as in the previous example, 
these results demonstrate the ability of the non-oscillatory limiters to adequately control unphysical oscillations and produce accurate solutions.

\begin{figure}
\begin{tabular}{cc}
(a)\includegraphics[width=.44\linewidth]{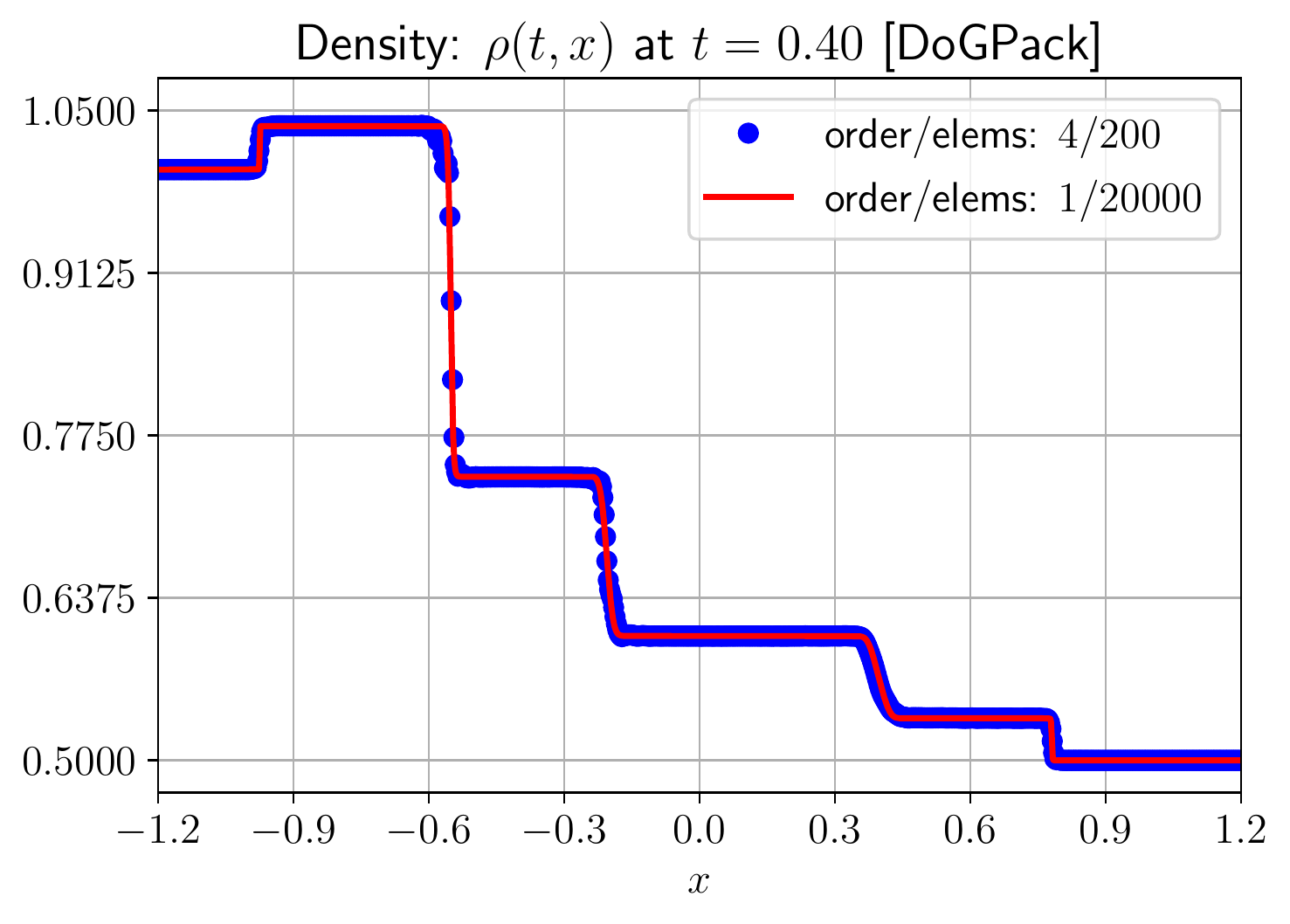} &
(b)\includegraphics[width=.44\linewidth]{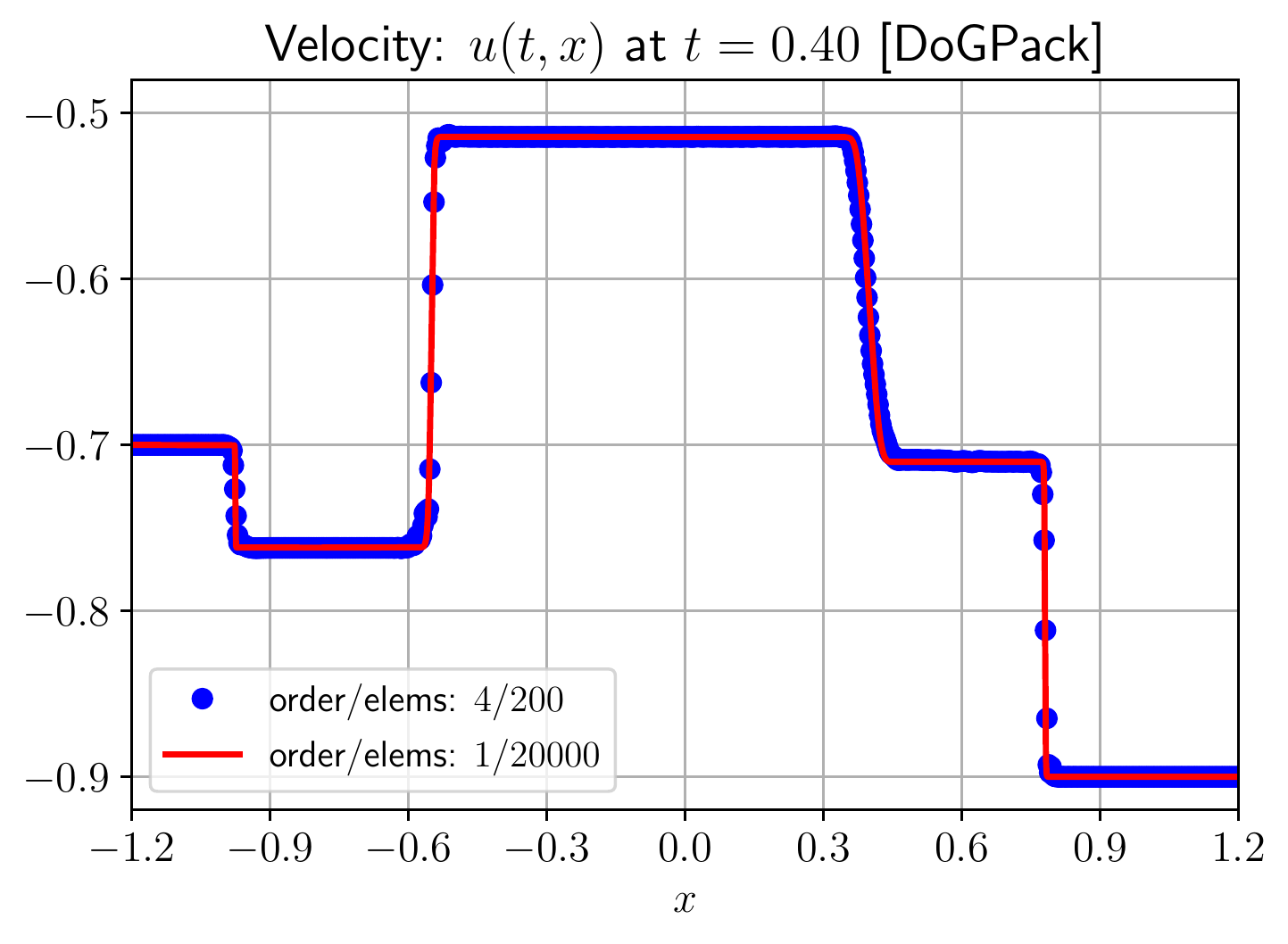} \\
(c)\includegraphics[width=.44\linewidth]{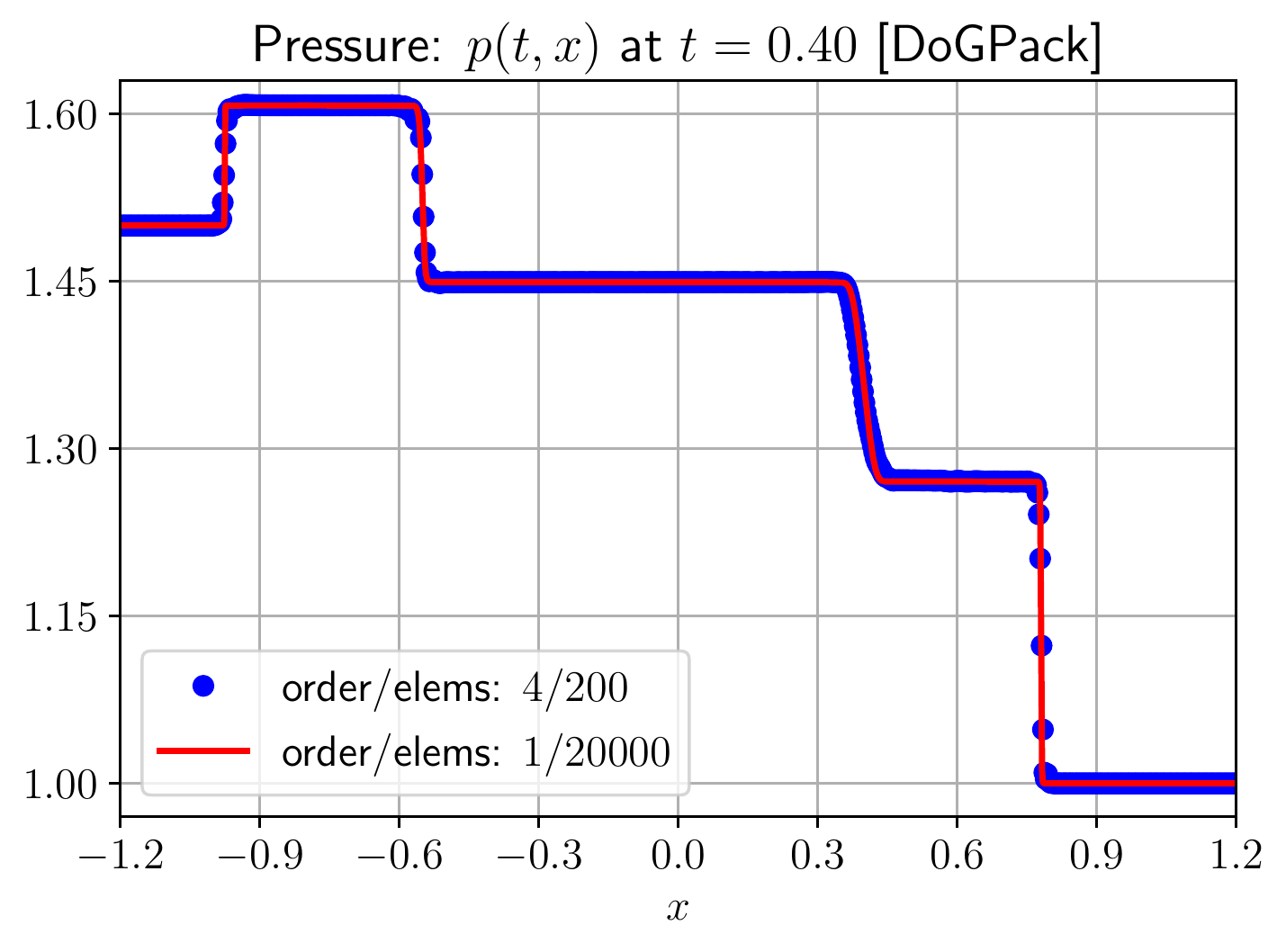} &
(d)\includegraphics[width=.44\linewidth]{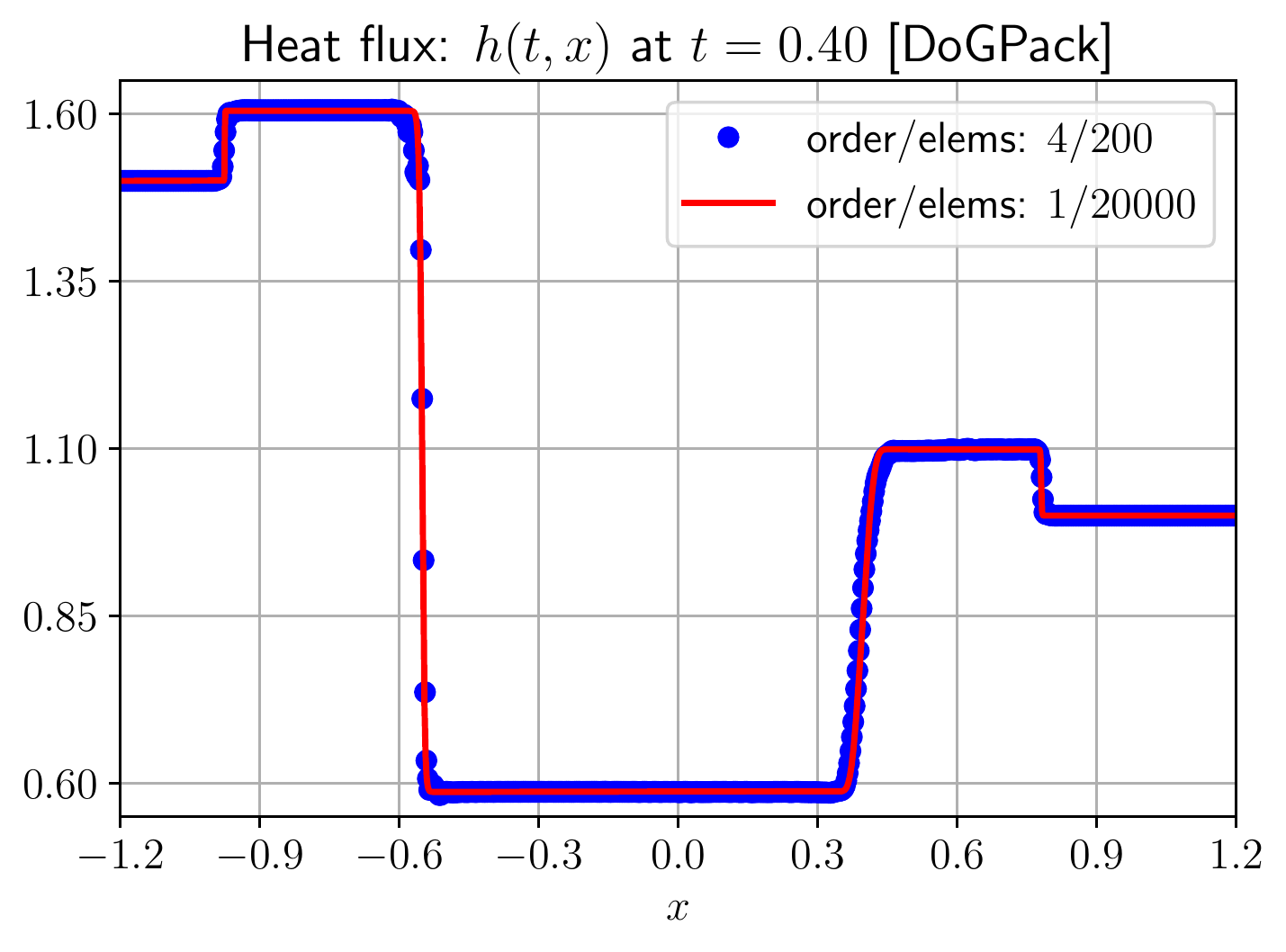} \\
(e)\includegraphics[width=.44\linewidth]{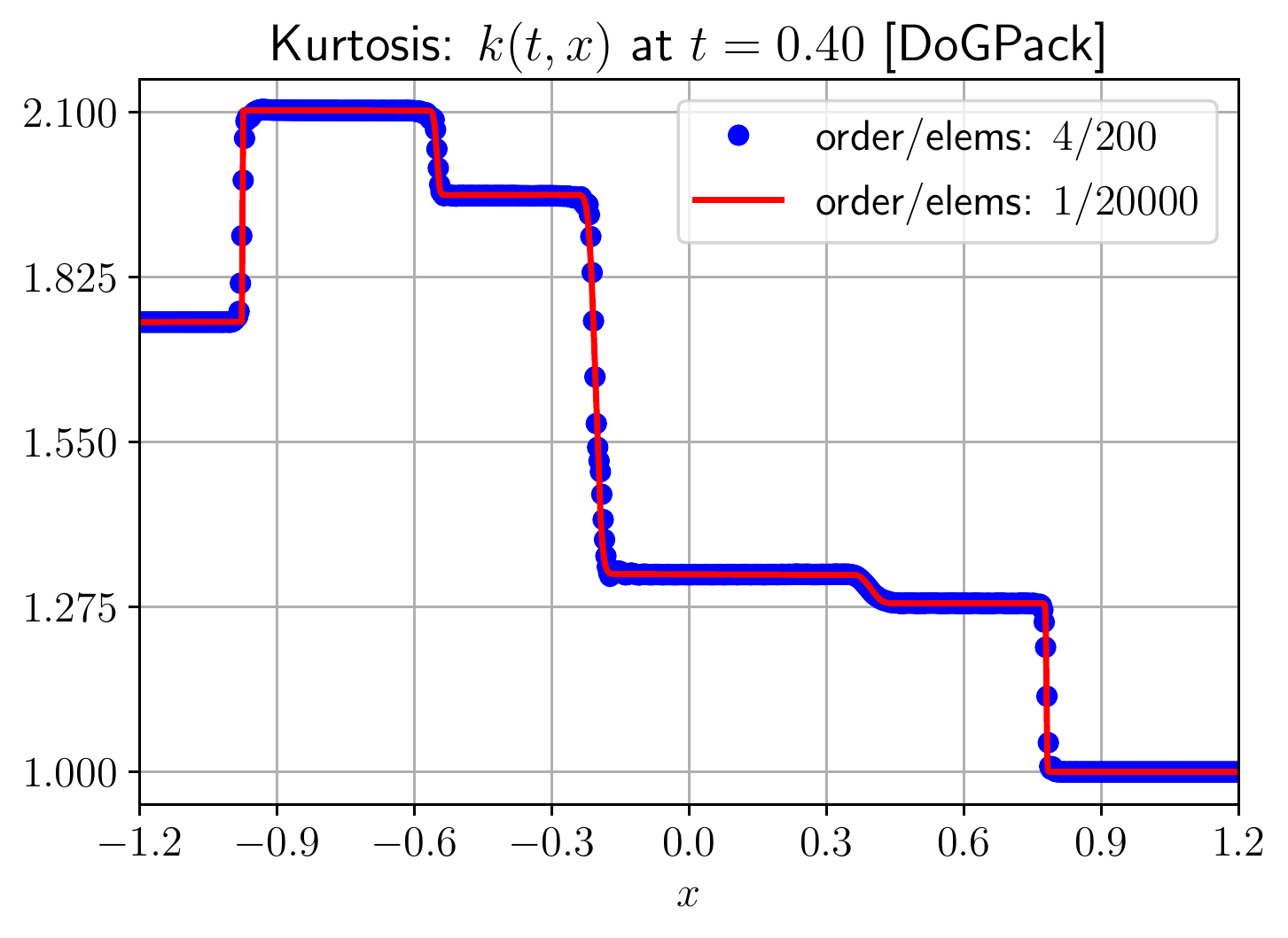} &
(f)\includegraphics[width=.44\linewidth]{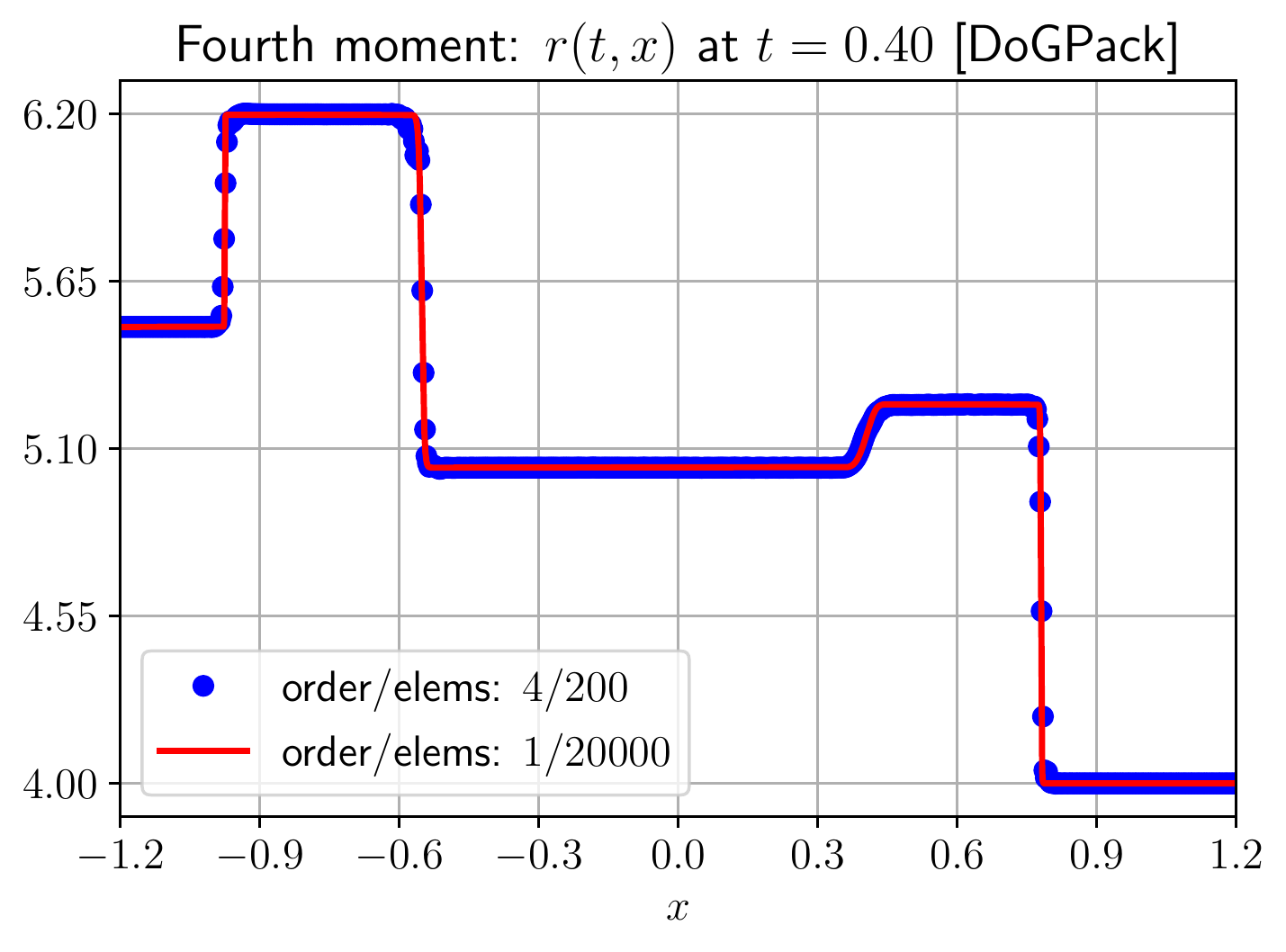}
\end{tabular}
\caption{(\S\ref{sub:ex2}: shock tube problem \#2) Numerical solution of shock tube problem \#2
on $x \in [-1.2, 1.2]$ with initial conditions given
by \eqref{eqn:shock2_init}. 
Shown are the results from a simulation run with 
two distinct methods: (1) the $\morder=4$ scheme
with 200 elements and full limiters (shown as blue dots), and (2)
the first-order Rusanov  scheme
with 20,000 elements (shown as a solid red line). For the
$\morder=4$ scheme, we are plotting four points per element in order
to show the intra-element solution structure. 
The panels show the primitive variables: (a) density: $\rho(t,x)$, (b) macroscopic velocity: $u(t,x)$, (c) pressure: $p(t,x)$, (d) heat flux: $\heat(t,x)$, (e) modified kurtosis: $k(t,x)$, and (f) primitive fourth-moment: $r(t,x)$.}
\label{fig:shocktube2}
\end{figure}

\subsection{Double rarefaction vacuum problem}
\label{sub:vacuum}
In the final example, we solve a vacuum problem where the right and left initial velocities are large and opposite, creating a vacuum state in the center
of the solution domain.  The initial states are
 \begin{equation}
 \label{eqn:vacuum_init}
\bigl(\rho, u, p, \heat, k \bigr)(t=0,x) = 
\begin{cases}
\bigl(1.0, \, -2.0, \, 1.0, \, 0.0, \, 2.0\bigr) & \quad x < 0, \\
\bigl(1.0, \, +2.0, \, 1.0, \, 0.0, \, 2.0\bigr) & \quad x > 0.
\end{cases}
\end{equation}
The computational domain is $x \in [-1.2, 1.2]$ with extrapolation boundary conditions.

Shown in Figure \ref{fig:vacuum} are results from a simulation run with 
two distinct methods: (1) the $\morder=4$ scheme
with 200 elements and full limiters (shown as blue dots), and (2)
the first-order Rusanov  scheme
with 20,000 elements (shown as a solid red line). For the
$\morder=4$ scheme, we are plotting four points per element in order
to show the intra-element solution structure. 
The panels show the primitive variables: (a) density: $\rho(t,x)$, (b) macroscopic velocity: $u(t,x)$, (c) pressure: $p(t,x)$, (d) heat flux: $\heat(t,x)$, (e) modified kurtosis: $k(t,x)$, and (f) primitive fourth-moment: $r(t,x)$.
Note that we used the value of ${\mathcal A}_0 = 5$ in formula \cref{eqn:limiter_min_max_bounds}.

We comment on two important findings from this simulation.
First, this example demonstrates the ability of the positivity limiters to prevent negative states in density, pressure, and modified kurtosis, both on the element average and the solution values internal to the element. In this simulation, all three variables, $\rho$, $p$, and $k$, become very small, but all stay strictly above zero. Because all three remain strictly positive, the moments remain realizable, and the numerical simulation remains nonlinear stable.
Second, while the simulation results from the $\morder=4$ scheme with 200 elements do show some differences in the vacuum region with the highly resolved Rusanov solution, especially in the density plot shown in Figure \ref{fig:vacuum}(a), the solution remains qualitatively correct.
We can investigate this further by increasing the grid resolution; in Figure \ref{fig:vacuum_converge} we show the density plots at different grid resolutions: (a) $N=200$,
(b) $N=400$, (c) $N=800$, and (b) $N=1600$. These results verify that the differences between
the $\morder=4$ scheme and the highly resolved Rusanov scheme disappear at higher resolutions.

\begin{figure}
\begin{tabular}{cc}
(a)\includegraphics[width=.44\linewidth]{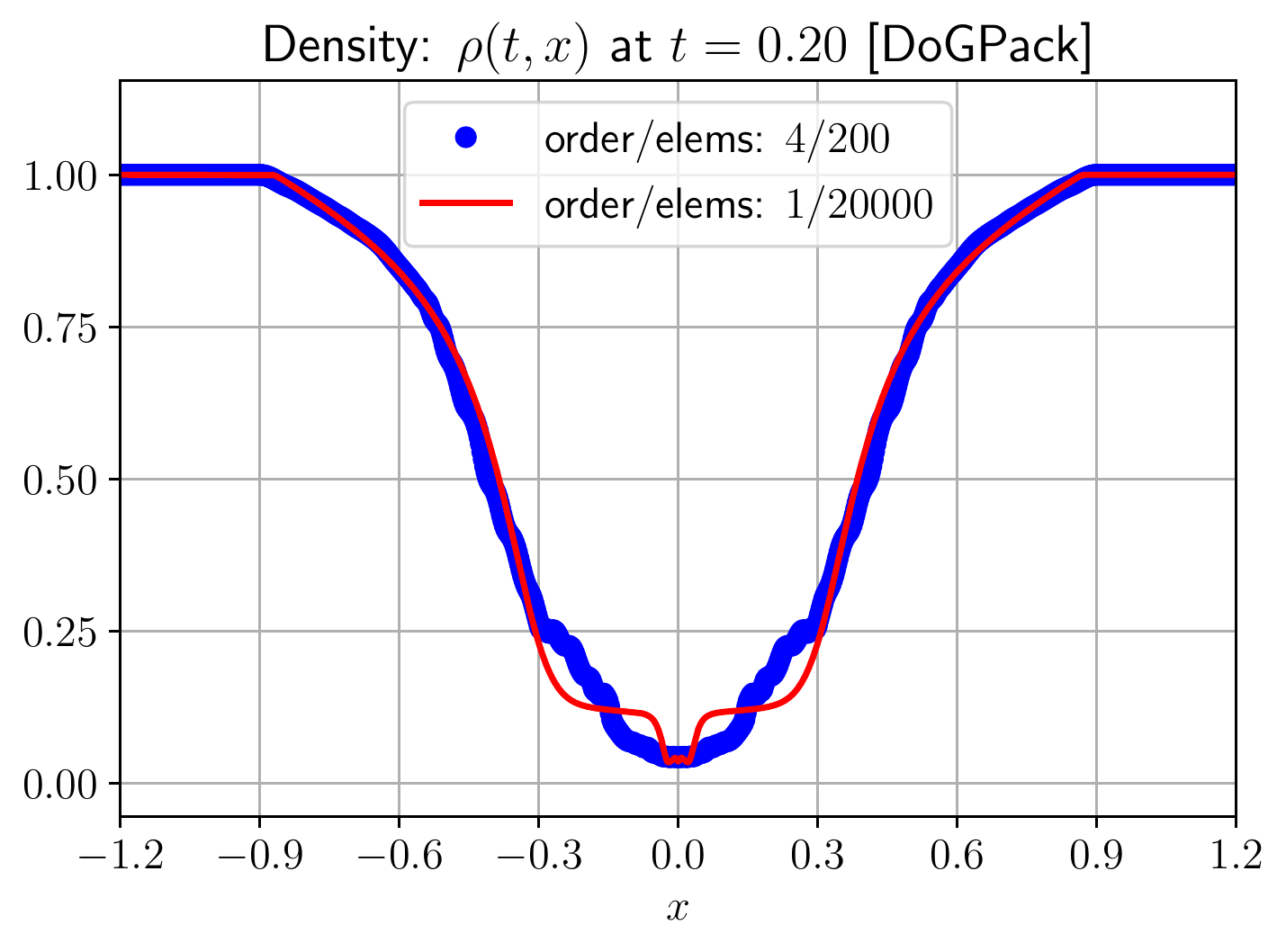} &
(b)\includegraphics[width=.44\linewidth]{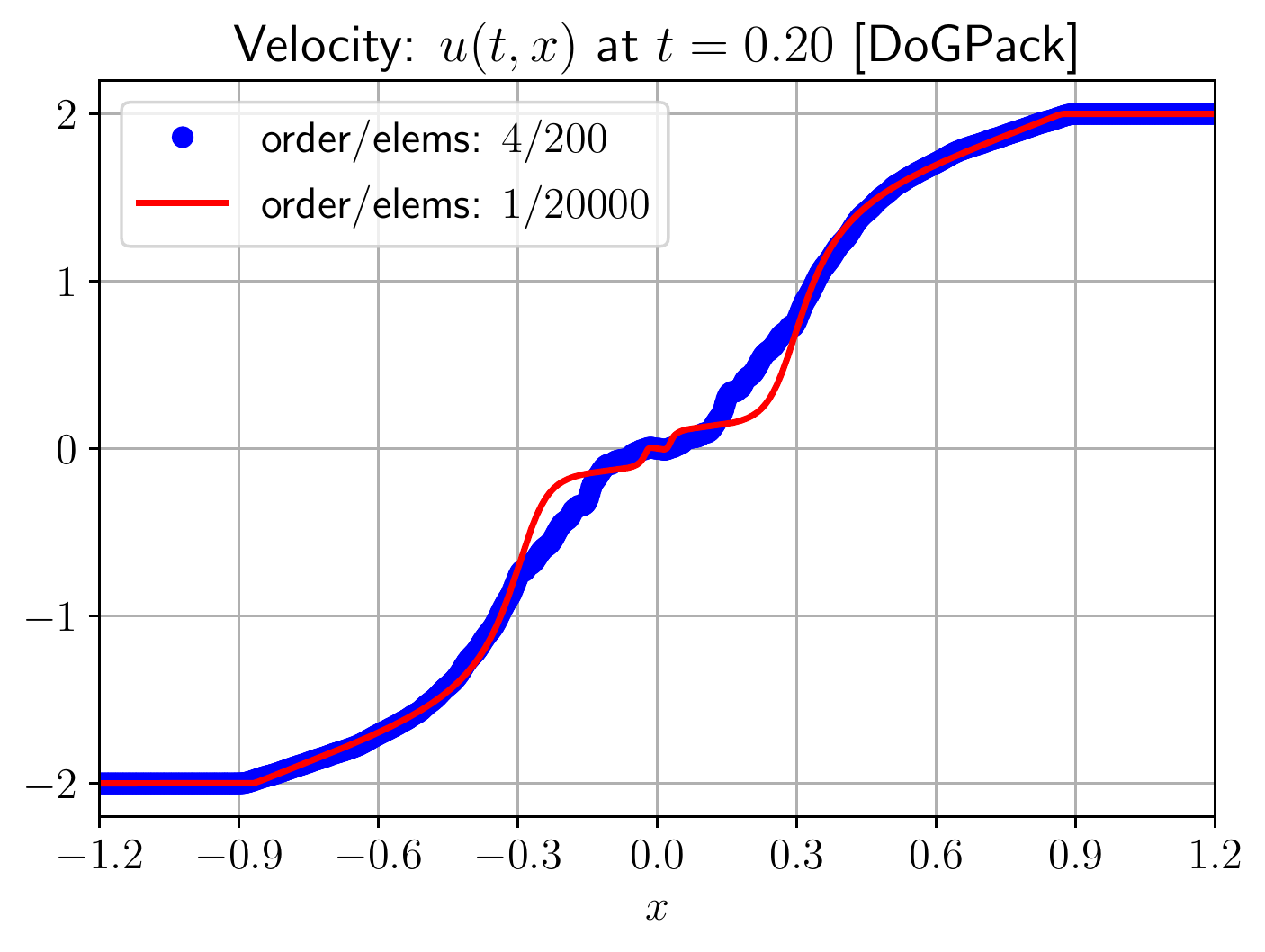} \\
(c)\includegraphics[width=.44\linewidth]{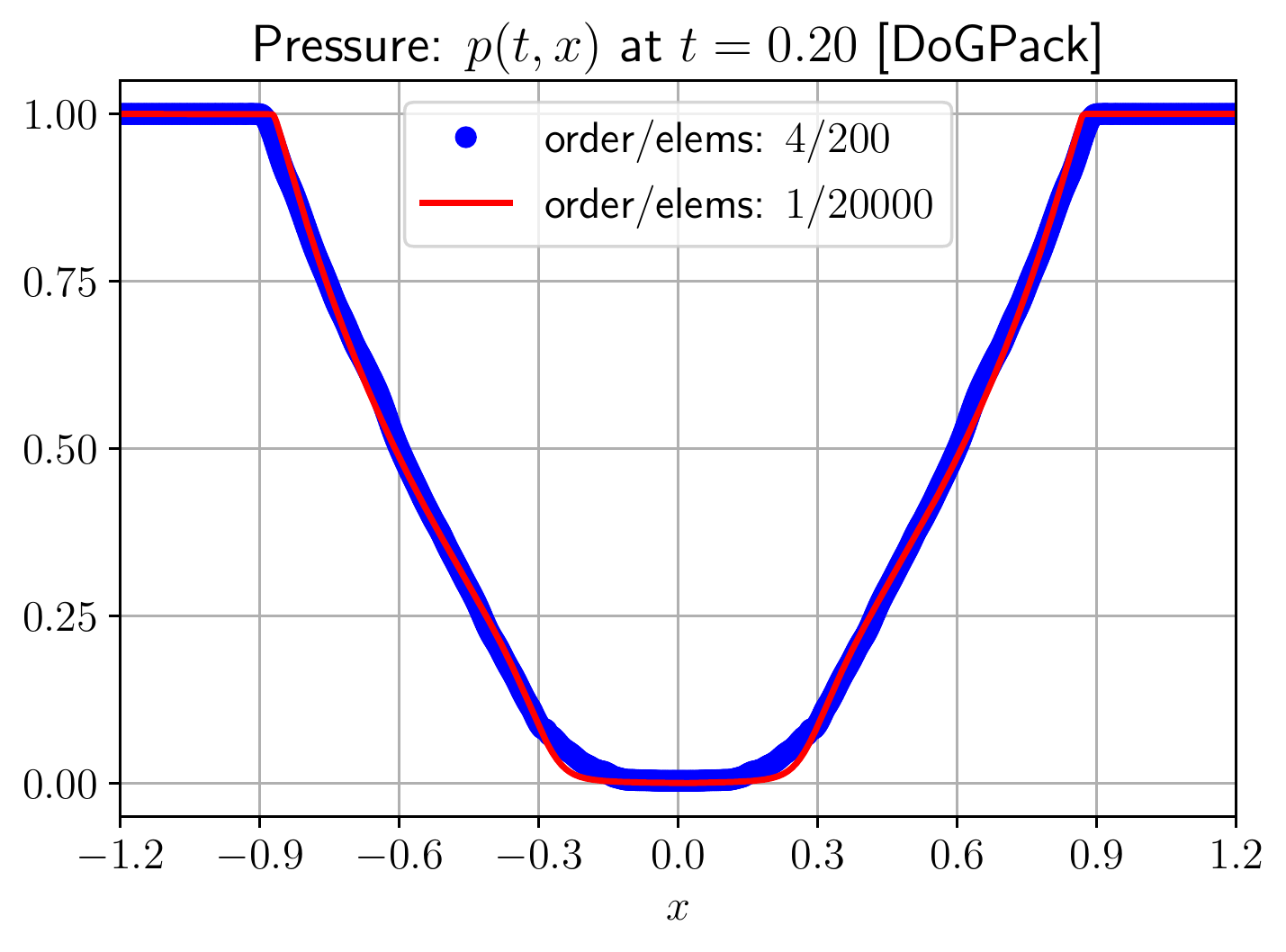} &
(d)\includegraphics[width=.44\linewidth]{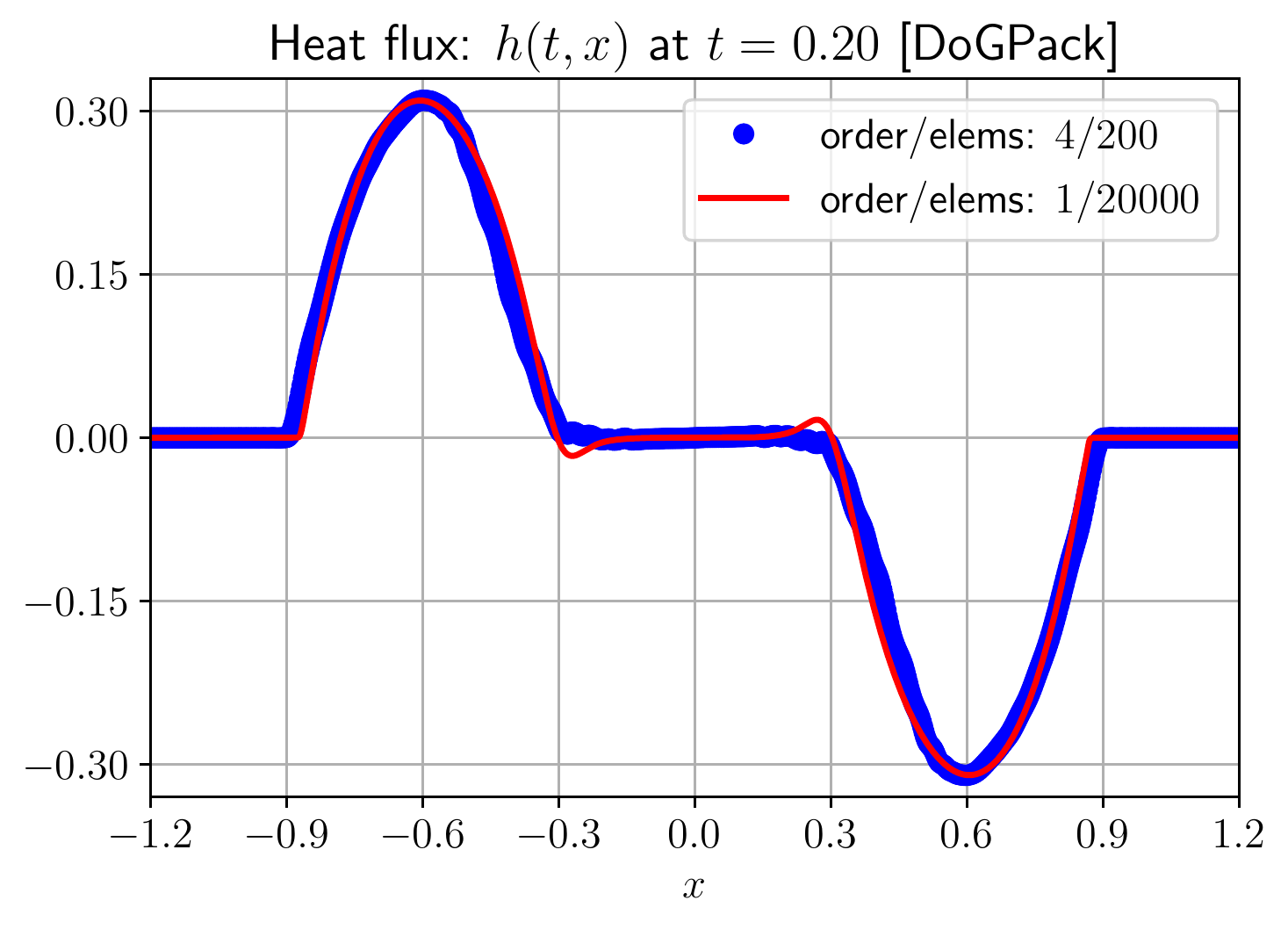} \\
(e)\includegraphics[width=.44\linewidth]{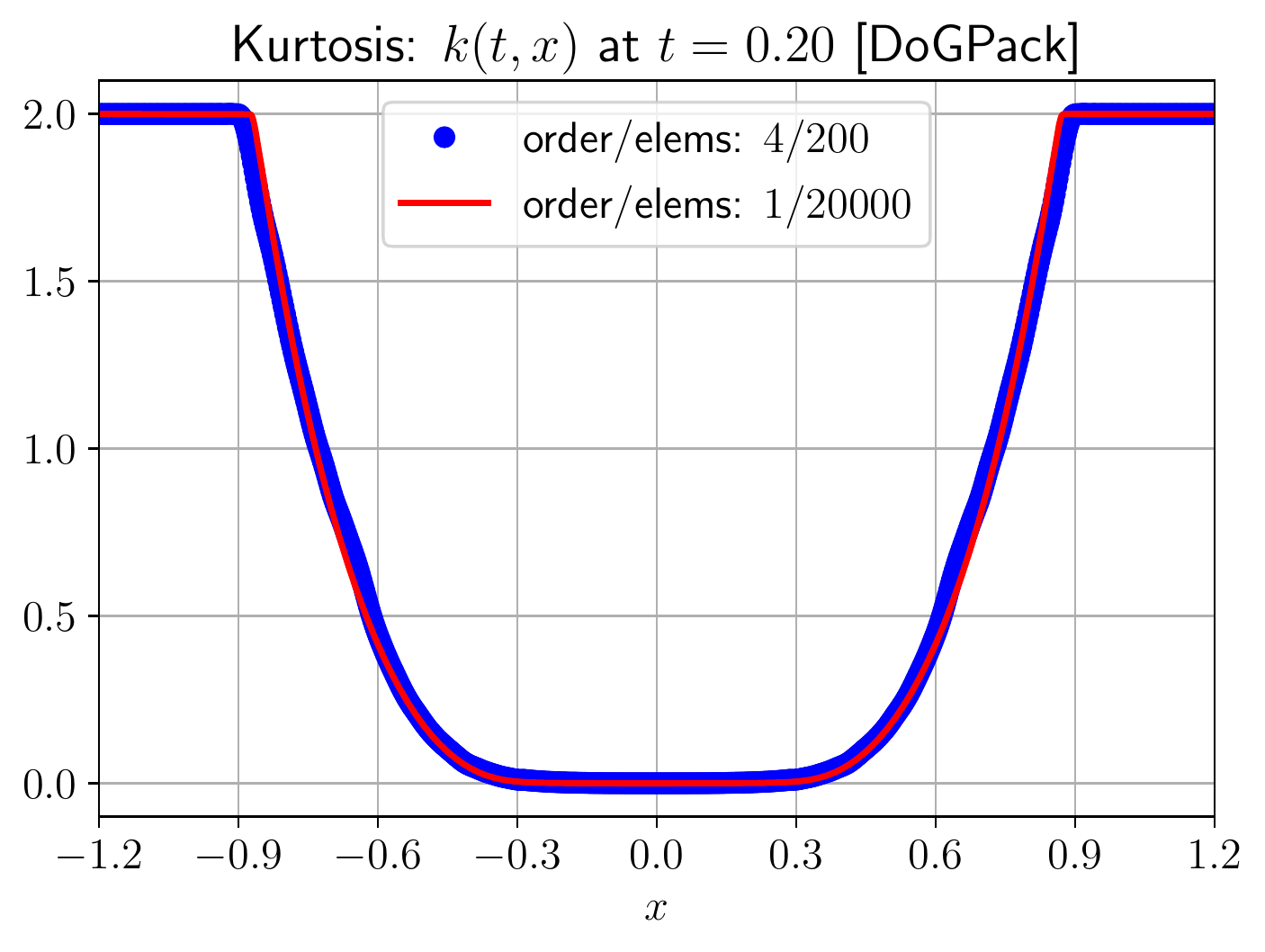} &
(f)\includegraphics[width=.44\linewidth]{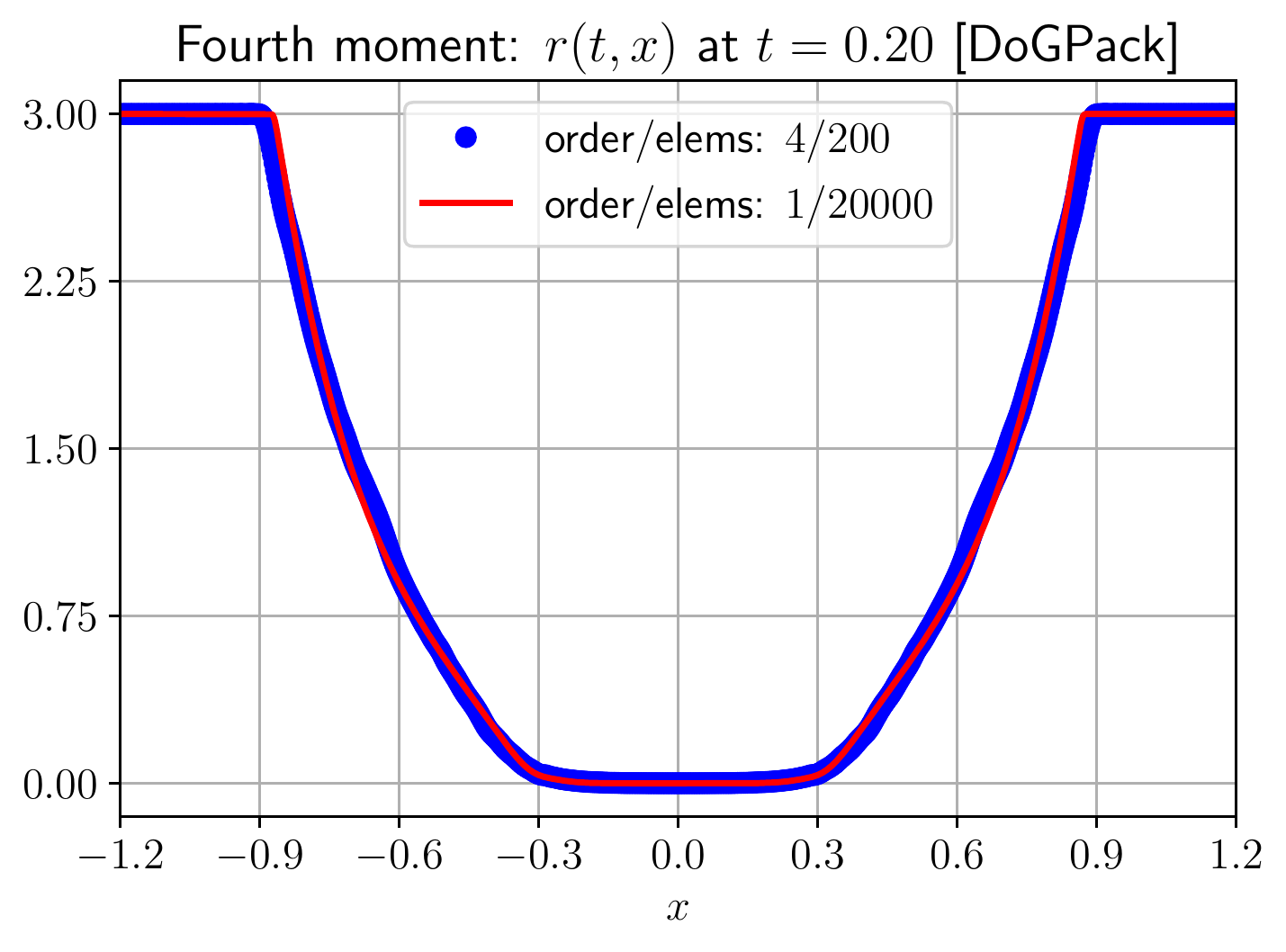}
\end{tabular}
\caption{(\S\ref{sub:vacuum}: double rarefaction vacuum problem)
Numerical solution of the double rarefaction vacuum problem
on $x \in [-1.2, 1.2]$ with initial conditions given
by \eqref{eqn:vacuum_init}. 
Shown are the results from a simulation run with 
two distinct methods: (1) the $\morder=4$ scheme
with 200 elements and full limiters (shown as blue dots), and (2)
the first-order Rusanov  scheme
with 20,000 elements (shown as a solid red line). For the
$\morder=4$ scheme, we are plotting four points per element in order
to show the intra-element solution structure. 
The panels show the primitive variables: (a) density: $\rho(t,x)$, (b) macroscopic velocity: $u(t,x)$, (c) pressure: $p(t,x)$, (d) heat flux: $\heat(t,x)$, (e) modified kurtosis: $k(t,x)$, and (f) primitive fourth-moment: $r(t,x)$.}
\label{fig:vacuum}
\end{figure}

\begin{figure}
\begin{tabular}{cc}
(a)\includegraphics[width=.44\linewidth]{figures/density_double_rarefaction_200.pdf} &
(b)\includegraphics[width=.44\linewidth]{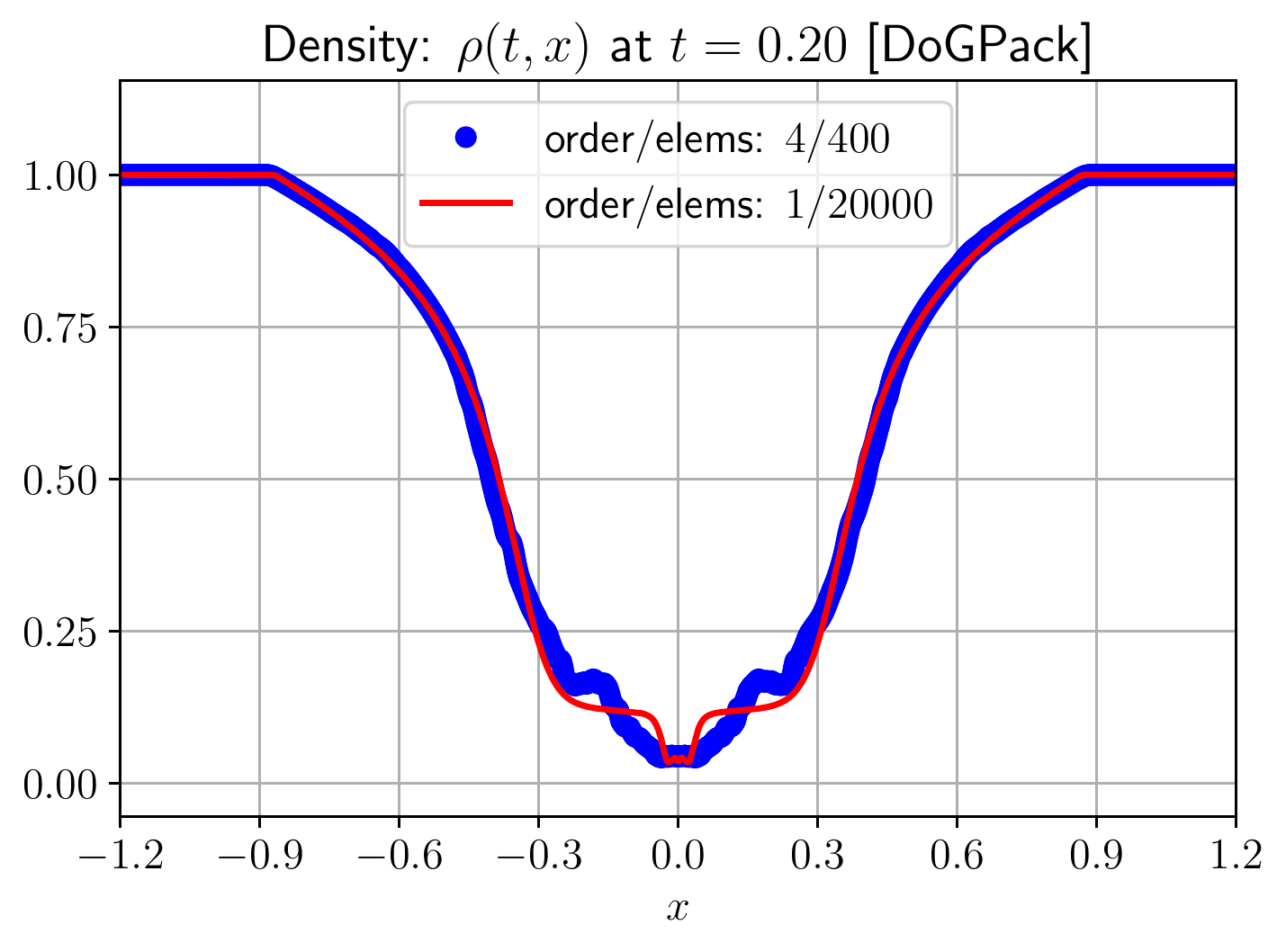} \\
(c)\includegraphics[width=.44\linewidth]{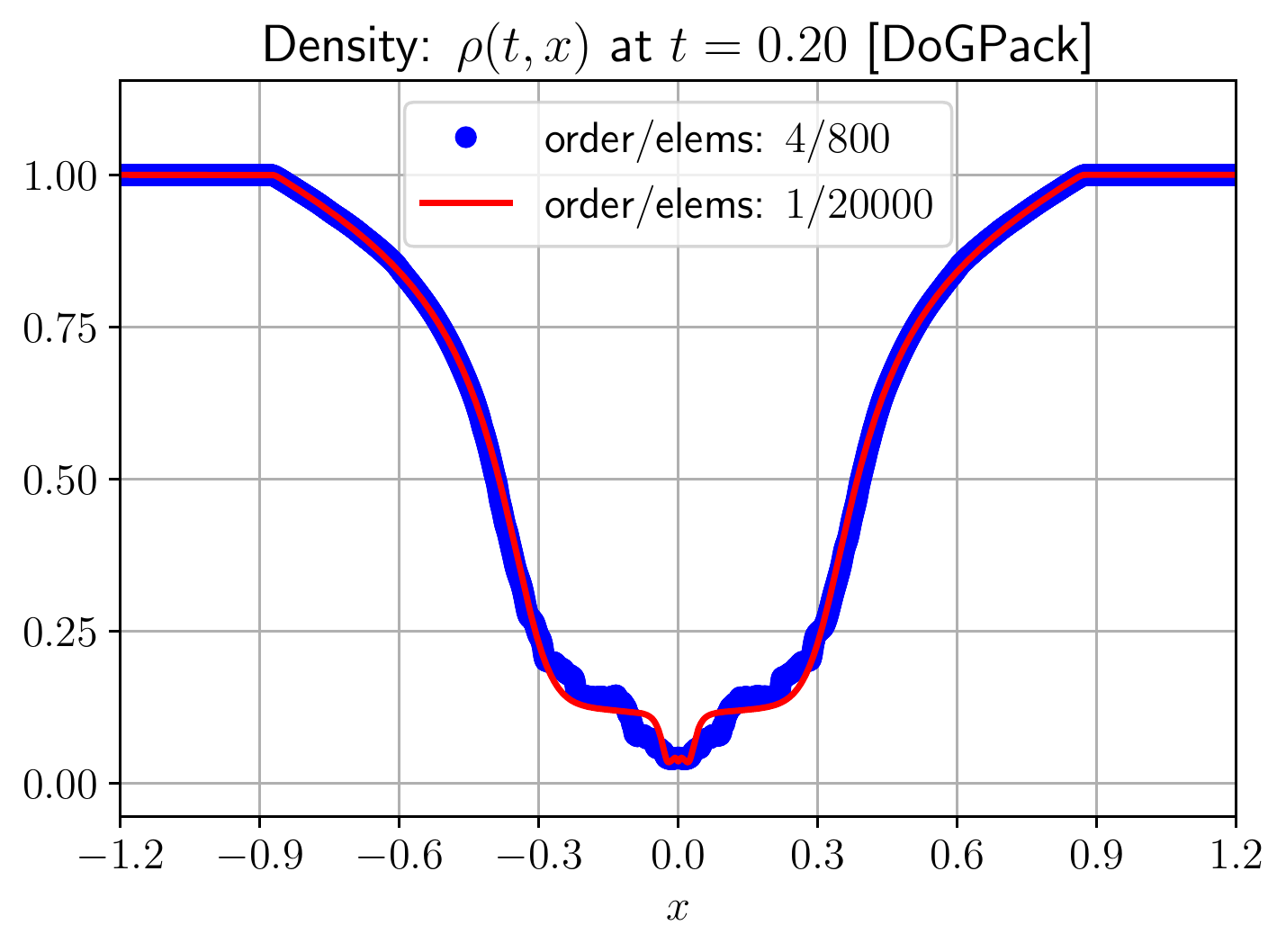} &
(d)\includegraphics[width=.44\linewidth]{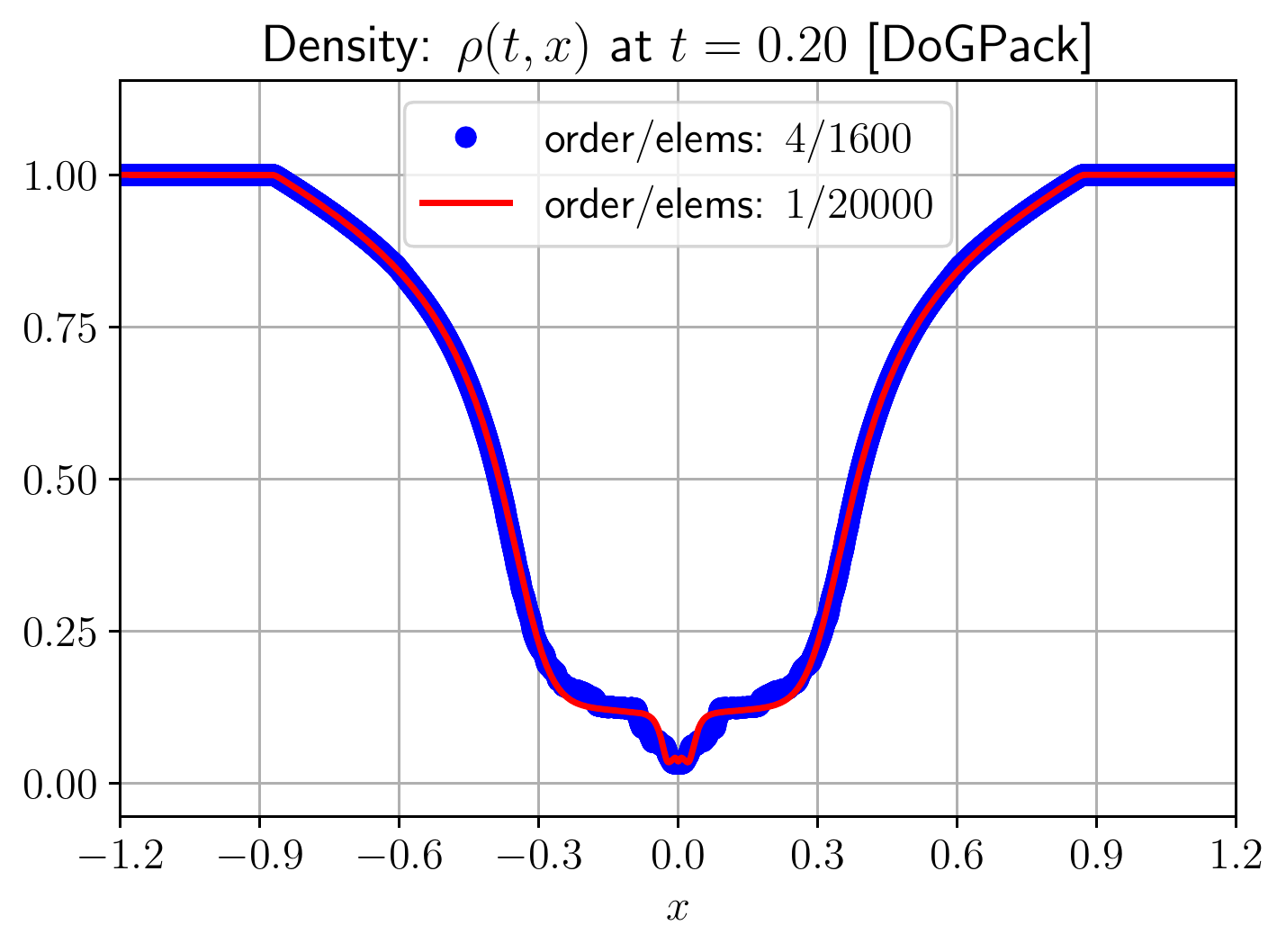} 
\end{tabular}
\caption{(\S\ref{sub:vacuum}: double rarefaction vacuum problem)
Numerical solution of the double rarefaction vacuum problem
on $x \in [-1.2, 1.2]$ with initial conditions given
by \eqref{eqn:vacuum_init}. 
Shown are the densities at various grid resolutions: (a) $N=200$,
(b) $N=400$, (c) $N=800$, and (d) $N=1600$.
In each panel, we compare the $\morder=4$ scheme with full limiters (shown as blue dots) 
with the first-order Rusanov  scheme
with 20,000 elements (shown as a solid red line).}
\label{fig:vacuum_converge}
\end{figure}

%% file: numerics_bgk.tex
Up to this point, we have only considered the HyQMOM approximation applied to the Vlasov model \cref{eqn:Boltzmann1d}; this allowed us to study the mathematical structure of HyQMOM and
to develop accurate high-order methods and limiters. On the other hand, the practicality of HyQMOM is not for solving
collisionless kinetic models since, in this regime, it would be far better to directly solve the Vlasov equation with Lagrangian or semi-Lagrangian approaches. Instead, the true benefit of the HyQMOM approximation is in the approximation of kinetic systems near thermodynamic equilibrium -- a regime we study in this section.

In this section, we extend the previously developed numerical method to HyQMOM with a BGK collision operator. Importantly, we develop this extension so that the resulting HyQMOM solver adheres to the following two key design parameters:
\begin{enumerate}
\item The method should remain high-order accurate irrespective of the Knudsen number: $\varepsilon>0$.
\item For fixed mesh parameters (i.e., fixed $\Delta x$ and $\Delta t$), the method should remain stable in the singular
limit: $\varepsilon \rightarrow 0^+$. This property is often referred to as the {\it asymptotic-preserving} (AP) property, and a variety of schemes with this property can be found in the literature (e.g.,
see \cite{article:Ben08,article:Caflish1997,article:Coron1991,article:Gabetta1997,article:Jin1995,article:Jin99,article:Jin12,article:Jin1996,article:Pieraccini2007,article:Xiong2017}).
\end{enumerate}
The specific approach we detail in this section is novel and directly relies on the prediction and correction format of the 
method developed in \cref{sec:numerics}.

\subsection{1D1V Boltzmann-BGK equation}
Consider the 1D1V Boltzmann-BGK equation \cite{article:BGK54}:
\begin{equation}
\label{eqn:boltzmann-bgk}
f_{,t} + v f_{,x} = \frac{1}{\varepsilon} \left( {\mathcal M} - f \right),
\end{equation}
where $\varepsilon>0$ is the Knudsen number, which is a non-dimensional ratio
of the particle mean-free path to a characteristic length scale,
and ${\mathcal M}(t,x,v):\reals_{\ge 0} \times \reals \times \reals \mapsto \reals_{\ge 0}$
is the Maxwell-Boltzmann distribution:
\begin{equation}
{\mathcal M}(t,x,v) := \frac{\rho}{\sqrt{2\pi T}} e^{-\frac{(v-u)^2}{2T}}.
\end{equation}
In this expression, $\rho$ is density, $p$ is pressure, and $T=p/\rho$ is temperature (e.g., see 
definitions \eqref{eqn:prim_vars_1}). For $\varepsilon \gg 1$, and for a fixed $t$ and $x$,
the collision operator is weak, and the solution behaves similarly to  Vlasov equation \eqref{eqn:Boltzmann1d}. For $\varepsilon \ll 1$, and for a fixed $t$ and $x$, the BGK collision operator forces $f$ towards the Maxwell-Boltzmann distribution (i.e., thermodynamic equilibrium):
\begin{equation}
f(t,x,v) \rightarrow {\mathcal M}(t,x,v) + {\mathcal O}(\varepsilon).
\end{equation}

\subsection{HyQMOM-BGK and the asymptotic-preserving property}
Relevant in this work are the first five moments of \eqref{eqn:boltzmann-bgk} with the
HyQMOM moment-closure \eqref{eqn:qmom_full_eqns}:
\begin{align}
\label{eqn:HyQMOM_BGK}
\vec{q}_{,t}+\vec{f}\left(\vec{q}\right)_{,x} = \frac{1}{\varepsilon} \, \vec{S^{\text{cons}}}\left(\vec{q}\right), \qquad
\vec{\alpha}_{,t}+\mat{B}\left(\vec{\alpha}\right) \, 
\vec{\alpha}_{,x} =  \frac{1}{\varepsilon} \, \vec{S^{\text{prim}}}\left(\vec{\alpha}\right),
\end{align}
where only the fourth and fifth components of the source terms are nonzero:
\begin{equation}
\label{eqn:HyQMOM_BGK_source}
S^{\text{cons}}_4  = S^{\text{prim}}_4 = -\heat, \quad S^{\text{cons}}_5 = -k  + \frac{2 p^2}{\rho} - \frac{4 p  u \heat + \heat^2}{p}, \quad  S^{\text{prim}}_5 = -k  + \frac{2 p^2}{\rho} + \frac{\heat^2}{p}.
\end{equation}
In \cref{eqn:HyQMOM_BGK} we are using definitions \eqref{eqn:qmom_full_eqns}, 
\eqref{eqn:HyQMOM_sys2}, and \eqref{eqn:HyQMOM_prim}. For $\varepsilon \ll 1$, and for a fixed $t$ and $x$, the BGK collision operator forces the heat flux, $\heat$, and modified kurtosis, $k$, towards their Maxwell-Boltzmann values:
\begin{equation}
\label{eqn:heat_kurt_equilibrium}
\heat(t,x) = 0 + {\mathcal O}(\varepsilon) \qquad \text{and} \qquad
k(t,x) = \frac{2 p(t,x)^2}{\rho(t,x)} + {\mathcal O}(\varepsilon).
\end{equation}

In particular, in the $\varepsilon \rightarrow 0^+$ limit, solutions of the HyQMOM-BGK system converge  to solutions 
of the 1D compressible Euler equations at a convergence
rate of ${\mathcal O}\left(\varepsilon\right)$:
\begin{equation}
\label{eqn:euler1d}
\begin{bmatrix}
\rho \\ \rho u \\ \rho u^2 + p 
\end{bmatrix}_{,t}
+ 
\begin{bmatrix}
\rho u \\ \rho u^2 + p \\ \rho u^3 + 3pu 
\end{bmatrix}_{,x} = \vec{0},
\end{equation}
where $\heat \equiv 0$ and $k \equiv \frac{2p}{\rho}$. Furthermore,
by including the next order term in the Chapman-Enskog expansion, one can show that solutions to HyQMOM-BGK converge
 to solutions of the 1D Navier-Stokes 
 equations at a convergence rate of
${\mathcal O}(\varepsilon^2)$  \cite{article:Bardos1991}:
  \begin{equation}
  \label{eqn:navierstokes1d}
\begin{bmatrix}
\rho \\ \rho u \\ \rho u^2 + p 
\end{bmatrix}_{,t}
+ 
\begin{bmatrix}
\rho u \\ \rho u^2 + p \\ \rho u^3 + 3pu 
\end{bmatrix}_{,x} = \begin{bmatrix}
0 \\ 0 \\  \frac{3}{2} \varepsilon p T_{,x}
\end{bmatrix}_{,x},
\end{equation}
where $T=\frac{p}{\rho}$ and 
$q \equiv -\frac{3}{2} \varepsilon p T_{,x}$. 

\begin{definition}[Asymptotic-preserving (AP) property \cite{article:Jin12}]
Let $\vec{q^{(\Delta t, \Delta x)}}(t,x; \varepsilon)$ be an approximation to the exact solution of 
\cref{eqn:HyQMOM_BGK}--\cref{eqn:HyQMOM_BGK_source} as
computed by a numerical method with mesh parameters $\Delta t, \Delta x > 0$. We assume that for a fixed $\varepsilon>0$, this
method is convergent to the exact solution of \cref{eqn:HyQMOM_BGK}--\cref{eqn:HyQMOM_BGK_source}.
This numerical method is said to be {\it asymptotic-preserving (AP)} provided that the vanishing mesh parameter limit, $\Delta t, \Delta x \rightarrow 0^+$, and the vanishing
Knudsen number limit, $\varepsilon \rightarrow 0^+$, commute:
\[
	\lim_{\varepsilon \rightarrow 0^+} \left[ \lim_{\Delta t, \Delta x \rightarrow 0^+} \, \vec{q^{(\Delta t, \Delta x)}}(t,x; \varepsilon)\right] =
	\lim_{\Delta t, \Delta x \rightarrow 0^+} \left[\lim_{\varepsilon \rightarrow 0^+}  \, \vec{q^{(\Delta t, \Delta x)}}(t,x; \varepsilon) \right].
\]
\end{definition}
Practically, this means that an AP scheme remains stable and accurate for a fixed mesh, $\Delta t$, and $\Delta x$,
for all $\varepsilon>0$, including in the limit $\varepsilon \rightarrow 0^+$.

\bigskip

The goal of this section is to develop an extension of the Lax-Wendroff DG scheme
developed in \cref{sec:numerics} and \cref{sec:limiters} for the HyQMOM-BGK system
\cref{eqn:HyQMOM_BGK}--\cref{eqn:HyQMOM_BGK_source} that  behaves, on the discrete
level, as a consistent and stable numerical method for  \cref{eqn:euler1d} and \cref{eqn:navierstokes1d} in the singular limit
$\varepsilon \rightarrow 0^+$.
The key innovation in this work is that we make use of the prediction-correction formulation of Lax-Wendroff DG
to incorporate the collision operator.

\subsection{Prediction}
\label{sec:prediction_bgk}
To describe the HyQMOM-BGK prediction step, it is first useful 
to define the following matrices that allow us to map Legendre coefficients to nodal space-time Gauss-Legendre quadrature points
and back again:
\begin{alignat}{2}
\label{eqn:map_to_spacetime_nodes}
  C^1_{(a,b)}&:=\Psi_b \left( \tau_a, \, \xi_a \right), \qquad
  &&\mat{C^1} \in \reals^{M_{\text{O}}^2 \times \mpred}, \\
  \label{eqn:map_to_spacetime_legendre}
    C^2_{(b,a)}&:=\frac{\omega_a}{4}  \, \Psi_b \left( \tau_a, \, \xi_a \right), \qquad
  &&\mat{C^2} \in \reals^{\mpred \times M_{\text{O}}^2},
\end{alignat}
where $\omega_a$ and $\left( \tau_a, \, \xi_a \right)$ for $a=1,\ldots,M_{\text{O}}^2$
are tensor product Gauss-Legendre weights and abscissas. These
two matrices satisfy
\begin{equation}
\mat{C^2} \, \mat{C^1} = \mat{ \, \, {\mathbb I} \, \, } \in \reals^{\mpred \times \mpred}.
\end{equation}


HyQMOM consists of five evolution equations, and the first three are unaffected by the collision operator; therefore, the update inside the Picard iteration for the three collision invariants (i.e., density, macroscopic velocity, and pressure) remains the same as in the collisionless case:
\cref{eqn:picard}.
On the other hand, the update for the heat flux, $\heat$, has a non-zero BGK contribution; however, the BGK term is linear in the heat flux
(see \cref{eqn:HyQMOM_BGK,eqn:HyQMOM_BGK_source}), which allows for simple treatment. The strategy we pursue here is to include the BGK source term in the implicit portion of the Picard update to remain uniformly stable in $\varepsilon>0$. After simple algebra, we arrive at the following update for the heat flux, $\heat$:
\begin{align}
\label{eqn:bgk_pred_heat_update}
\begin{split} 
\left( \frac{\Delta t}{2} \, \mat{I} + \varepsilon \, \mat{L} \, \right) \, \vec{W^{n+\half (j)}_{i \, (:,4)}} = \, &
 \frac{\varepsilon}{4} \sum_{a=1}^{\morder}
\sum_{b=1}^{\morder} \, \omega_a  \, \omega_b \, 
  \vec{{\Psi}}\left(\mu_b, \, \mu_a \right)
   \, {\Theta}_{(a,b,4)} \\
+ & \frac{\varepsilon}{4} \sum_{b=1}^{\morder} \omega_{b} \, \vec{{\Psi}} \left(-1,\xi_b \right)  \, \vec{\Phi} \left(\xi_b \right)^T \vec{A^{n}_{i \, (:,4)}},
\end{split}
\end{align}
where we are using the short-hand:
\begin{equation}
{\Theta}_{(a,b,m)} := {\Theta}_m\left( \vec{\Psi}\left(\mu_b, \mu_a \right)^T \mat{W^{n+\half (j-1)}_i} \right).
\end{equation} 

The update for the final primitive variable, $k$ (modified kurtosis),
requires more work. The source term shown in \cref{eqn:HyQMOM_BGK,eqn:HyQMOM_BGK_source} is linear in $k$, but it also includes nonlinear terms from three previously updated quantities: $\rho$, $p$, and $h$. To construct these nonlinear quantities, we first apply the mapping from Legendre to nodal values via
\cref{eqn:map_to_spacetime_nodes}, and then evaluate the nonlinear portion of the source:
\begin{equation}
\label{eqn:bgk_nonlinar_pred_source}
\begin{aligned}
  m=1,3,4: \quad &\vec{\widehat{W}_{(:,m)}} =  \mat{C^1} \, \vec{W_{i(:,m)}^{n+\half (j)}} 
  \in \reals^{M_{\text{O}}^2},
   \\
  a=1,\ldots,M_{\text{O}}^2: \quad &\left\{ \rho_a, p_a, \heat_a \right\} = \widehat{W}_{(a,\{1,3,4\})}, \quad
  \widehat{\mathcal S}^{ \, (j)}_{i a} = \frac{2 p_a^2}{\rho_a} + \frac{h_a^2}{p_a}.
  \end{aligned}
  \end{equation}
From here, the update for the modified kurtosis inside the Picard iteration looks very similar to the update for heat flux (see \cref{eqn:bgk_pred_heat_update}), but with the additional nonlinear terms computed from \cref{eqn:bgk_nonlinar_pred_source}, which now need to be mapped back to Legendre coefficients via \cref{eqn:map_to_spacetime_legendre}.
 After some simple algebra, the update takes the following form:
 \begin{equation}
 \label{eqn:bgk_pred_kurt_update}
\begin{split}
\left( \, \frac{\Delta t}{2} \mat{I} + \varepsilon \, \mat{L} \, \right) \, \vec{W^{n+\half (j)}_{i \, (:,5)}} = \, &  
 \frac{\varepsilon}{4} \sum_{a=1}^{\morder}
\sum_{b=1}^{\morder} \, \omega_a  \, \omega_b \, 
  \vec{{\Psi}}\left(\mu_b, \, \mu_a \right)
   \, {\Theta}_{(a,b,5)}  \\ 
   + &
 \frac{\varepsilon}{4} \sum_{b=1}^{\morder} \omega_{b} \, \vec{{\Psi}} \left(-1,\xi_b \right)  \, \vec{\Phi} \left(\xi_b \right)^T \vec{A^{n}_{i \, (:,5)}} + \frac{\Delta t}{2} \, \mat{C^2} \, \, \vec{\widehat{\mathcal S}^{(j)}_i}.
\end{split}
\end{equation}

\subsection{Post-prediction BGK source evaluation}
\label{sec:post_prediction_bgk}
Once the Picard iterations are complete and all five primitive variables have been predicted, there is one final computation that must be completed to prepare us for the correction step: we need to evaluate and project the BGK source term, $\vec{S^{\text{cons}}}$, from \cref{eqn:HyQMOM_BGK_source}. This is done similar to \cref{eqn:bgk_nonlinar_pred_source} by first mapping the predicted solution from Legendre to nodal values via
\cref{eqn:map_to_spacetime_nodes}, 
then evaluating the source components at nodal values, and finally mapping back to Legendre coefficients via \cref{eqn:map_to_spacetime_legendre}:
  \begin{equation}
  \label{eqn:pre_correction_source}
  \begin{aligned}
    m=1,\ldots,5: \quad &\vec{\widehat{W}_{(:,m)}} =  \mat{C^1} \, \vec{W_{i(:,m)}^{n+\half}} 
  \in \reals^{M_{\text{O}}^2}, \\
  a=1,\ldots,M_{\text{O}}^2: \quad  &\left\{ \rho_a, \, u_a, \, p_a, \, \heat_a, \, k_a \right\}
  = \widehat{W}_{(a,1:5)}, \\ 
    &\widehat{\Delta{\mathcal M}}_{(a,4)} = -\heat_a, \, \, \, \, \, 
  \widehat{\Delta {\mathcal M}}_{(a,5)} = -k_a + \frac{2 p_a^2}{\rho_a} - \frac{4  p_a u_a \heat_a + \heat_a^2}{p_a},  \\
      m=4,5: \quad &\vec{\Delta {\mathcal M}_{i(:,m)}} = \mat{C^2} \, \vec{\widehat{\Delta {\mathcal M}}_{(:,m)}}  \in \reals^{\mpred}.
\end{aligned}
 \end{equation}
In the above expressions, we use the notation $\Delta {\mathcal M}$ to signify that these BGK source terms are, in fact, measuring the deviations in the heat flux, $h$, and the modified kurtosis, $k$, from their Maxwell-Boltzmann values (e.g., see \eqref{eqn:heat_kurt_equilibrium}).

We choose to do the above BGK source evaluation and projection, 
\cref{eqn:pre_correction_source}, as a separate step rather than just as part of the correction update since we need to be extra careful in assuring that the final update is asymptotic-preserving. Indeed, we show in the next section how to obtain a fully asymptotic-preserving scheme.

\subsection{Correction}
\label{sec:correction_bgk}
Just as in the prediction step, we begin by defining matrices that allow us to map Legendre coefficients to nodal space Gauss-Legendre quadrature points and back again:
\begin{alignat}{2}
\label{eqn:map_to_space_nodes}
   C^3_{(a,b)}&:=\Phi_b \left( \xi_a \right), \qquad
  &&\mat{C^3} \in \reals^{M_{\text{O}} \times \mcorr}, \\
\label{eqn:map_to_space_legendre}
      C^4_{(b,a)}&:=\frac{\omega_a}{2}  \, \Phi_b \left( \xi_a \right), \qquad
  &&\mat{C^4} \in \reals^{\mcorr \times M_{\text{O}}},
\end{alignat}
where $\omega_a$ and $\xi_a$ for $a=1,\ldots,M_{\text{O}}$
are Gauss-Legendre weights and abscissas. These
two matrices satisfy
\begin{equation}
\mat{C^4} \, \mat{C^3} = \mat{ \, \, {\mathbb I} \, \, } \in \reals^{\mcorr \times \mcorr}.
\end{equation}

As far as the correction step is concerned, the only difference between the collisionless update, as shown through
\cref{eqn:corr_step_update_1,eqn:corr_step_update_2,eqn:corr_step_update_3,eqn:corr_step_update_4} and the BGK version is the
additional BGK source integral needed in \cref{eqn:corr_step_update_1}:
\begin{equation}
 \label{eqn:corr_step_update_BGK}
 \begin{split}
        \mat{Q_i^{n+1}} &= 
 \mat{Q^{n}_i} \, + \,  \ldots  \,
+ \, \underset{\text{BGK source}}{\underbrace{\frac{\frac{\Delta t}{2}}{2\varepsilon} \iint_{-1}^{1} \vec{\Phi} \, \left[ \vec{S^{\text{cons}}} \right]^T \, d\tau \, d\xi}}.
\end{split}
    \end{equation}
To eventually achieve the asymptotic-preserving (AP) property, we introduce the following Legendre-in-space-Radau-in-time quadrature:
\begin{equation}
\label{eqn:Legendre-in-space-Radau-in-time}
\iint_{-1}^{1} g\left(\tau, \xi \right) \, d\tau \, d\xi \approx
\sum_{k=1}^{M_{\text{O}}} \sum_{\ell=1}^{M_{\text{O}}} \omega_k \, 
 \omega^{R}_{\ell} \,
 g\left(\tau^R_{\ell}, \xi_{k} \right),
\end{equation}
where for $a=1,\ldots,{M_{\text{O}}}$, $(\omega_a,\xi_a)$ are again 1D Gauss-Legendre weights/abscissas, while $(\omega^{R}_{a}, \xi^{R}_{a})$
are 1D Gauss-Radau weights/abscissas. In particular, what we aim to do here is to handle the
BGK source in \cref{eqn:corr_step_update_BGK} using a strategy that replaces the actual Legendre-in-space-Radau-in-time quadrature
shown via \cref{eqn:Legendre-in-space-Radau-in-time}, by a version where the function values at the $\tau=1$ quadrature points are replaced
by the unknown solution $Q^{n+1}$:
\begin{align}
\label{eqn:implicit_gauss_radau}
m=4,5: \quad \frac{1}{2\varepsilon} \iint_{-1}^{1} \vec{\Phi} \, S^{\text{cons}}_{m} \, d\tau \, d\xi \approx 
\underset{\text{explicit}}{\underbrace{\frac{1}{\varepsilon} \, \mat{R} \, \, \vec{\Delta {\mathcal M}_{i(:,m)}}}} + 
\underset{\text{implicit}}{\underbrace{\frac{1}{r \varepsilon} \left( \vec{{S}_{(:,m)}} - \vec{Q^{n+1}_{i(:,m)}}  \right)}},
\end{align}
where $\Delta {\mathcal M}$ is defined by \cref{eqn:pre_correction_source}, ${S}$ are Maxwell-Boltzmann moments (the precise definition is provided below in \cref{eqn:maxwell_source}),
and 
\begin{gather}
 \mat{R} = \frac{1}{2}\sum_{k=1}^{M_{\text{O}}} \sum_{\ell=1}^{M_{\text{O}}-1} \omega_k \, 
 \omega^{R}_{\ell} \,
 \vec{\Phi}\left(\xi_{k} \right) \vec{\Psi}\left(\tau^R_{\ell}, \xi_{k} \right)^T \in \reals^{\mcorr \times \mpred}, \quad
  r = 1\bigl/\omega^{R}_{M_{\text{O}}}.
\end{gather}
This quadrature provides a strategy for implicitly handling the BGK collision term, which is critical
for achieving the asymptotic-preserving (AP) property.  We illustrate the modified 
Gauss-Radau quadrature strategy in \cref{fig:radau_quad}.

The full correction update is detailed below. The first three moments are collision invariants and thus updated 
via \cref{eqn:corr_step_update_1,eqn:corr_step_update_2,eqn:corr_step_update_3,eqn:corr_step_update_4}. 
From these updated moments, we compute the Maxwell-Boltzmann moments, $S$, that are required in \cref{eqn:implicit_gauss_radau}:
\begin{equation}
\label{eqn:maxwell_source}
\begin{aligned}
m=1,2,3: \quad &\vec{\widehat{M}_{(:,m)}} =  \mat{C^3} \, \, \vec{Q_{i(:,m)}^{n+1}} 
  \in \reals^{M_{\text{O}}}, \\
  a=1,\ldots,M_{\text{O}}: \quad &\left\{ \rho, \, u, \, p \right\} = 
  \left\{ \widehat{M}_{(a,1)}, \, \frac{\widehat{M}_{(a,2)}}{\widehat{M}_{(a,1)}}, \, 
   {\widehat{M}_{(a,3)}} - \frac{\widehat{M}_{(a,2)}^2}{\widehat{M}_{(a,1)}} \right\}, \\
  &\left\{ \widehat{S}_{i(a,4)}, \, \widehat{S}_{i(a,5)} \right\} =
  \left\{ \rho u^3 + 3 p u, \, \, \rho u^4 + 6 p u^2 + \frac{3 p^2}{\rho} \right\}.
  \end{aligned}
  \end{equation}
We then update the final two moments in a two-step process, where the first step is to apply a collisionless
update:
   \begin{gather}
    \begin{split}
    m=4,5: \quad {\vec{\widetilde{Q}_{i(:,m)}^{n+1}}} &= \vec{Q_{i(:,m)}^{n}} 
-  {\frac{\Delta t}{\Delta x}}\left( \, \vec{\Phi}(1) \, {{\mathcal F}^{n+\half}_{i+\half (m)}} 
-  \vec{\Phi}(-1) \, {{\mathcal F}^{n+\half}_{i-\half \, (m)}}  \, \right) \\
&+  \frac{\Delta t}{2 \Delta x}  \sum_{a=1}^{\morder}
\sum_{b=1}^{\morder} \, \omega_a  \, \omega_b \, \vec{\Phi}_{,\xi}\left(\mu_a \right) 
 f_{m} \left( \mat{W^{n+\half}_{i}} \, \vec{\Psi}\left(\mu_b,\mu_a\right) \right),
\end{split}
\end{gather}
followed by a collision step:
\begin{equation}
\label{eqn:collision_correct}
        \vec{Q_{i(:,4:5)}^{n+1}} = 
        \left( \frac{r \varepsilon}{\frac{\Delta t}{2} + r \varepsilon} \right)
        {\vec{\widetilde{Q}_{i(:,4:5)}^{n+1}}} +    
 \left( \frac{\frac{\Delta t}{2}}{\frac{\Delta t}{2} + r \varepsilon} \right)
\Biggl( r \mat{R} \, \vec{\Delta {\mathcal M}_{i(:,4:5)}} +  \mat{C^4} \, \, \vec{\widehat{S}_{i(:,4:5)}}  \Biggr). 
\end{equation}

\begin{note}
All the limiters described in \cref{sec:limiters} can still be applied to the HyQMOM-BGK 
solver described in this section.
\end{note}

\begin{figure}[!t]
\centering
\begin{tabular}{cc}
(a)  \begin{tikzpicture}[scale=0.85]
\draw (-2.5,-2.5) rectangle (2.5,2.5);
\node at (-2.1528407789851314380598662222320,-2.0570602024364802630) [draw,circle,fill=blue,inner sep=2pt]{};
\node at (-0.8499526089621406620066643977581,-2.0570602024364802630) [draw,circle,fill=blue,inner sep=2pt]{};
\node at (0.8499526089621406620066643977581,-2.0570602024364802630) [draw,circle,fill=blue,inner sep=2pt]{};
\node at (2.1528407789851314380598662222320,-2.0570602024364802630) [draw,circle,fill=blue,inner sep=2pt]{};
\node [anchor=south] at (-1.8,-1.95) [fill=white]{\footnotesize $\widehat{W}^{n+\frac{1}{2}}_1$};
\node [anchor=south] at (-0.48,-1.95) [fill=white]{\footnotesize $\widehat{W}^{n+\frac{1}{2}}_2$};
\node [anchor=south] at (1.2,-1.95) [fill=white]{\footnotesize $\widehat{W}^{n+\frac{1}{2}}_3$};
\node [anchor=south] at (2.5,-1.95) [fill=white]{\footnotesize $\widehat{W}^{n+\frac{1}{2}}_4$};
\node at (-2.1528407789851314380598662222320,-0.45266567779632644568) [draw,circle,fill=blue,inner sep=2pt]{};
\node at (-0.8499526089621406620066643977581,-0.45266567779632644568) [draw,circle,fill=blue,inner sep=2pt]{};
\node at (0.8499526089621406620066643977581,-0.45266567779632644568) [draw,circle,fill=blue,inner sep=2pt]{};
\node at (2.1528407789851314380598662222320,-0.45266567779632644568) [draw,circle,fill=blue,inner sep=2pt]{};
\node [anchor=south] at (-1.8,-0.34) [fill=white]{\footnotesize $\widehat{W}^{n+\frac{1}{2}}_5$};
\node [anchor=south] at (-0.48,-0.34) [fill=white]{\footnotesize $\widehat{W}^{n+\frac{1}{2}}_6$};
\node [anchor=south] at (1.2,-0.34) [fill=white]{\footnotesize $\widehat{W}^{n+\frac{1}{2}}_7$};
\node [anchor=south] at (2.5,-0.34) [fill=white]{\footnotesize $\widehat{W}^{n+\frac{1}{2}}_8$};
\node at (-2.1528407789851314380598662222320,1.4382973088042352801) [draw,circle,fill=blue,inner sep=2pt]{};
\node at (-0.8499526089621406620066643977581,1.4382973088042352801) [draw,circle,fill=blue,inner sep=2pt]{};
\node at (0.8499526089621406620066643977581,1.4382973088042352801) [draw,circle,fill=blue,inner sep=2pt]{};
\node at (2.1528407789851314380598662222320,1.4382973088042352801) [draw,circle,fill=blue,inner sep=2pt]{};
\node [anchor=south] at (-1.8,1.545) [fill=white]{\footnotesize $\widehat{W}^{n+\frac{1}{2}}_9$};
\node [anchor=south] at (-0.48,1.545) [fill=white]{\footnotesize $\widehat{W}^{n+\frac{1}{2}}_{10}$};
\node [anchor=south] at (1.2,1.545) [fill=white]{\footnotesize $\widehat{W}^{n+\frac{1}{2}}_{11}$};
\node [anchor=south] at (2.5,1.545) [fill=white]{\footnotesize $\widehat{W}^{n+\frac{1}{2}}_{12}$};
\node at (-2.1528407789851314380598662222320,2.5) [draw,circle,fill=blue,inner sep=2pt]{};
\node at (-0.8499526089621406620066643977581,2.5) [draw,circle,fill=blue,inner sep=2pt]{};
\node at (0.8499526089621406620066643977581,2.5) [draw,circle,fill=blue,inner sep=2pt]{};
\node at (2.1528407789851314380598662222320,2.5) [draw,circle,fill=blue,inner sep=2pt]{};
\node [anchor=south] at (-1.8,2.62) [fill=white]{\footnotesize $\widehat{W}^{n+\frac{1}{2}}_{13}$};
\node [anchor=south] at (-0.48,2.62) [fill=white]{\footnotesize $\widehat{W}^{n+\frac{1}{2}}_{14}$};
\node [anchor=south] at (1.2,2.62) [fill=white]{\footnotesize $\widehat{W}^{n+\frac{1}{2}}_{15}$};
\node [anchor=south] at (2.5,2.62) [fill=white]{\footnotesize $\widehat{W}^{n+\frac{1}{2}}_{16}$};
\end{tikzpicture} & \hspace{-3mm}
(b)  \begin{tikzpicture}[scale=0.85]
\draw (-2.5,-2.5) rectangle (2.5,2.5);
\node at (-2.1528407789851314380598662222320,-2.0570602024364802630) [draw,circle,fill=blue,inner sep=2pt]{};
\node at (-0.8499526089621406620066643977581,-2.0570602024364802630) [draw,circle,fill=blue,inner sep=2pt]{};
\node at (0.8499526089621406620066643977581,-2.0570602024364802630) [draw,circle,fill=blue,inner sep=2pt]{};
\node at (2.1528407789851314380598662222320,-2.0570602024364802630) [draw,circle,fill=blue,inner sep=2pt]{};
\node [anchor=south] at (-1.8,-1.95) [fill=white]{\footnotesize $\widehat{W}^{n+\frac{1}{2}}_1$};
\node [anchor=south] at (-0.48,-1.95) [fill=white]{\footnotesize $\widehat{W}^{n+\frac{1}{2}}_2$};
\node [anchor=south] at (1.2,-1.95) [fill=white]{\footnotesize $\widehat{W}^{n+\frac{1}{2}}_3$};
\node [anchor=south] at (2.5,-1.95) [fill=white]{\footnotesize $\widehat{W}^{n+\frac{1}{2}}_4$};
\node at (-2.1528407789851314380598662222320,-0.45266567779632644568) [draw,circle,fill=blue,inner sep=2pt]{};
\node at (-0.8499526089621406620066643977581,-0.45266567779632644568) [draw,circle,fill=blue,inner sep=2pt]{};
\node at (0.8499526089621406620066643977581,-0.45266567779632644568) [draw,circle,fill=blue,inner sep=2pt]{};
\node at (2.1528407789851314380598662222320,-0.45266567779632644568) [draw,circle,fill=blue,inner sep=2pt]{};
\node [anchor=south] at (-1.8,-0.34) [fill=white]{\footnotesize $\widehat{W}^{n+\frac{1}{2}}_5$};
\node [anchor=south] at (-0.48,-0.34) [fill=white]{\footnotesize $\widehat{W}^{n+\frac{1}{2}}_6$};
\node [anchor=south] at (1.2,-0.34) [fill=white]{\footnotesize $\widehat{W}^{n+\frac{1}{2}}_7$};
\node [anchor=south] at (2.5,-0.34) [fill=white]{\footnotesize $\widehat{W}^{n+\frac{1}{2}}_8$};
\node at (-2.1528407789851314380598662222320,1.4382973088042352801) [draw,circle,fill=blue,inner sep=2pt]{};
\node at (-0.8499526089621406620066643977581,1.4382973088042352801) [draw,circle,fill=blue,inner sep=2pt]{};
\node at (0.8499526089621406620066643977581,1.4382973088042352801) [draw,circle,fill=blue,inner sep=2pt]{};
\node at (2.1528407789851314380598662222320,1.4382973088042352801) [draw,circle,fill=blue,inner sep=2pt]{};
\node [anchor=south] at (-1.8,1.545) [fill=white]{\footnotesize $\widehat{W}^{n+\frac{1}{2}}_9$};
\node [anchor=south] at (-0.48,1.545) [fill=white]{\footnotesize $\widehat{W}^{n+\frac{1}{2}}_{10}$};
\node [anchor=south] at (1.2,1.545) [fill=white]{\footnotesize $\widehat{W}^{n+\frac{1}{2}}_{11}$};
\node [anchor=south] at (2.5,1.545) [fill=white]{\footnotesize $\widehat{W}^{n+\frac{1}{2}}_{12}$};
\node [anchor=south] at (-1.9,2.62) [fill=white]{\footnotesize $Q^{n+1}_{1}$};
\node [anchor=south] at (-0.65,2.62) [fill=white]{\footnotesize $Q^{n+1}_{2}$};
\node [anchor=south] at (1.05,2.62) [fill=white]{\footnotesize $Q^{n+1}_{3}$};
\node [anchor=south] at (2.4,2.62) [fill=white]{\footnotesize $Q^{n+1}_{4}$};
\node at (-2.1528407789851314380598662222320,2.5) [draw,
    fill=red,
    minimum width=0.2cm,
    minimum height=0.2cm
]  (controller) {};

\node at (-0.8499526089621406620066643977581,2.5) [draw,
    fill=red,
    minimum width=0.2cm,
    minimum height=0.2cm
]  (controller) {};

\node at (0.8499526089621406620066643977581,2.5) [draw,
    fill=red,
    minimum width=0.2cm,
    minimum height=0.2cm
]  (controller) {};

\node at (2.1528407789851314380598662222320,2.5) [draw,
    fill=red,
    minimum width=0.2cm,
    minimum height=0.2cm
]  (controller) {};
\end{tikzpicture}
\end{tabular}
\caption{Numerical quadrature on the canonical space-time element $(\tau,\xi) \in [-1,1]^2$ using a tensor product between 1D
Gauss-Legendre points in $\xi$ and Gauss-Radau points in $\tau$. Panel (a) shows the Legendre-in-space-Radau-in-time
points in the case $\morder=4$ and the known solution values from the prediction step at those
quadrature points. Panel (b) shows the same thing, but the function values at the $\tau=1$ quadrature points are replaced
by the unknown solution $Q^{n+1}$. This strategy of replacing the $\tau=1$ quadrature point function values with the
unknown solution provides a strategy for implicitly handling the BGK collision term. \label{fig:radau_quad}}
\end{figure}
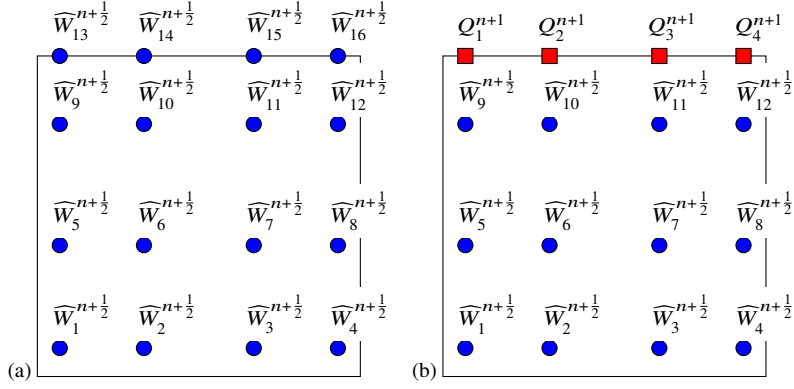

\subsection{Asymptotic-preserving condition}
The advantage of the above-proposed scheme for the HyQMOM-BGK is that it remains high-order accurate uniformly in $\varepsilon>0$ and is asymptotic-preserving in the $\varepsilon\rightarrow 0^+$ limit. The first claim is demonstrated via numerical examples in the next section; the second claim is easily demonstrated in this section.

\begin{lemma}
The method LxW-DG method for HyQMOM-BGK described in \cref{sec:prediction_bgk}, \cref{sec:post_prediction_bgk}, and \cref{sec:correction_bgk}
is asymptotic-preserving in the $\varepsilon\rightarrow 0^+$ limit.
\end{lemma}

\begin{proof}
In the prediction step, the only updates directly affected by the Knudsen number, $\varepsilon>0$, are the updates for the heat flux, $h$, and the modified kurtosis, $k$. Taking the $\varepsilon\rightarrow 0^+$ limit of both \cref{eqn:bgk_pred_heat_update} and \cref{eqn:bgk_pred_kurt_update} yields:
\begin{equation}
	\vec{W^{n+\half (j)}_{i \, (:,4)}} \rightarrow \vec{0} \qquad \text{and} \qquad
	\vec{W^{n+\half (j)}_{i \, (:,5)}} \rightarrow \mat{C^2} \, \, \vec{\widehat{\mathcal S}^{(j)}_i}.
\end{equation}
This is precisely the desired effect: moments converge to their Maxwell-Boltzmann values.

In the correction step, the only update directly affected by the Knudsen number, $\varepsilon>0$, is \cref{eqn:collision_correct}.
Taking the $\varepsilon\rightarrow 0^+$ limit of this update yields:
\begin{equation}
 \vec{\Delta {\mathcal M}_{i(:,4:5)}} \rightarrow \vec{0} \qquad \text{and} \qquad
 \vec{Q_{i(:,4:5)}^{n+1}} \rightarrow  \mat{C^4} \, \, \vec{\widehat{S}_{i(:,4:5)}}.
\end{equation}
Again, this is precisely the desired effect: moments converge to their Maxwell-Boltzmann values.
\qedsymbol
\end{proof}

%% file: examples_bgk.tex
In this section, we apply the proposed HyQMOM-BGK scheme to several test cases.
In \S\ref{sub:conv_test_bgk} we verify the claimed orders of accuracy on a smooth manufactured solution with different
Knudsen numbers. These tests also show the scheme's uniform accuracy and order of accuracy as a function of the Knudsen number. In \S\ref{sub:shock_tube_bgk} we apply the scheme to {\it shock tube} initial data with different Knudsen numbers. 
These results demonstrate the ability of the non-oscillatory limiter to control unphysical oscillations. Also shown by these results
is the asymptotic-preserving (AP) property of the scheme for small $\varepsilon>0$; in particular, we include the exact Riemann solution
for the compressible Euler equations as a point of comparison.

\subsection{Manufactured solution convergence test}
\label{sub:conv_test_bgk}
We consider the following manufactured solution:
\begin{equation}
\begin{split}
\vec{\alpha^{\text{ms}}}(t,x) &:= \left( \rho^{\text{ms}}, \, u^{\text{ms}}, \, p^{\text{ms}}, \, h^{\text{ms}}, \, k^{\text{ms}} \right)(t,x), \\
\left( \rho^{\text{ms}}, \, p^{\text{ms}}, \, h^{\text{ms}}, \, k^{\text{ms}} \right) \left(t,x\right) &:= \left(
\rho_{\ve}, \, p_{\ve}, \, h_{\ve}, \, k_{\ve} \right) \sqrt{\pi} \left(2 - \cos\left(2 \pi \left( t - x \right) \right) \right), \\
u^{\text{ms}}\left(t,x\right) &:= \frac{1-3\ve}{4+8\ve},
\end{split}
\end{equation}
where
\begin{equation}
\begin{split}
\rho_{\ve} = \frac{1+2\ve}{2+2\ve}, \quad
p_{\ve} = \frac{2 + 33  \left(\ve + \ve^2\right)}{32 (1 + \ve) (1 + 2 \ve)}, \quad
h_{\ve} = -\frac{125 \ve}{128 \left(1 + 2 \ve\right)^2}, \\
r_{\ve} = \frac{12 + \left(\ve + \ve^2\right) \left(1021 + 2017  \left(\ve + \ve^2\right)\right)}{512 (1+ \ve) (1 + 2\ve)^3}, \quad
k_{\ve} = r_{\ve} - \frac{p^2_{\ve}}{\rho_{\ve}} - \frac{h_{\ve}^2}{p_{\ve}}.
\end{split}
\end{equation}
Note that this solution is $\varepsilon$-dependent and well-defined for all $0<\varepsilon<\infty$.
Since this is not an exact solution to HyQMOM-BGK, we need to augment \cref{eqn:HyQMOM_BGK} with an additional manufactured solution
source term:
\begin{align}
\vec{q}_{,t}+\vec{f}\left(\vec{q}\right)_{,x} = \frac{1}{\varepsilon} \, \vec{S^{\text{cons}}}\left(\vec{q}\right)+ \vec{s_q}, \quad
\vec{\alpha}_{,t}+\mat{B}\left(\vec{\alpha}\right) \, 
\vec{\alpha}_{,x} =  \frac{1}{\varepsilon} \, \vec{S^{\text{prim}}}\left(\vec{\alpha}\right)+ \vec{q}^{-1}_{,\vec{\alpha}}\left(\vec{\alpha^{\text{ms}}}\right)
\vec{s_{q}},
\end{align}
where
\begin{equation}
\begin{gathered}
\vec{s_q} = \left[ \pi^{3/2} \sin\bigl(2 \pi (t - x)\bigr) \right] \vec{v_1} 
+ \left[ \pi^{1/2} \Bigl(2 - \cos\bigl(2 \pi (t - x)\bigr) \Bigr) \right] \vec{v_2}, \\
\vec{v_1} = \left( A_1, \, A_2, \, A_3, \, A_4, \, A_7 \right), \quad
\vec{v_2} = \left( 0, \, 0, \, 0, \,  A_5, \,  A_6 \right).
\end{gathered}
\end{equation}
with
\begin{equation}
\begin{gathered}
A_1 = \frac{3 + 11 \ve}{4(1 + \ve)}, \quad
A_2 = \frac{1 - 33 \ve}{16 (1 + \ve)}, \quad
A_3 = \frac{5 (1  + 33 \ve)}{64 (1 + \ve)}, \quad
A_4 = \frac{3 - 809 \ve}{256 (1 + \ve)}, \\
A_5 = \frac{-125}{128 (1 + 2 \ve)^2}, \quad
A_6 = \frac{125 (1 + 2 \ve - 10 \ve^2)}{512 (1 + 2 \ve)^3}, \\
A_7 = \frac{76 + 3620 \ve + 521895 \ve^2 + 5285445 \ve^3 + 9544425 \ve^4 + 4794867 \ve^5}{1024 (1 + \ve) (2 + 33 \ve (1 + \ve))^2}.
\end{gathered}
\end{equation}

Convergence tables for the $\morder=4$ scheme are shown in \cref{table:qmom3deltas_bgk_1d_error}. Importantly, we consider
various values of the Knudsen number that span ten orders of magnitude: $\varepsilon = 10^{4}, \, 10^{2}, \, 10^{0}, \,  10^{-2}, \, 10^{-4}, \, \text{and} \, 10^{-6}$, and in each case, we achieve optimal convergence. These results confirm the asymptotic-preserving (AP) property
for small $\varepsilon>0$.

\begin{table}
\begin{center}
\begin{Large}
\begin{tabular}{|c||c|c||c|c||c|c|}
\hline
{\normalsize $N$} & {\normalsize $\varepsilon = 10^4$} & {\normalsize $\log_2\frac{\text{e}_{N/2}}{\text{e}_{N}}$}  & {\normalsize $\varepsilon = 10^2$} & {\normalsize $\log_2\frac{\text{e}_{N/2}}{\text{e}_{N}}$}  & {\normalsize $\varepsilon = 10^0$} & {\normalsize $\log_2\frac{\text{e}_{N/2}}{\text{e}_{N}}$} \\
\hline\hline
{\normalsize 10} & {\normalsize 1.181e-03} & -- & {\normalsize 1.193e-03} & -- & {\normalsize 1.426e-03} & -- \\\hline
{\normalsize 20} & {\normalsize 5.809e-05} & {\normalsize $4.346$} & {\normalsize 5.897e-05} & {\normalsize $4.339$} & {\normalsize 6.321e-05} & {\normalsize $4.496$} \\\hline
{\normalsize 40} & {\normalsize 3.541e-06} & {\normalsize $4.036$} & {\normalsize 3.515e-06} & {\normalsize $4.068$} & {\normalsize 3.655e-06} & {\normalsize $4.112$} \\\hline
{\normalsize 80} & {\normalsize 2.212e-07} & {\normalsize $4.001$} & {\normalsize 2.622e-07} & {\normalsize $3.745$} & {\normalsize 2.601e-07} & {\normalsize $3.813$} \\\hline
{\normalsize 160} & {\normalsize 1.376e-08} & {\normalsize $4.007$} & {\normalsize 1.622e-08} & {\normalsize $4.015$} & {\normalsize 1.529e-08} & {\normalsize $4.088$} \\\hline
{\normalsize 320} & {\normalsize 8.592e-10} & {\normalsize $4.001$} & {\normalsize 9.983e-10} & {\normalsize $4.022$} & {\normalsize 8.986e-10} & {\normalsize $4.089$} \\\hline
\hline\hline\hline\hline
{\normalsize $N$} & {\normalsize $\varepsilon = 10^{-2}$} & {\normalsize $\log_2\frac{\text{e}_{N/2}}{\text{e}_{N}}$}  & {\normalsize $\varepsilon = 10^{-4}$} & {\normalsize $\log_2\frac{\text{e}_{N/2}}{\text{e}_{N}}$}  & {\normalsize $\varepsilon = 10^{-6}$} & {\normalsize $\log_2\frac{\text{e}_{N/2}}{\text{e}_{N}}$} \\
\hline\hline
{\normalsize 10} & {\normalsize 1.327e-03} & -- & {\normalsize 1.644e-03} & -- & {\normalsize 1.660e-03} & -- \\\hline
{\normalsize 20} & {\normalsize 6.608e-05} & {\normalsize $4.328$} & {\normalsize 6.826e-05} & {\normalsize $4.590$} & {\normalsize 6.861e-05} & {\normalsize $4.597$} \\\hline
{\normalsize 40} & {\normalsize 4.040e-06} & {\normalsize $4.032$} & {\normalsize 4.222e-06} & {\normalsize $4.015$} & {\normalsize 4.148e-06} & {\normalsize $4.048$} \\\hline
{\normalsize 80} & {\normalsize 2.537e-07} & {\normalsize $3.993$} & {\normalsize 2.575e-07} & {\normalsize $4.035$} & {\normalsize 2.557e-07} & {\normalsize $4.020$} \\\hline
{\normalsize 160} & {\normalsize 1.572e-08} & {\normalsize $4.012$} & {\normalsize 1.622e-08} & {\normalsize $3.989$} & {\normalsize 1.601e-08} & {\normalsize $3.998$} \\\hline
{\normalsize 320} & {\normalsize 9.929e-10} & {\normalsize $3.985$} & {\normalsize 1.006e-09} & {\normalsize $4.011$} & {\normalsize 1.008e-09} & {\normalsize $3.989$} \\\hline
\end{tabular} 
\caption{(\S\ref{sub:conv_test_bgk}: HyQMOM-BGK manufactured solution problem) Relative $L^2$ errors for a manufactured solution example for the one-dimensional HyQMOM equations with a BGK collision operator. 
The errors are computed for various values of the Knudsen number: $\varepsilon = 10^{4}, \, 10^{2}, \, 10^{0}, \,  10^{-2}, \, 10^{-4}, \, \text{and} \, 10^{-6}$. In each case, we use the fourth order
method: $\morder=4$.}
\label{table:qmom3deltas_bgk_1d_error}
\end{Large}
\end{center}
\end{table}

\subsection{BGK shock tube problem}
\label{sub:shock_tube_bgk}
Consider the Riemann problem for  \eqref{eqn:HyQMOM_BGK}--\eqref{eqn:HyQMOM_BGK_source} with the following initial data at $t=0$:
\begin{equation}
\label{eqn:bgk_shock_init}
\bigl(\rho, u, p, \heat, k \bigr)(t=0,x) = 
\begin{cases}
\bigl(1.0, \, 0.0, \, 1.0, \, 0.0, \, 2.0 \bigr) & \quad x < 0, \\
\bigl(0.125, \, 0.0, \, 0.1, \, 0.0, \, 0.16 \, \bigr) & \quad x > 0,
\end{cases}
\end{equation}
on $x\in[-1, 1]$ with extrapolation boundary conditions. This is the standard Sod shock tube problem \cite{article:Sod78}, which is ubiquitous in shock-capturing literature, and
is also often found as a standard test for Boltzmann-BGK solvers (e.g., see \cite{article:Ben08}).

We consider three different values of the Knudsen number: 
(a) $\varepsilon = 10^{-2}$, 
(b) $\varepsilon = 10^{-3}$, and (c) $\varepsilon = 10^{-4}$. In each case we run the $\morder=4$ scheme with $M_{\text{elem}} = 200$; we also compare in each case the HyQMOM-BGK solution to the exact solution for the compressible Euler equations \cref{eqn:euler1d} (e.g., 
see Chapter 14 of LeVeque \cite{book:Le02} for a derivation).
We used the following values of ${\mathcal A}_0$ in formula \cref{eqn:limiter_min_max_bounds}: (a) ${\mathcal A}_0 = 50$ 
for $\varepsilon = 10^{-2}$, (b) ${\mathcal A}_0 = 50$
for $\varepsilon = 10^{-3}$, and 
(c) ${\mathcal A}_0 = 350$
for $\varepsilon = 10^{-4}$.

Figure \ref{fig:bgk_shocktube_1} displays the $\varepsilon = 10^{-2}$ numerical simulation at $t=0.28$ showing the
(a) density: $\rho(t,x)$, (b) macroscopic velocity: $u(t,x)$, (c) pressure: $p(t,x)$, and (d) heat flux: $\heat(t,x)$.
At this Knudsen number, the solution is still significantly different than the compressible Euler solution, which is also shown
in each panel. The results are consistent with fully kinetic solutions \cite{article:Ben08}.

Figure \ref{fig:bgk_shocktube_2} displays the $\varepsilon = 10^{-3}$ numerical simulation at $t=0.28$ showing the
(a) density: $\rho(t,x)$, (b) macroscopic velocity: $u(t,x)$, (c) pressure: $p(t,x)$, and (d) heat flux: $\heat(t,x)$.
At this Knudsen number, the solution looks closer to the compressible Euler solution, which is also shown
in each panel. The results are again consistent with fully kinetic solutions \cite{article:Ben08}.

Figure \ref{fig:bgk_shocktube_3} displays the $\varepsilon = 10^{-4}$ numerical simulation at $t=0.28$ showing the
(a) density: $\rho(t,x)$, (b) macroscopic velocity: $u(t,x)$, (c) pressure: $p(t,x)$, and (d) heat flux: $\heat(t,x)$.
At this Knudsen number, the solution is very close to the compressible Euler solution, which is also shown
in each panel. The results are again consistent with fully kinetic solutions \cite{article:Ben08}.

\begin{figure}
\begin{tabular}{cc}
(a)\includegraphics[width=.44\linewidth]{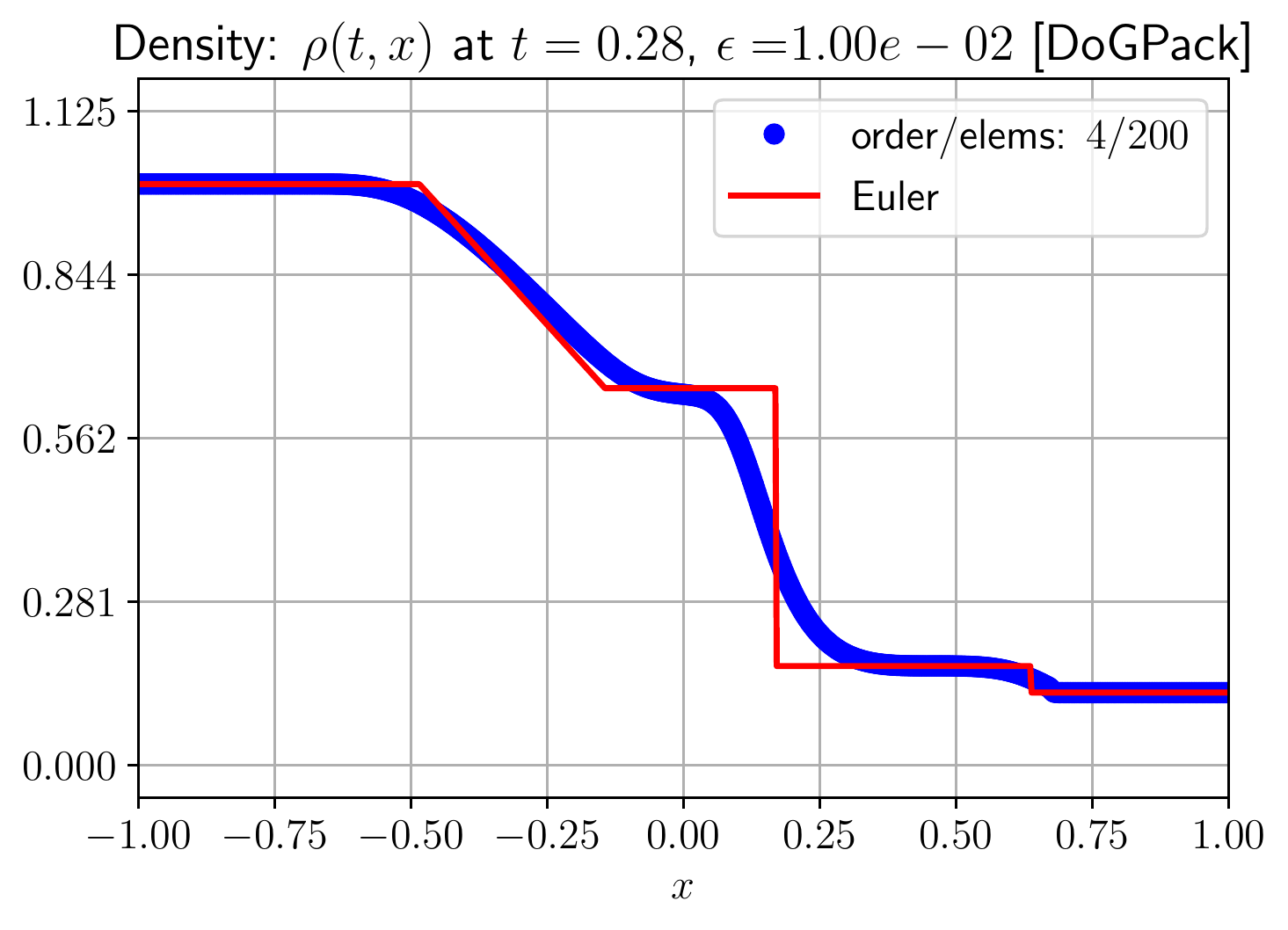} &
(b)\includegraphics[width=.44\linewidth]{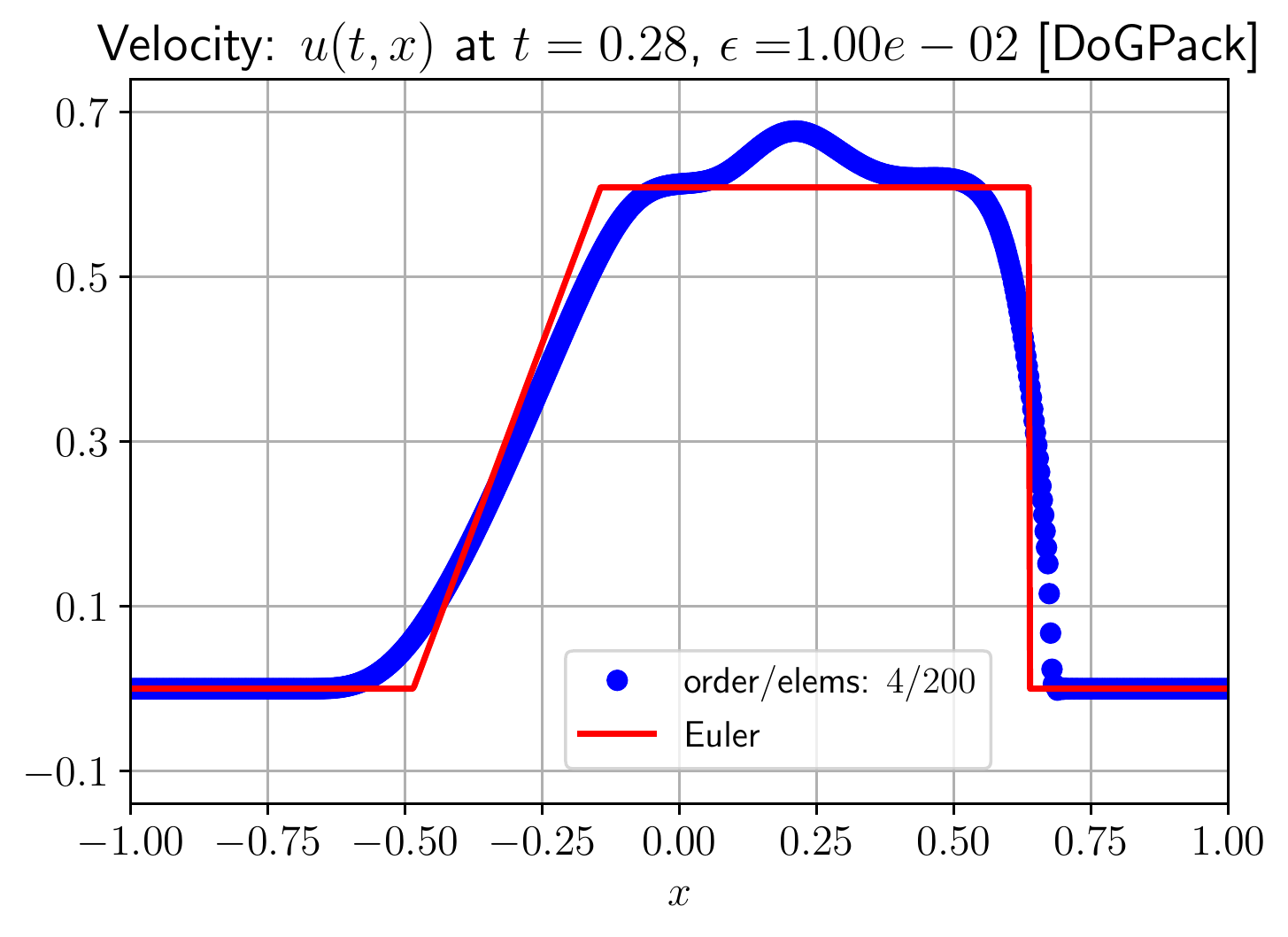} \\
(c)\includegraphics[width=.44\linewidth]{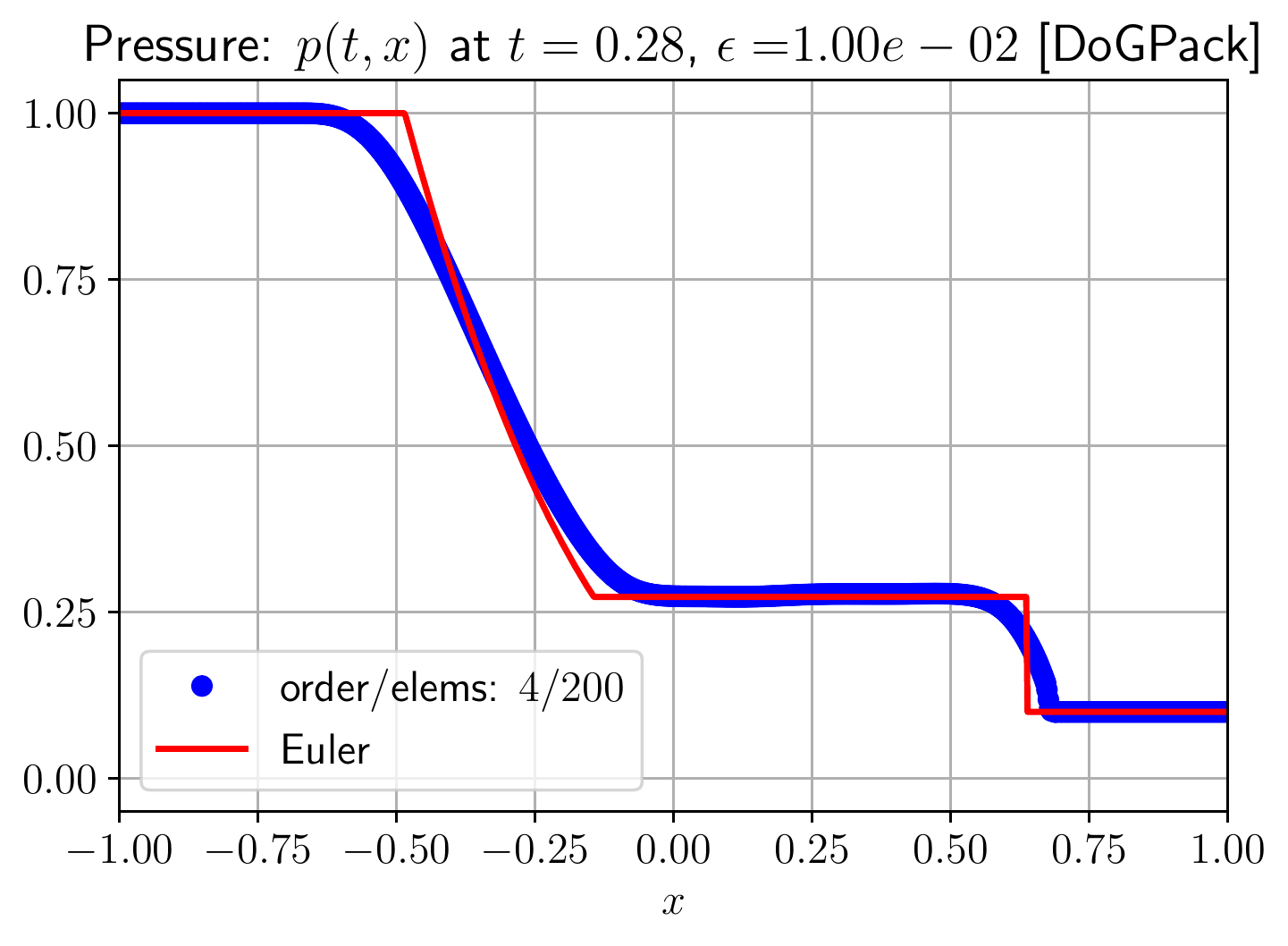} &
(d)\includegraphics[width=.44\linewidth]{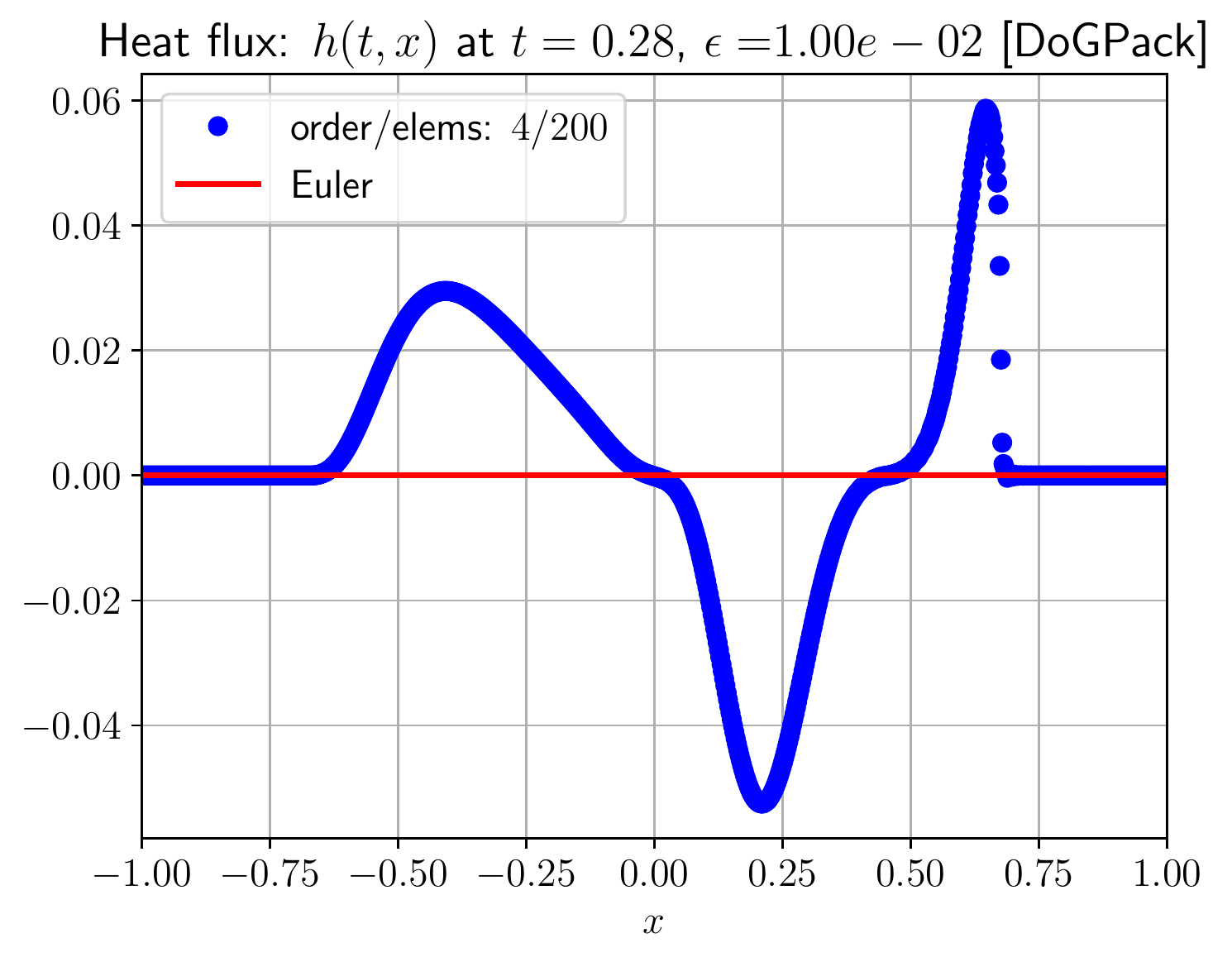}
\end{tabular}
\caption{(\S\ref{sub:shock_tube_bgk}: BGK shock tube problem with $\varepsilon = 10^{-2}$) Numerical solution of shock tube problem
on $x \in [-1, 1]$ with initial conditions given
by \eqref{eqn:bgk_shock_init}. 
Shown are the results from a simulation run with 
the proposed $\morder=4$ scheme
with 200 elements (shown as blue dots) and the exact solution of the compressible Euler equations with the same initial conditions (shown as a solid red line). For the
$\morder=4$ scheme, we are plotting four points-per-element in order
to show the intra-element solution structure. 
The panels show the primitive variables: (a) density: $\rho(t,x)$, (b) macroscopic velocity: $u(t,x)$, (c) pressure: $p(t,x)$, and (d) heat flux: $\heat(t,x)$. \label{fig:bgk_shocktube_1}}
\end{figure}

\begin{figure}
\begin{tabular}{cc}
(a)\includegraphics[width=.44\linewidth]{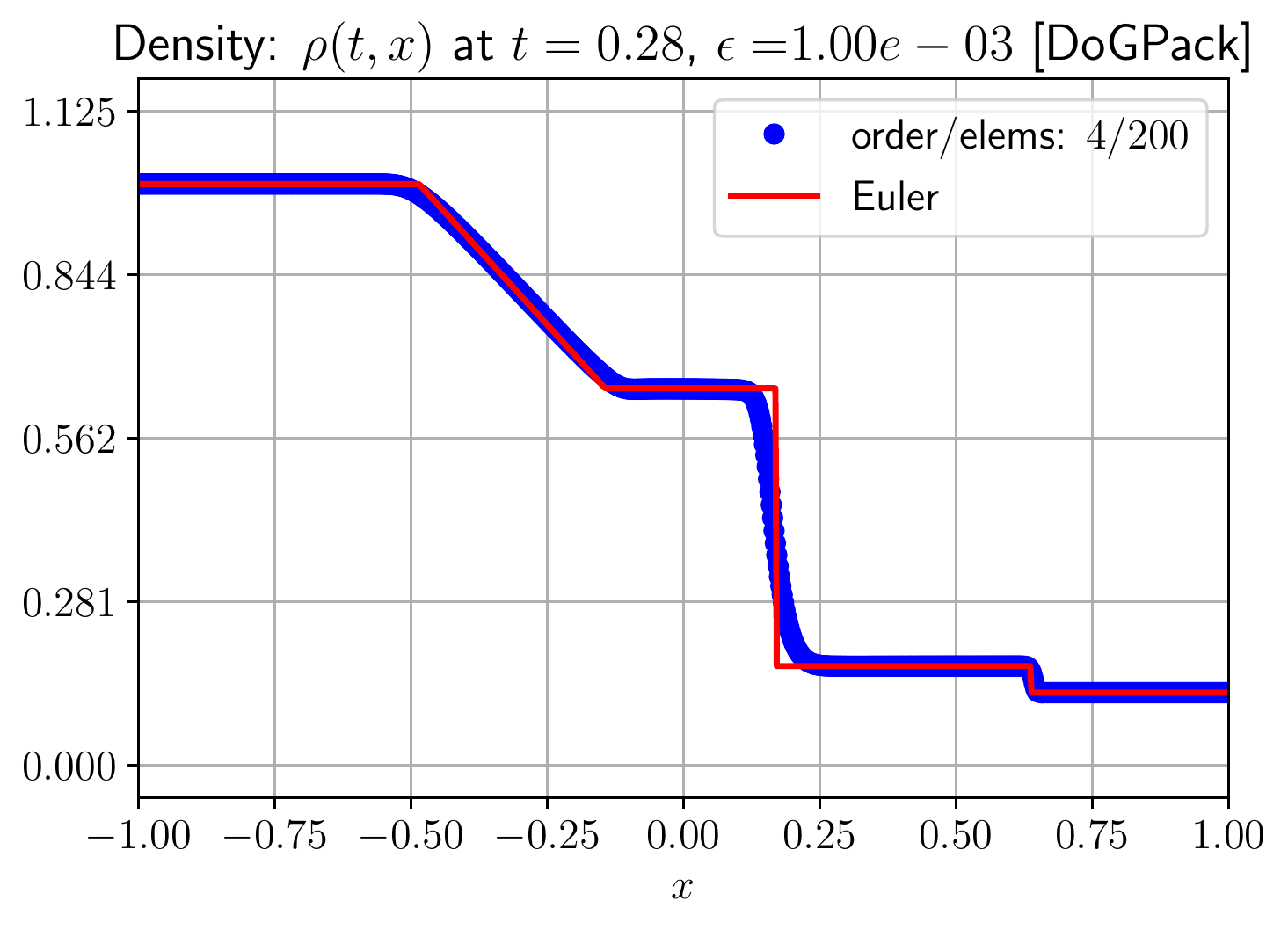} &
(b)\includegraphics[width=.44\linewidth]{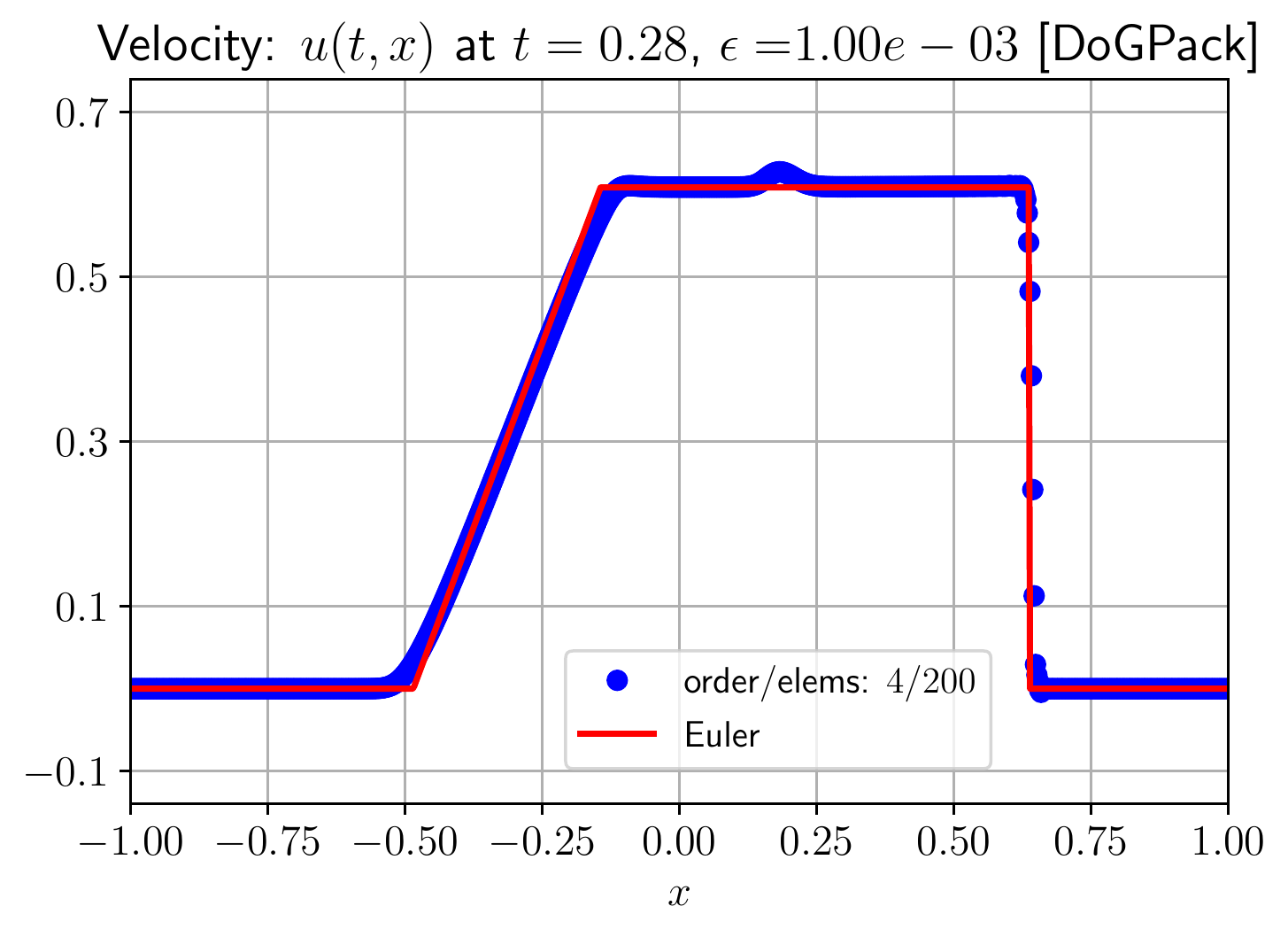} \\
(c)\includegraphics[width=.44\linewidth]{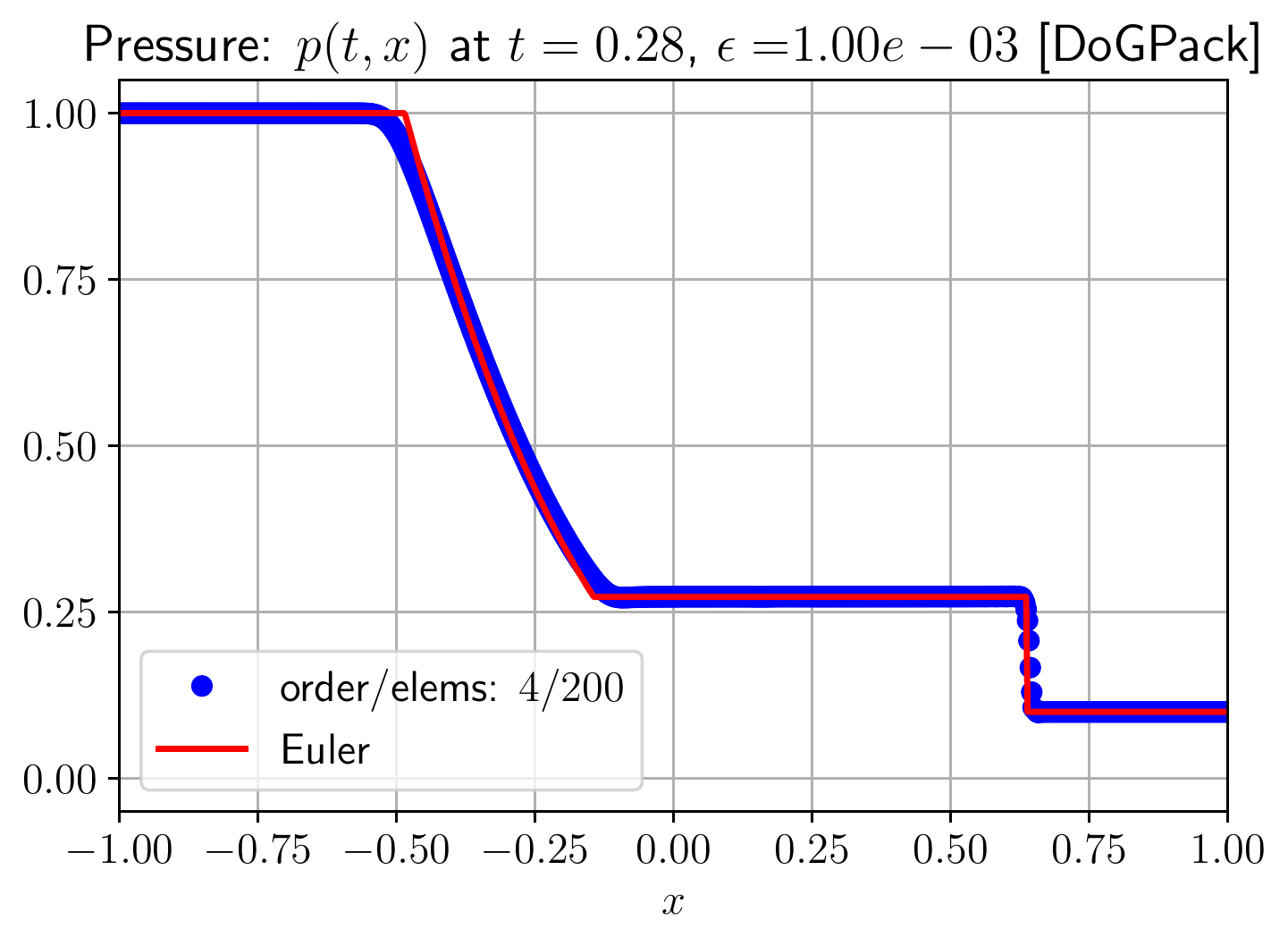} &
(d)\includegraphics[width=.44\linewidth]{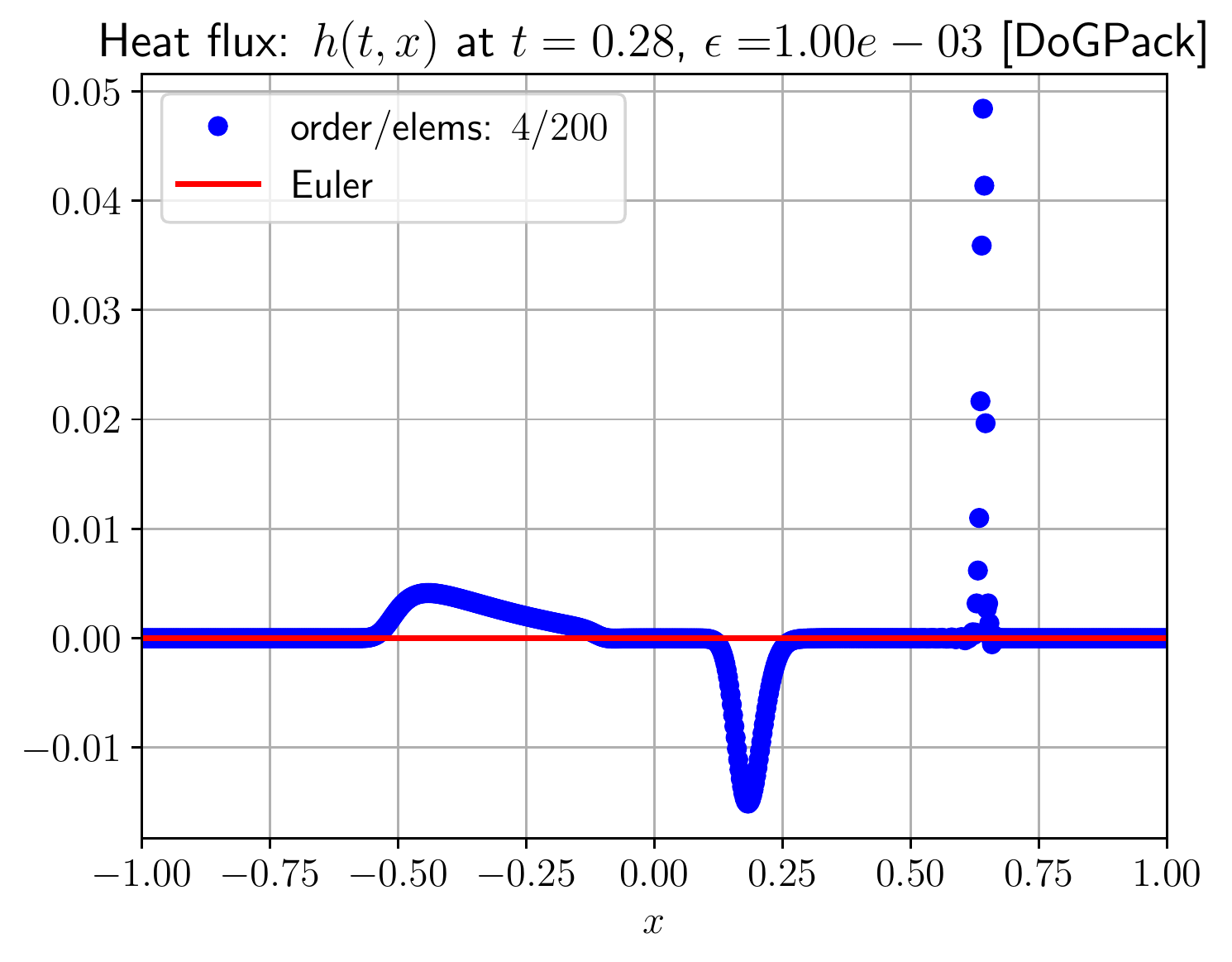}
\end{tabular}
\caption{(\S\ref{sub:shock_tube_bgk}: BGK shock tube problem with $\varepsilon = 10^{-3}$) Numerical solution of shock tube problem
on $x \in [-1, 1]$ with initial conditions given
by \eqref{eqn:bgk_shock_init}. 
Shown are the results from a simulation run with 
the proposed $\morder=4$ scheme
with 200 elements (shown as blue dots) and the exact solution of the compressible Euler equations with the same initial conditions (shown as a solid red line). For the
$\morder=4$ scheme, we are plotting four points per element in order
to show the intra-element solution structure. 
The panels show the primitive variables: (a) density: $\rho(t,x)$, (b) macroscopic velocity: $u(t,x)$, (c) pressure: $p(t,x)$, and (d) heat flux: $\heat(t,x)$. \label{fig:bgk_shocktube_2}}
\end{figure}

\begin{figure}
\begin{tabular}{cc}
(a)\includegraphics[width=.44\linewidth]{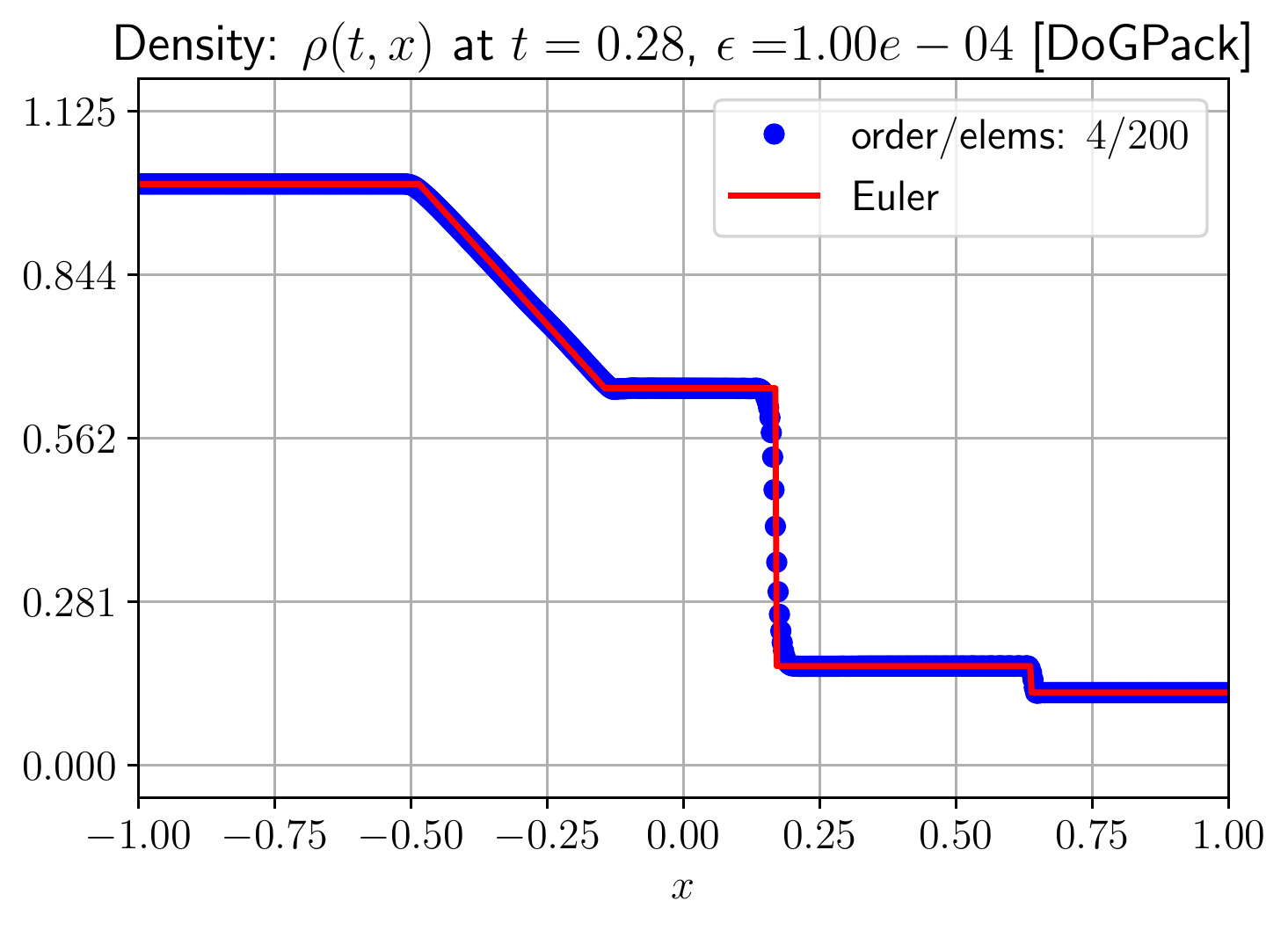} &
(b)\includegraphics[width=.44\linewidth]{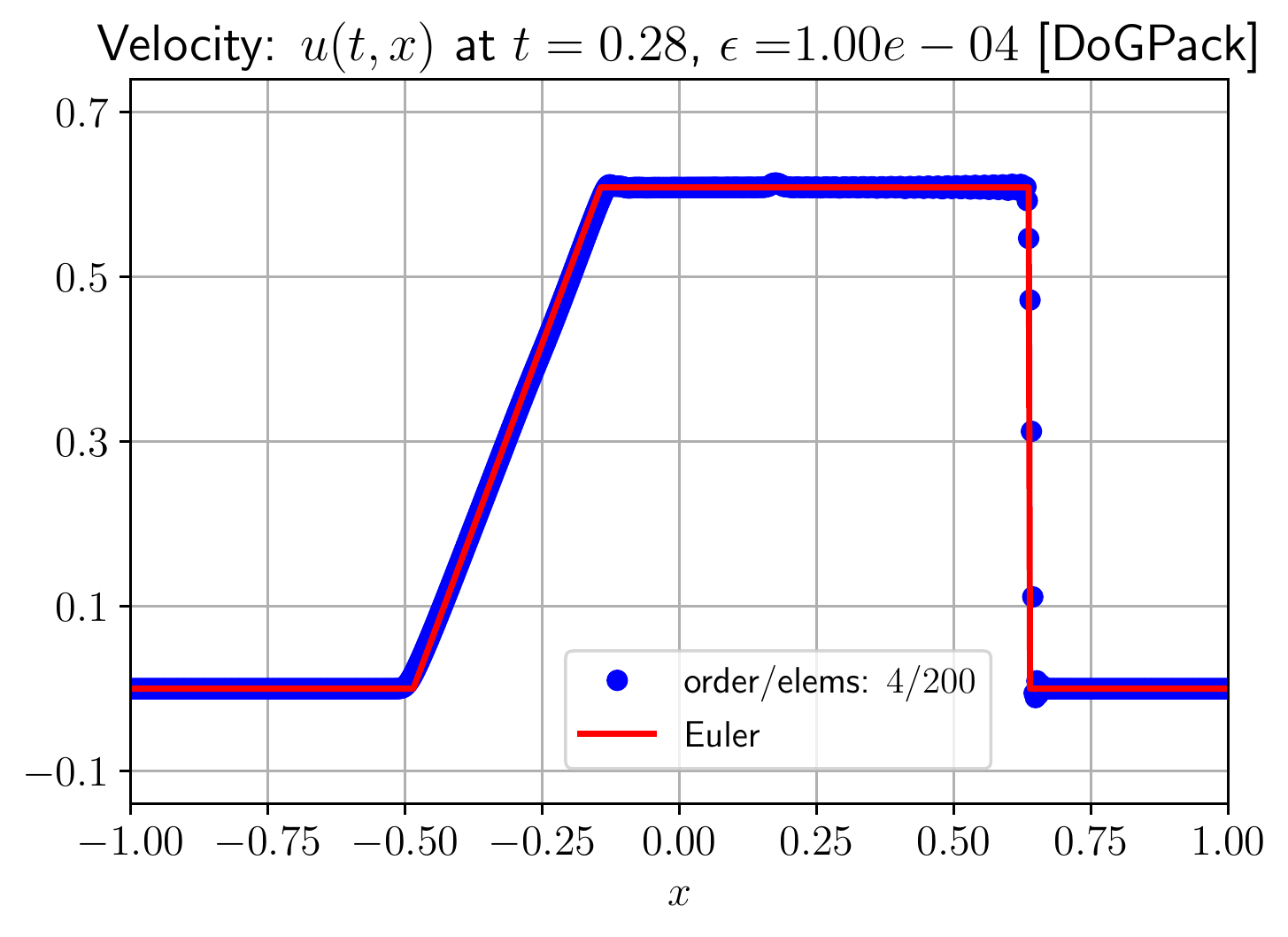} \\
(c)\includegraphics[width=.44\linewidth]{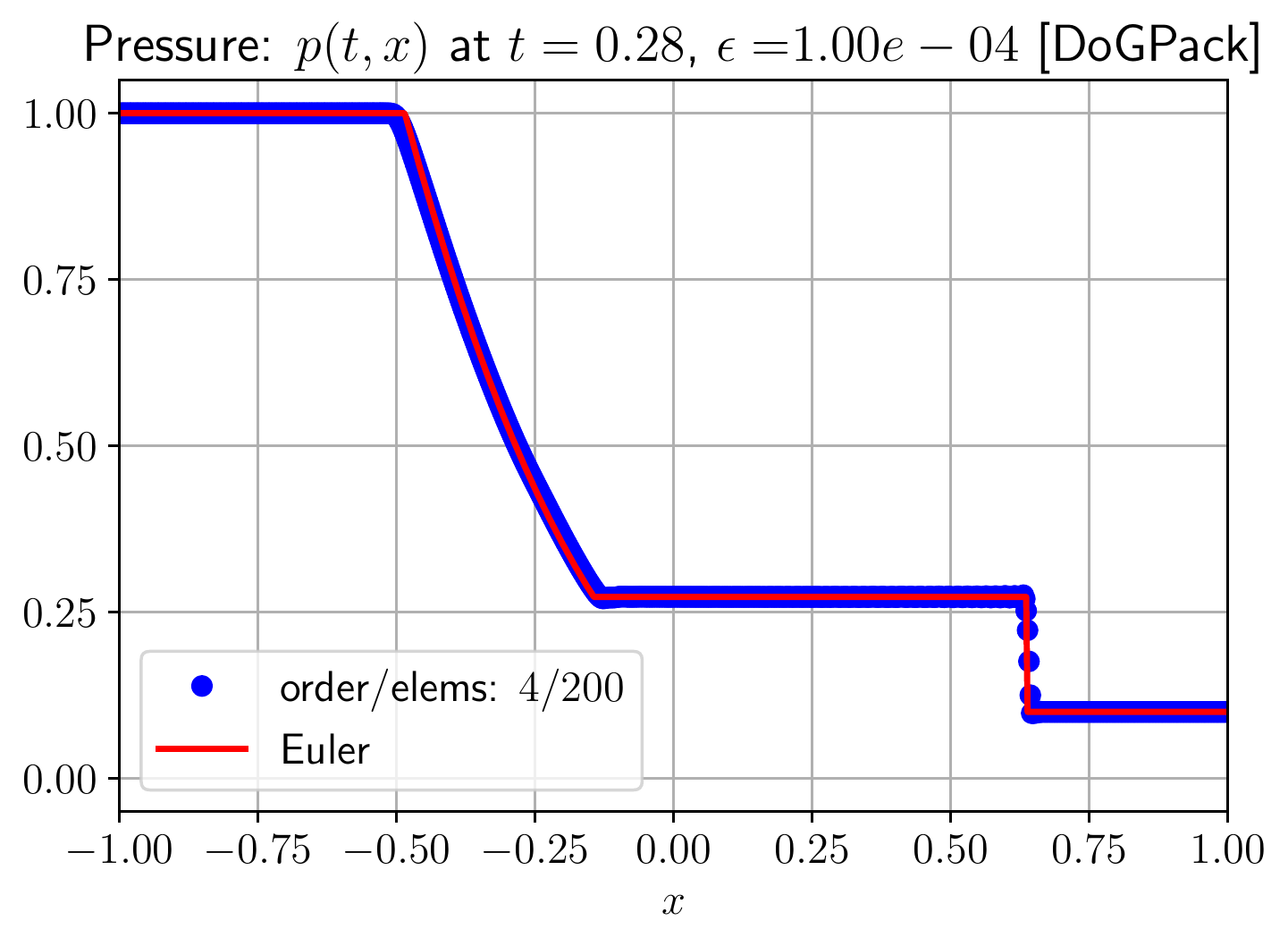} &
(d)\includegraphics[width=.44\linewidth]{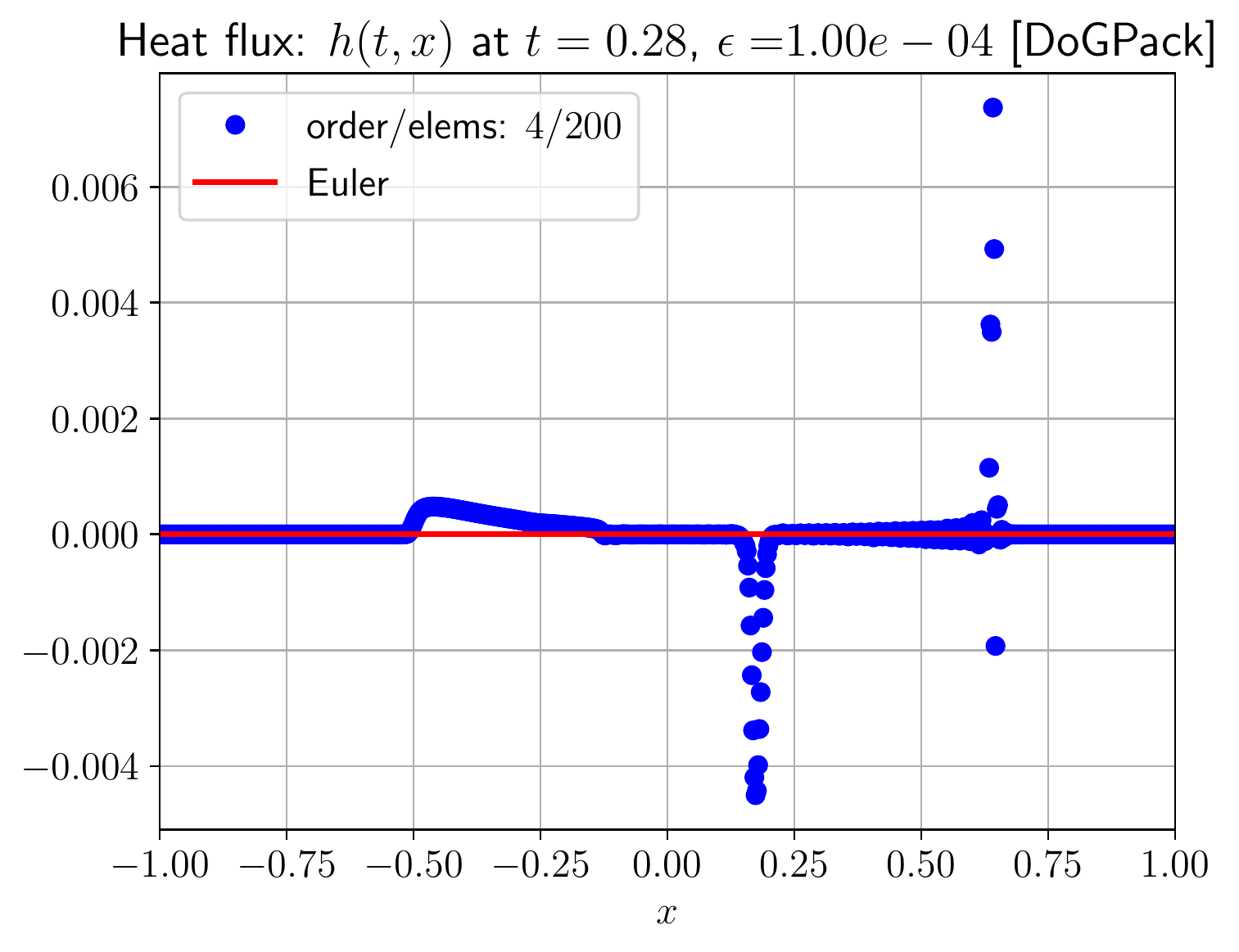}
\end{tabular}
\caption{(\S\ref{sub:shock_tube_bgk}: BGK shock tube problem with $\varepsilon = 10^{-4}$) Numerical solution of shock tube problem
on $x \in [-1, 1]$ with initial conditions given
by \eqref{eqn:bgk_shock_init}. 
Shown are the results from a simulation run with 
the proposed $\morder=4$ scheme
with 200 elements (shown as blue dots) and the exact solution of the compressible Euler equations with the same initial conditions (shown as a solid red line). For the
$\morder=4$ scheme, we are plotting four points-per-element in order
to show the intra-element solution structure. 
The panels show the primitive variables: (a) density: $\rho(t,x)$, (b) macroscopic velocity: $u(t,x)$, (c) pressure: $p(t,x)$, and (d) heat flux: $\heat(t,x)$. \label{fig:bgk_shocktube_3}}
\end{figure}

%% file: conclusions.tex

In this work, we considered a particular moment closure called HyQMOM (the hyperbolic quadrature-based method of moments), which was originally introduced by Fox, Laurent, Vie \cite{article:FoxLauVie2018} and further studied by Johnson \cite{thesis:Johnson2017} and Wiersma \cite{thesis:Wiersma2019}. 
Quadrature-based method of moments (QMOM), including the HyQMOM variant, are a promising class of approximation techniques for reducing kinetic equations to fluid equations that are valid beyond thermodynamic equilibrium.
In particular, the goal of the present work was to develop high-order discontinuous Galerkin schemes and corresponding limiters that control both unphysical oscillations and eliminate positivity violations.

The numerical scheme developed is based on the Lax-Wendroff discontinuous Galerkin scheme
 introduced by Qiu, Dumbser, and Shu \cite{article:Qiu05}, with the predictor-corrector interpretation developed by Gassner et al. \cite{article:GasDumHinMun2011}, and further
 refinements developed by Felton et al. \cite{article:REU2018}. 
 The resulting numerical method is performed in two phases at each time step.

\begin{description}
\item[{\bf Prediction step.}]
The equation and numerical solution are written in the primitive variables in this phase. In the space-time DG approximation, which is applied to each element, integration-by-parts is only performed on the time variable. The result is a system of local nonlinear equations on each element. These equations are solved using a Picard iteration, which provides a sufficiently accurate solution after $\morder-1$ iterations, where $\morder$ is the order of accuracy of the method. 

\medskip 

\item[{\bf Correction step.}]
The correction is a straightforward explicit update based on the time-integral of the evolution equation in conservation form, where the space-time prediction replaces all instances of the exact solution.
\end{description}
Several limiters were applied to the scheme to guarantee positivity and achieve solutions without unphysical oscillations.
\begin{description}
\item [{\bf Limiter I: Prediction step positivity limiter.}]
This limiter is completely local and minimally damps high-order corrections to the primitive variables to get pointwise positivity of the predicted density, pressure, and modified kurtosis on all space-time quadrature points (Gauss-Legendre + edges). The limiter is applied once after each Picard iteration, meaning it is applied a total of $\morder-1$ times per time step.

\medskip

\item [{\bf Limiter II: Correction step positivity limiter on cell average.}]
This limiter is applied once per time step and blends high-order numerical fluxes with positivity preserving low-order fluxes in such a way as to preserve the positivity of the corrected element averages of density, pressure, and modified kurtosis. This limiter is applied once per time step.

\medskip

\item [{\bf Limiter III: Correction step positivity limiter on quadrature points.}] 
This limiter is similar to Limiter I and involves minimally damping the high-order corrections to preserve the positivity of the corrected density, pressure, and modified kurtosis, on all spatial quadrature points (Gauss-Legendre + edges). This limiter is applied once per time step at the end of the step.

\medskip

\item [{\bf Limiter IV: Oscillation Limiter.}]
This limiter damps the solution if the primitive solution variables on the current element significantly exceed the primitive solution variables on neighboring
elements. This limiter is applied once per time step at the end of the time step.  
\end{description}

\medskip

In the collisionless regime, the proposed high-order method and the limiting strategy were tested on both smooth and Riemann problems. The smooth solution was used to perform convergence tests that demonstrated the expected orders of accuracy. The Riemann data tests clearly showed that the limiters were successful in damping unphysical oscillations without adversely diffusing the solution and preserving the positivity of the relevant variables.  

Once the collisionless method is fully developed, we propose a version of the scheme for HyQMOM with a BGK collision operator. We carefully show how to handle the collision operator in both the prediction and collision steps to achieve an asymptotic-preserving (AP) property in the 
high-collision limit. Several numerical examples are provided to validate the scheme both for smooth solutions and Riemann initial data.
The asymptotic-preserving property is validated for smooth solutions and Riemann initial data.

Future work will focus on extending this work to higher dimensions; perhaps using the conditional moment strategy of 
\cite{article:FoxLauVie2018,article:PaDeFo2019}, or some other higher-dimensional extension.

%% file: appendix.tex

\begin{lemma}[Hermite interpolation]
\label{lemma:hermite}
Consider the Hermite interpolation problem of interpolating the function $f(v) = v^{2N}$ with
a polynomial of degree $2N-1$:
\begin{equation}
\label{eqn:hermite_qmom1}
P_{2N-1}(v) = a_0 + a_1 v + a_2 v^2 + \cdots + a_{2N-1} v^{2N-1} = \sum_{{j}=0}^{2N-1} a_{{j}} v^{j},
\end{equation}
with interpolating conditions for ${\ell}=1,2,\ldots,N$:
\begin{equation}
\label{eqn:hermite_qmom2}
P_{2N-1}(\mu_{\ell}) = f(\mu_{\ell}) = \mu_{\ell}^{2N} \qquad \text{and} \qquad
P'_{2N-1}(\mu_{\ell}) = f'(\mu_{\ell})  = 2N\mu_{\ell}^{2N-1},
\end{equation}
where
\begin{equation}
\mu_1 < \mu_2 < \cdots < \mu_{N-1} < \mu_{N}.
\end{equation}
Applying these conditions yields the following formula for the polynomial coefficients:
\begin{equation}
\begin{bmatrix}[1.5]
a_0 \\  \vdots \\ a_{N-1} \\ a_{N} \\ \vdots \\ a_{2N-1}
\end{bmatrix} = 
\begin{bmatrix}[1.5]
1 & \mu_1 & \mu_1^2 & \mu_1^3 & \cdots & \mu_1^{2N-1} \\
\vdots & \vdots & \vdots & & & \vdots \\
1 & \mu_N & \mu_N^2 & \mu_N^3 & \cdots & \mu_N^{2N-1} \\
0 & 1 & 2 \mu_1 & 3 \mu_1^2 & \cdots & (2N-1) \mu_1^{2N-2} \\
\vdots & \vdots & \vdots &  \vdots & & \vdots \\
0 & 1 & 2 \mu_N & 3 \mu_N^2 & \cdots & (2N-1) \mu_N^{2N-2}
\end{bmatrix}^{-1}
\begin{bmatrix}[1.5]
\mu_1^{2N} \\  \vdots \\ \mu_N^{2N} \\ 2N \mu_1^{2N-1} \\  \vdots \\ 2N\mu_N^{2N-1}
\end{bmatrix}.
\end{equation}
\end{lemma}

\begin{proof}
The claimed result follows directly from applying the interpolating conditions to the
polynomial $P_{2N-1}(v)$.
\qedsymbol
\end{proof}

\begin{lemma}[Moment gradient operator I]
\label{lemma:gradient1}
Let $f\left(\vec{\omega},\vec{\mu} \right): \reals^N \times \reals^N \mapsto \reals$ be a continuously differentiable function,
where $\vec{\omega},\vec{\mu} \in \reals^N$ satisfy the moment condition \eqref{eqn:qmom_missing_moment}.
The gradient of $f$ with respect to the moments, $\vec{\mom}= \left(\mom_0, \mom_1, 
\ldots, \mom_{2N-1} \right)$, is given by
\begin{equation}
\label{eqn:moment_gradient1}
\nabla_{\vec{\mom}} f\left(\vec{\omega},\vec{\mu} \right) = 
\mat{B}^{-1} \nabla_{\left(\vec{\omega},\vec{\mu} \right)} f\left(\vec{\omega},\vec{\mu} \right),
\end{equation}
where
\begin{align}
\label{eqn:moment_gradient2}
\nabla_{\vec{\mom}} :=& \left( \frac{\partial}{\partial \mom_0}, \, 
\frac{\partial}{\partial \mom_1}, \ldots, \frac{\partial}{\partial \mom_{2N-1}} \right), \\
\label{eqn:moment_gradient3}
\nabla_{\left(\vec{\omega},\vec{\mu} \right)} :=& 
\left( \frac{\partial}{\partial \omega_1}, \ldots, 
\frac{\partial}{\partial \omega_N}, \frac{\partial}{\partial \mu_1}, \ldots, 
\frac{\partial}{\partial \mu_N} \right), \\
\label{eqn:moment_gradient4}
\mat{B} := \frac{\partial \vec{\mom}}{\partial \left(\vec{\omega}, \vec{\mu} \right)}
 =&  \begin{bmatrix}[1.5]
1 & \mu_1 & \mu_1^2 &  \cdots & \mu_1^{2N-1} \\
\vdots & \vdots & \vdots & & \vdots \\
1 & \mu_N & \mu_N^2 & \cdots & \mu_N^{2N-1} \\
0 & \omega_1 & 2 \omega_1 \mu_1 & \cdots & (2N-1) \omega_1 \mu_1^{2N-2} \\
\vdots & \vdots & \vdots & & \vdots \\
0 & \omega_N & 2 \omega_N \mu_N & \cdots & (2N-1) \omega_N \mu_N^{2N-2}
\end{bmatrix}.
\end{align}
\end{lemma}

\begin{proof}
The results follows directly from the chain rule applied to the moment condition: \eqref{eqn:qmom_missing_moment}.
\qedsymbol
\end{proof}

\begin{lemma}[Moment gradient operator II]
\label{lemma:gradient2}
The moment gradient as defined by \eqref{eqn:moment_gradient1}--\eqref{eqn:moment_gradient4}
applied to the function $f\left(\vec{\omega},\vec{\mu} \right) = \mu_{\ell}$ for some ${\ell}=1,2,\ldots,N$
is the following vector:
\begin{equation}
\label{eqn:lemma_bvec}
\reals^{2N} \ni \vec{b} := \Bigl[ b_0, \, b_1, \, \ldots, b_{2N-1} \Bigr]^T =
\nabla_{\vec{\mom}} \, \mu_{\ell} = \mat{B}^{-1} \,  \nabla_{\left(\vec{\omega},\vec{\mu} \right)} \, \mu_{\ell} = \mat{B}^{-1} \, \vec{e_{N+{\ell}}},
\end{equation}
where $\vec{e_{N+{\ell}}} \in \reals^{2N}$ is a vector with a value of one in component $N+{\ell}$ and a value of zero in all other components.

Furthermore, $\vec{b}$ as defined above can be interpreted as the vector of coefficients of the
following polynomial:
\begin{equation}
\label{eqn:lemma_qpoly1}
Q_{2N-1}(v) := b_0 + b_1 v + b_2 v^2 + \cdots + b_{2N-1} v^{2N-1} 
= \sum_{j=0}^{2N-1} b_{{j}} v^{j},
\end{equation}
which satisfies all of the following conditions:
\begin{equation}
\label{eqn:lemma_qpoly2}
Q_{2N-1}\left(\mu_{m} \right) = 0 \qquad \text{and} \qquad
Q_{2N-1}'\left(\mu_{m} \right) = \frac{1}{\omega_{\ell}} \, \delta^{{\ell}}_{m} \qquad \text{for}
\qquad m=1,2,\ldots,N.
\end{equation}
The polynomial $Q_{2N-1}$ can be explicitly written as follows:
\begin{equation}
\label{eqn:lemma_qpoly3}
  Q_{2N-1}(v) = \frac{\left(v-\mu_{\ell}\right)}{\omega_{\ell}} \left[ {\displaystyle\prod_{\substack{j=1 \\ j\ne {\ell}}}^N \left( v- \mu_j \right)^2}\middle/{\displaystyle\prod_{\substack{j=1 \\ j\ne {\ell}}}^{N}
   \left( \mu_{\ell} - \mu_j \right)^2} \right].
\end{equation}

Finally, the dot product between the vector $\vec{b} \in \reals^{2N}$ defined by
\eqref{eqn:lemma_bvec} and the following vector:
\begin{equation}
\vec\rvec(s) := \Bigl[ 1, \, s, \, s^2, \, \ldots, \, s^{2N-1} \Bigr]^T, \qquad s \in \reals,
\end{equation}
can be written as
\begin{equation}
\label{eqn:lemma_bvec_dotprod}
\vec{b} \cdot \vec\rvec(s) = Q_{2N-1}(s) =
\frac{\left(s-\mu_{\ell}\right)}{\omega_{\ell}} \left[ {\displaystyle\prod_{\substack{j=1 \\ j\ne {\ell}}}^N \left( s- \mu_j \right)^2}\middle/{\displaystyle\prod_{\substack{j=1 \\ j\ne {\ell}}}^{N}
   \left( \mu_{\ell} - \mu_j \right)^2} \right].
\end{equation}
\end{lemma}

\begin{proof}
Equation \eqref{eqn:lemma_bvec} follows directly from definitions
\eqref{eqn:moment_gradient1}--\eqref{eqn:moment_gradient4}.
Polynomial \eqref{eqn:lemma_qpoly1} with Hermite interpolation conditions
\eqref{eqn:lemma_qpoly2} follows from an argument similar to the one provided in 
Lemma \eqref{lemma:hermite}. Equation \eqref{eqn:lemma_qpoly3} follows from invoking the
Lagrange form of the interpolating polynomial that satisfies conditions
\eqref{eqn:lemma_qpoly2}.

Finally, dot product \eqref{eqn:lemma_bvec_dotprod} follows from
the simple observation that
\begin{align*}
& Q_{2N-1}(sv) = \sum_{{j}=0}^{2N-1} b_{{j}} s^{j} v^{j}
= b_0 + \left( b_1 s \right) v 
+ \cdots + \left( b_{2N-1} s^{2N-1} \right) v^{2N-1} \\
\Longrightarrow \quad & 
Q_{2N-1}(s) = \sum_{{j}=0}^{2N-1} b_{{j}} \rvec_{j}(s) = \vec{b} \cdot \vec\rvec(s),
\end{align*}
which when combined with \eqref{eqn:lemma_qpoly3} gives the desired result.
\qedsymbol
\end{proof}

\begin{theorem}[Weak hyperbolicity and linear degeneracy of QMOM]
\label{theorem:weakhyp}
The classical quadrature-based moment (QMOM) closure for a fixed $N \in {\mathbb N}_{\ge 1}$, denoted 
by \cref{eqn:delta}, leads to a system of partial
differential equations that has the following quasilinear form:
\begin{equation}
\label{eqn:qmom_theorem_eqn1}
\vec{q}_{,t} + \mat{A}\left(\vec{q} \right) \, \vec{q}_{,x}, \qquad \text{where} \quad
\vec{q} = \begin{bmatrix} \mom_0, \, \mom_1, \, \cdots,  \, \mom_{2N-2}, \, \mom_{2N-1} \end{bmatrix},
\end{equation}
where the flux Jacobian matrix is given by
\begin{equation}
\label{eqn:qmom_theorem_eqn2}
\mat{A}\left(\vec{q} \right) = \begin{bmatrix}[1.5]
 0&1  &      &        &   \\
 &   0&1     &        &   \\
 &   &\ddots  &\ddots & \\
 &   &      &        0& 1 \\
 \frac{\partial \mom^{\star}_{2N}}{\partial \mom_0} & \frac{\partial \mom^{\star}_{2N}}{\partial \mom_1}  & 
 \cdots & \frac{\partial \mom^{\star}_{2N}}{\partial \mom^{\star}_{2N-2}} & \frac{\partial \mom^{\star}_{2N}}{\partial \mom_{2N-1}}
 \end{bmatrix},
\end{equation}
where $\mom^{\star}_{2N}$ is given by \eqref{eqn:qmom_missing_moment}.
System \eqref{eqn:qmom_theorem_eqn1} and \eqref{eqn:qmom_theorem_eqn2}
 is weakly hyperbolic for any integer $N\ge 1$ with the following properties:
\begin{enumerate}
\item The eigenvalues of \eqref{eqn:qmom_theorem_eqn2} 
are $\lambda_{\ell} = \mu_{\ell}$ for ${\ell}=1,2,3,\ldots,N$, where $\mu_{\ell}$ are the abscissas in \eqref{eqn:delta};
\item Every eigenvalue has algebraic multiplicity exactly two;
\item Every eigenvalue has geometric multiplicity exactly one; and
\item Every wave in the system is linearly degenerate: $\frac{\partial \lambda_{\ell}}{\partial \vec{q}}  \cdot \vec{\rvec^{\ell}} = 0$ for ${\ell}=1,2,3,\ldots,N$, where $\left(\lambda_{\ell}, \vec{\rvec^{\ell}}\right)$ is an eigenvalue-eigenvector pair of flux Jacobian \eqref{eqn:qmom_theorem_eqn2}.
\end{enumerate}
\end{theorem}

\begin{proof}
The key to understanding the eigenvalues of flux Jacobian \eqref{eqn:qmom_theorem_eqn2}
is to understand the last row. To this end, consider:
\begin{equation}
\label{eqn:proof_start}
\mom^{\star}_{2N} = \sum_{{j}=1}^N \omega_{j} \mu_{j}^{2N} \quad \Longrightarrow \quad
\frac{\partial \mom^{\star}_{2N}}{\partial \mom_{\ell}} = \sum_{{j}=1}^N \left[ \mu_{j}^{2N} \frac{\partial \omega_{j}}{\partial \mom_{\ell}} 
+ 2N \omega_{j} \mu_{j}^{2N-1} \frac{\partial \mu_{j}}{\partial \mom_{\ell}} \right].
\end{equation}
To make sense of this we need to obtain expressions for the partial derivatives of the
quadrature weights and abscissas with respect to the moments. To this end, we compute
the related quantities:
\begin{gather}
\label{eqn:proof1}
\mom_{s} = \sum_{{j}=1}^N \omega_{j} \mu_{j}^{s} \quad \Longrightarrow \quad
\frac{\partial \mom_{s}}{\partial \mom_{\ell}} = \sum_{{j}=1}^N \left[ \mu_{j}^{s} \frac{\partial \omega_{j}}{\partial \mom_{\ell}} 
+ s \omega_{j} \mu_{j}^{s-1} \frac{\partial \mu_{j}}{\partial \mom_{\ell}} \right]
= \delta^s_{\ell},
\end{gather}
where $s,\ell = 0,1,\ldots, 2N-1$ and $\delta^s_{\ell}$ is the Kronecker delta, which
arises due to the fact that $\mom_{\ell}$ and $\mom_{s}$ are independent variables if $s\ne \ell$.
The expression in \eqref{eqn:proof1} can be written in matrix form to obtain the
following result:
\begin{gather}
\begin{bmatrix}[1.5]
\frac{\partial \omega_1}{\partial \mom_{0}} &   \cdots &  \frac{\partial \omega_1}{\partial \mom_{2N-1}} \\
\vdots  & & \vdots \\
\frac{\partial \omega_N}{\partial \mom_{0}} &  \cdots &  \frac{\partial \omega_N}{\partial \mom_{2N-1}} \\
\omega_1 \frac{\partial \mu_1}{\partial \mom_{0}} &  \cdots &  \omega_1 \frac{\partial \mu_1}{\partial \mom_{2N-1}} \\
\vdots  & & \vdots \\
\omega_N \frac{\partial \mu_N}{\partial \mom_{0}} & \cdots &  \omega_N \frac{\partial \mu_N}{\partial \mom_{2N-1}}
\end{bmatrix}^T =
\begin{bmatrix}[1.5]
1 & \mu_1 & \mu_1^2 & \cdots & \mu_1^{2N-1} \\
\vdots & \vdots & \vdots & & \vdots \\
1 & \mu_N & \mu_N^2  & \cdots & \mu_N^{2N-1} \\
0 & 1 & 2 \mu_1  & \cdots & (2N-1) \mu_1^{2N-2} \\
\vdots & \vdots & \vdots & & \vdots \\
0 & 1 & 2 \mu_N  & \cdots & (2N-1) \mu_N^{2N-2}
\end{bmatrix}^{-1}.
\end{gather}
Using this result in \eqref{eqn:proof_start} produces expressions for the 
last row of flux Jacobian \eqref{eqn:qmom_theorem_eqn2}:
\begin{gather}
\begin{bmatrix}[1.5]
\frac{\partial \mom^{\star}_{2N}}{\partial \mom_{0}} \\  \frac{\partial \mom^{\star}_{2N}}{\partial \mom_{1}} \\  \frac{\partial \mom^{\star}_{2N}}{\partial \mom_{2}} \\ \vdots \\ \frac{\partial \mom^{\star}_{2N}}{\partial \mom_{2N-2}} \\ \frac{\partial \mom^{\star}_{2N}}{\partial \mom_{2N-1}}
\end{bmatrix} = 
\begin{bmatrix}[1.5]
1 & \mu_1 &  \cdots & \mu_1^{2N-1} \\
\vdots &  \vdots & & \vdots \\
1 & \mu_N & \cdots & \mu_N^{2N-1} \\
0 & 1 & \cdots & (2N-1) \mu_1^{2N-2} \\
\vdots &  \vdots & & \vdots \\
0 & 1 &  \cdots & (2N-1) \mu_N^{2N-2}
\end{bmatrix}^{-1}
\begin{bmatrix}[1.5]
\mu_1^{2N} \\  \vdots \\ \mu_N^{2N} \\ 2N \mu_1^{2N-1} \\  \vdots \\ 2N\mu_N^{2N-1}
\end{bmatrix} = \begin{bmatrix}[1.5]
a_0 \\  \vdots \\ a_{N-1} \\ a_{N} \\ \vdots \\ a_{2N-1}
\end{bmatrix},
\end{gather}
where the last equality follows from Lemma \ref{lemma:hermite} and $a_j$ for $j=0,1,\ldots,2N-1$ are
the coefficients of the Hermite interpolating polynomial defined through
\eqref{eqn:hermite_qmom1} and \eqref{eqn:hermite_qmom2}.

Next we attempt to directly compute the eigenvalues of the flux Jacobian:
\begin{equation}
\left| \mat{A} - v  \mat{\mathbb I} \right| = \begin{vmatrix}[1.5] -v & \quad 1 \\ & \quad \ddots & 
\quad \ddots \\
& &  -v & 1 \\
a_0 & \quad \cdots & \quad a_{2N-2} & \quad \left(a_{2N-1} - v \right)
\end{vmatrix} = v^{2N} - \sum_{{j}=0}^{2N-1} a_{j} v^{{j}}.
\end{equation}
Using a classical result from Hermite polynomial interpolation, we can write
the right-most term in the above expression as follows 
(e.g., see Theorem 6.4 on page 190 of S\"uli and Mayer \cite{book:suli2003}):
\begin{equation}
\left| \mat{A} - v  \mat{\mathbb I} \right| = \left( v - \mu_1 \right)^2 \left( v - \mu_2 \right)^2 \cdots
\left( v - \mu_N \right)^2.
\end{equation}
This proves the first two claims of the theorem: (1) the eigenvalues are the 
quadrature abscissas, and (2) each eigenvalue has algebraic multiplicity exactly two.

Next we look at the eigenvectors. For example, the ${\ell}^{\text{th}}$ eigenvector
for each ${\ell}=1,2,\ldots,N$ satisfies the relationship:
\begin{equation}
\left( \mat{A} - \mu_{\ell} \, \mat{\mathbb I} \right) \vec{\rvec^{\ell}} = \vec{0}.
\end{equation}
By inspection, we see that $\vec{\rvec^{\ell}} \not\equiv \vec{0}$ if and only if 
the first component of $\vec{\rvec^{\ell}}$ is not zero. Without loss of generality the
first component is taken to be unity, and then by inspection we note that
the only eigenvector associated to eigenvalue $v = \mu_{\ell}$ must be
\begin{equation}
\vec{\rvec^{\ell}} = \left( 1, \, \mu_{\ell}, \, \mu_{\ell}^2, \, \ldots, \, \mu_{\ell}^{2N-1} \right)^T.
\end{equation}
This proves the third claim of the theorem: (3) 
each eigenvalue has geometric multiplicity exactly one. Since the geometric multiplicity for
each eigenvalue is strictly less than the algebraic multiplicity, system \eqref{eqn:qmom_theorem_eqn1} and \eqref{eqn:qmom_theorem_eqn2} is weakly hyperbolic for any integer $N\ge 1$.

The final claim in the theorem is that each wave is linearly degenerate. Proving this requires
us to investigate the following dot product
\begin{equation}
\frac{\partial  \mu_{\ell}}{\partial \vec{q}}  \cdot  \vec{\rvec^{\ell}}.
\end{equation}
Invoking Lemmas \eqref{lemma:gradient1} and \eqref{lemma:gradient2} shows that
for every ${\ell}=1,2,\ldots,N$:
\begin{equation}
\frac{\partial  \mu_{\ell}}{\partial \vec{q}}  \cdot  \vec{\rvec^{\ell}} = 
\lim_{s \rightarrow \mu_{\ell}} \left\{ \frac{\left(s-\mu_{\ell}\right)}{\omega_{\ell}} \left[ {\displaystyle\prod_{\substack{j=1 \\ j\ne {\ell}}}^N \left( s- \mu_j \right)^2}\middle/{\displaystyle\prod_{\substack{j=1 \\ j\ne {\ell}}}^{N}
   \left( \mu_{\ell} - \mu_j \right)^2} \right] \right\} = 0,
\end{equation}
which proves that every wave for ${\ell}=1,2,\ldots,N$ is linearly degenerate.
\qedsymbol
\end{proof}

\begin{lemma}[Convexity property]
\label{lemma:convexity}
Let ${\mathcal C}_{\alpha}\left(\vec{Q}\right): \reals^5 \mapsto \reals$ be a convex function. Then
for all $a,b \in \reals$ such that
\begin{equation}
	a \ge 0, \quad b\ge 0, \quad 1-a-b \ge 0,
\end{equation}
the following inequality holds:
\begin{equation}
\label{eqn:appendix_convexity}
{\mathcal C}_{\alpha}\Bigl((1-a-b) \vec{P} + a \vec{Q} + b \vec{R} \Bigr) \le 
(1-a-b) {\mathcal C}_{\alpha}\Bigl( \vec{P} \Bigr) + a {\mathcal C}_{\alpha}\Bigl(\vec{Q} \Bigr)
+ b {\mathcal C}_{\alpha}\Bigl(\vec{R} \Bigr) \quad \forall \vec{P}, \vec{Q}, \vec{R} \in \reals^5.
\end{equation}
\end{lemma}

\begin{proof}
By definition, the function ${\mathcal C}_{\alpha}$ is convex {\it if and only if} the following is true for all $\theta \in [0,1]$:
\begin{equation}
\label{eqn:convexity_defn}
{\mathcal C}_{\alpha}\Bigl( \bigl(1-\theta \bigr) \vec{P} + \theta \vec{Q} \Bigr) \le
\bigl(1-\theta \bigr) \, {\mathcal C}_{\alpha}\Bigl( \vec{P} \Bigr) + \theta \, {\mathcal C}_{\alpha}\Bigl( \vec{Q} \Bigr)
\quad \forall \vec{P}, \vec{Q} \in \reals^5.
\end{equation}
Consider the convex function applied to a sum of three vectors of the following form:
\begin{equation}
{\mathcal C}_{\alpha}\Bigl((1-a-b) \vec{P} + a \vec{Q} + b \vec{R} \Bigr) \quad \text{where} \quad a,b,(1-a-b) \ge 0.
\end{equation}
We can temporarily define the following vector:
\begin{equation}
	\vec{S} := \biggl(\frac{a}{a+b} \biggr) \, \vec{Q} + \biggl(\frac{b}{a+b} \biggr) \, \vec{R} \quad \Longrightarrow \quad 
	  a \vec{Q} + b \vec{R} = (a+b) \vec{S},
\end{equation}
such that
\begin{equation}
{\mathcal C}_{\alpha}\Bigl((1-a-b) \vec{P} + a \vec{Q} + b \vec{R} \Bigr) = 
  {\mathcal C}_{\alpha}\Bigl((1-a-b) \vec{P} + (a+b) \vec{S} \Bigr).
\end{equation}
Invoking the convexity definition \eqref{eqn:convexity_defn} with $\theta = a+b$, which by assumption satisfies
$\theta \in [0,1]$, we have that
\begin{equation}
\label{eqn:appendix_step1}
{\mathcal C}_{\alpha}\Bigl((1-a-b) \vec{P} + a \vec{Q} + b \vec{R} \Bigr) \le 
(1-a-b) {\mathcal C}_{\alpha}\Bigl( \vec{P} \Bigr) + (a+b) {\mathcal C}_{\alpha}\Biggl( \biggl(\frac{a}{a+b}\biggr) \vec{Q} 
+ \biggl( \frac{b}{a+b} \biggr) \vec{R} \Biggr).
\end{equation}
We then again invoke definition \eqref{eqn:convexity_defn}, this time with $\theta= \frac{b}{a+b}$, which also satisfies
$\theta \in [0,1]$, to get that
\begin{equation}
\label{eqn:appendix_step2}
{\mathcal C}_{\alpha}\Biggl( \biggl(\frac{a}{a+b}\biggr) \vec{Q} 
+ \biggl( \frac{b}{a+b} \biggr) \vec{R} \Biggr) \le \biggl(\frac{a}{a+b}\biggr) {\mathcal C}_{\alpha}\Bigl(\vec{Q} \Bigr)
+ \biggl( \frac{b}{a+b} \biggr) {\mathcal C}_{\alpha}\Bigl(\vec{R} \Bigr).
\end{equation}
Combining the last two inequalities, \eqref{eqn:appendix_step1} and \eqref{eqn:appendix_step2}, results in desired inequality: \eqref{eqn:appendix_convexity}.
\qedsymbol
\end{proof}